\documentclass{elsarticle}

\usepackage{lineno,hyperref}

\journal{arXiv}

\usepackage{amsmath}
\usepackage{amsfonts}
\usepackage{amsthm}
\usepackage{enumitem}
\usepackage{xcolor}

\usepackage{upgreek}
\usepackage{float}
\usepackage{subfigure}
\usepackage{tikz}
\usepackage{tikz-3dplot}
\usetikzlibrary{hobby}
\usetikzlibrary{positioning}
\usetikzlibrary{shapes}

\newtheorem{theorem}{Theorem}[section]

\newtheorem{remark}{Remark}

\newtheorem{lemma}{Lemma}[section]

\begin{document}

\begin{frontmatter}

\title{Modeling calcium dynamics in neurons with endoplasmic reticulum: existence, uniqueness and an implicit-explicit finite element scheme   }
\tnotetext[mytitlenote]{This work was funded by National Institute of Mental Health (NIMH), Grant number: R01MH118930.}

\author{Qingguang Guan\corref{mycorrespondingauthor}}
\address{Department of Mathematics, Temple University, Philadelphia, PA 19122, USA}
\ead{qingguang.guan@temple.edu}
\author{Gillian Queisser}
\address{Department of Mathematics, Temple University, Philadelphia, PA 19122, USA}
\ead{gillian.queisser@temple.edu}
\cortext[mycorrespondingauthor]{Corresponding author}

\begin{abstract}
	Like many other biological processes, calcium dynamics in neurons containing an endoplasmic reticulum is governed by diffusion-reaction equations on interface-separated domains. Interface conditions are typically described by systems of ordinary differential equations that provide fluxes across the interfaces. Using the calcium model as an example of this class of ODE-flux boundary interface problems, we prove the existence, uniqueness and boundedness of the solution by applying comparison theorem, fundamental solution of the parabolic operator and a strategy used in Picard's existence theorem. Then we propose and analyze an efficient implicit-explicit finite element scheme which is implicit for the parabolic operator and explicit for the nonlinear terms. We show that the stability does not depend on the spatial mesh size. Also the optimal convergence rate in $H^1$ norm is obtained. Numerical experiments illustrate the theoretical results.
\end{abstract}

\end{frontmatter}


\section{Introduction}\label{intro}

In a variety of applications, particularly in biology, spatio-temporal dynamics can be described by diffusion-reaction systems. Intuitively, and from an energy consumption perspective, such systems appear optimal but usually have the drawback of producing slow, short-range, and potentially inefficient communication pathways. In order to overcome these drawbacks, active and energy-consuming processes are introduced by biology on two-dimensional manifolds, i.e. interfaces which separate multiple domains. Such processes, in the biological context, are channels, pumps, and receptors capable of exchanging specific ions across the interfaces and are mathematically described by systems of ordinary differential equations (ODEs), nonlinearly coupled to the domain equations, see  \cite{Amar2005,Breit2018,MacKrill2012,HODGKIN1952,Luo1991,Matano2011,Sneyd2003,Veneroni2006}. Given the ubiquitous nature of this modeling approach, we study a mathematical model consisting of diffusion-reaction equations in $\Omega_c$ and $\Omega_e$ (see Fig. \ref{fig1} for a 2D case), coupled by nonlinear dynamic boundary conditions involving an ODE system on the interface $\Upsilon$. 
To provide context to this ODE-coupled system of partial differential equations (PDEs), we chose cellular calcium dynamics as a leading example. Calcium dynamics are among the most important regulators in neurons and their three-dimensional (3D) spatio-temporal dynamics have been shown to play a critical role in cellular function, learning, and many neurodegenerative diseases, see \cite{Clapham1995,MacKrill2012, Ozturk2020,Raffaello2016,Rosales2004}. Calcium dynamics is further controlled by the 3D organization of cells,  see \cite{Ali2019,Ozturk2020,Raffaello2016,Wu2017}, thus making calcium in neurons a prime candidate for the problems studied in this paper. 
\begin{figure}[H]\label{fig1}
	\begin{center}
		\begin{tikzpicture}[scale=0.8] 
			\path[use as bounding box] (-0.18,-0.8) rectangle (4.65,3.2);
			\draw[ yshift=-0.2cm](0,0.5) to [closed, curve through = {(1,-.5 )..(2,-0.6)..(3,-.5)} ] (4,0);
			\filldraw[opacity=0.3,gray, yshift=2.5cm](1,-0.7) 
			to [closed, curve through = {(1.5,-0.4 )..(2,-0.4)..(2.5,-0.4)} ] (3,-0.7);
			\draw[black, yshift=2.5cm](1,-0.7) 
			to [closed, curve through = {(1.5,-0.4 )..(2,-0.4)..(2.5,-0.4)} ] (3,-0.7);
			\node [above right] at (3.5,0.7) {$\Omega_c$};
			\node [above right] at (4.5,0.7) {$\partial \Omega$};
			\node [above] at (2.3,1) {$\Omega_e$};
			\node [ ] at (1.03,1.2) {${\Upsilon}$};
		\end{tikzpicture}
	\end{center}
	\caption{Domain $\Omega = \Omega_c\cup {\Upsilon}\cup\Omega_e$.}\label{fg2}
\end{figure}
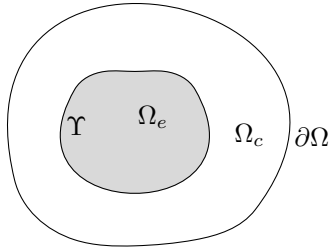

The geometry structures of neurons can be found in   \cite{Ozturk2020,Wu2017}. Neuron has a tubular organelle - endoplasmic reticulum (i.e. ER), which can be thought of as a cell-within-a-cell. For simplicity, without loss of generality, we assume the domain of the whole cell in 3D, including ER, is simply connected without holes, with smooth boundary; so does the domain of ER. Also there are no other organelles in this cell, boundaries (i.e. membranes) of ER and the cell have no contact. A cross section of an axon of the simplified neuron cell can be shown by Figure \ref{fig1}, where 
$\Omega_e$ is the region of ER lumen; ${\Upsilon}$ is the ER membrane; $\Omega_c$ is the region of cytosol; $\partial \Omega$ is the plasma membrane.

Here, let $\Omega$ be a bounded domain with $C^{1,\beta}$ boundary in $\mathbb{R}^n$, $n=2,3$, ${\Upsilon}$ be the $C^{1,\beta}$ interface, $0<\beta<1$. We define 
\begin{itemize}
	\item[] $u:\ \overline{\Omega}_c\times[0,T] \rightarrow \mathbb{R}$ as the Ca$^{2+}$ concentration in cytosol,
	\item[] $b:\  \overline{\Omega}_c\times[0,T] \rightarrow \mathbb{R}$ as the concentration of a buffer interacting with $u$,
	\item[] $u_e: \overline{\Omega}_e\times[0,T] \rightarrow \mathbb{R}$ as the Ca$^{2+}$ concentration in ER,
\end{itemize}
where $[0,T]$ is the time domain. We define the PDE model by \eqref{uomega1} to \eqref{bBdyonOmega2}
\begin{eqnarray}
	&&{\partial_t u} \ -\nabla\cdot ({D_c} \nabla u) = 
	f(b,u) 
	\quad \text{on\ } \Omega_c\times(0,T]
	\label{uomega1}\\
	&&{\partial_t b}\ -\nabla\cdot ({D_b} \nabla b) =
	f(b,u) 
	\quad \text{on\ } \Omega_c\times(0,T]
	\label{bomega1}\\
	&&{\partial_t u_e} -\nabla\cdot ({D_e} \nabla u_e) = 
	0 \quad \text{on\ } \Omega_e\times(0,T]
	\label{ueomega2}
\end{eqnarray}
${\partial_t }$ is the partial derivative in time, $f(b,u) = K_b^-(b^0-b)-K_b^+ bu$ is the reaction term {\color{black} which models storing and releasing Ca$^{2+}$ in cytosol, see \cite{Breit2018} for more details,} and $D_c, D_b, D_e,$ $ K_b^{\pm}, b^0$ are positive constants. 
Denoting the outer normal derivative by $\partial_{n}$,  direction of the unit vector $n$ depends on which domain the function lies, the boundary conditions describing fluxes across the interfaces are
\begin{eqnarray}
	&&D_c\partial_{n} u \  = J_{l,p}-J_N-J_P \quad {\rm on\ } \partial\Omega\times(0,T]
	\label{uBdyonOmega1}\\
	&&D_c\partial_{n} u \  = J_R+J_{l,e}-J_S \quad {\ \rm on\  } \Upsilon\times(0,T]
	\label{ueBdyonGamma2}\\
	&&D_e\partial_{n} u_e = J_S-J_R-J_{l,e}\quad {\ \rm on\ } \Upsilon\times(0,T]
	\label{ueBdyonGamma1}\\
	&&D_b\partial_{n} b\ \  =\ 0\quad {\rm on\ } \partial\Omega\cup \Upsilon\times(0,T]
	\label{bBdyonOmega1}
\end{eqnarray}
with initial data
\begin{eqnarray}
	&&u(0,x) = u_0(x),\ b(0,x) = b_0(x)  
	\ {\rm on\ } \Omega_c; \ u_e(0,x) = u_{e,0}(x)\ {\rm on\ } \Omega_e\\
	&&c_1(0) = c_{1}(0,u_0(x)),\ o(0) =o(0,u_0(x)),\  c_2(0) = c_{2}(0,u_0(x))
	\ {\rm on\ } \Upsilon.
	\label{bBdyonOmega2}
\end{eqnarray}
The choice of the fluxes $J_N,J_P,J_{l,p},J_R,J_S,J_{l,e}$ in \eqref{uBdyonOmega1}-\eqref{ueBdyonGamma1} is motivated by the calcium models previously studied in \cite{Atri1993,Breit2018a,Breit2018,Keizer1996,Means2006, Sneyd2003}. They describe the pumps that exchange Ca$^{2+}$ across $\partial \Omega$ and $\Upsilon$. {\color{black}
Ca$^{2+}$ flux crossing plasma membrane ($\partial\Omega$) is governed by PMCA pumps ($J_P$), NCX pumps ($J_N$) and leak channels ($J_{l,p}$); Ca$^{2+}$ flux crossing ER membrane ($\Upsilon$) is controlled by RyR channels ($J_R$), SERCA pumps ($J_S$) and leak channels ($J_{l,e}$), see \cite{Clapham1995, Ozturk2020} for introduction, and \cite{Breit2018} for an illustration.
Below we list the fluxes under consideration. On $\Upsilon$, the flux through RyR channels (or Ryanodine receptors) is}
\begin{equation}\label{J_R}
	J_{R} = C_1^e P(t,u)(u_e-u),
\end{equation}
where $C_1^e$ is a positive constant and for $u\geq 0$, $P(t,u)\in [0,1]$ is the probability that RyR channel is open:
\begin{equation}\label{ptu}
	P(t,u) = 1-c_1(t)-c_2(t),
\end{equation}
and $c_1, c_2$ come from the ODE system
\begin{equation}\label{matrixf1}
	\begin{bmatrix}
		c_1'\\
		o'\\
		c_2'
	\end{bmatrix}
	=
	\begin{bmatrix}
		-u^4k_a^+-k_a^- & -k_a^- & -k_a^- \\
		-u^3k_b^+ & -u^3k_b^+-k_b^- & -u^3k_b^+ \\
		-k_c^+    & -k_c^+       & -k_c^+-k_c^- 
	\end{bmatrix}
	\begin{bmatrix}
		c_1\\
		o\\
		c_2
	\end{bmatrix}
	+
	\begin{bmatrix}
		k_a^-\\
		u^3k_b^+\\
		k_c^+
	\end{bmatrix},
\end{equation}
where $k_a^{\pm},k_b^{\pm},k_c^{\pm}$ are positive constants, and obviously 
$\{c_1,o,c_2\}$ depend on position due to $u$, but for simplicity, the spatial variable is hidden, the initial values of $\{c_1,o,c_2\}$ are non-negative and $c_1(0)+o(0)+c_2(0)\leq 1$,
see \cite{Breit2018} and \cite{Keizer1996} for details. Next, we define the flux through SERCA pumps
\begin{equation}\label{J_S}
	J_{S} = C_2^e\frac{u}{(K_s+u)\phi_{m}(u_e)},\quad J_{l,e} = C_3^e(u_e-u),
\end{equation}
where $K_s, C_2^e, C_3^e$ are positive numbers, $\phi_{m}(\cdot)$ is defined as
$$
\phi_{m}(x) = 
\left\{
\begin{aligned} 
	&{m}/{2}, & {\rm \ if\ } &x\leq 0\\ 
	&{m^6}/{(2 m^5-5 m^2 x^3+6 m x^4-2 x^5)}, & {\rm \ if\ } &0<x<m\\ 
	&x, & {\rm \ if\ } & x \geq m
\end{aligned}
\right.
$$
where $m>0$ can be any small number.
On $\partial \Omega$ we define
\begin{equation}\label{J_N}
	J_P=C_1^c\frac{u^2}{K_p^2+u^2},\quad J_N=C_2^c\frac{u}{K_n+ u},\quad  J_{l,p}=C_3^c(c_o-u)
\end{equation}
where $C_1^c, C_2^c, C_3^c, K_p, K_n$ are positive constants, $c_o$ is the extracellular Ca$^{2+}$ concentration, we assume it is a positive  constant, however $c_o$ can also be a positive bounded function. $J_N$ and $J_P$ are commonly used first and second order Hill equations, see \cite{Breit2018}.

An analysis of linear parabolic equations in two adjoining domains, coupled with nonlinear but nondynamic boundary conditions, can be found in Calabr\`o 
\cite{Calabro2006}. 
An implicit DG method in \cite{Cangiani2013,Cangiani2016} was developed to solve such problems that originated from mass transfer through semipermeable membranes.
Poisson equations coupled with nonlinear dynamic boundary conditions in \cite{Amar2005,Matano2011,Veneroni2006}, arising in electrodynamics on membranes are analyzed. A numerical method in \cite{Henriquez2017} based on a boundary integral formulation, together with an implicit-explicit scheme, is proposed to solve these electrical activity problems, also implicit-explicit partitioned time stepping for a parabolic two domain problem is considered in Connors\cite{connors2009}. Both motivate us to design our scheme. However the analysis and numerical methods in \cite{Amar2005,Calabro2006,Cangiani2013,Cangiani2016,Henriquez2017,Matano2011,Veneroni2006} cannot be applied to our model directly, since the reaction term $f(b,u)$, and the interface conditions, especially the non monotonic probability function $P(t,u)$ involving ODE systems,  are not globally Lipschitz continuous. Previous work on Ca$^{2+}$ models focused on simulations employing numerical methods
such as the Finite Volume Method in \cite{Breit2018a,Breit2018}, the Finite Element Method in \cite{Gil2021,Means2006}, but most of them are fully implicit and less efficient.
Despite the importance of calcium models, to our knowledge, a thorough study for the wellposedness, fast numerical methods and convergence analysis have not been done. 
The goal of this paper is to prove existence and uniqueness, to retrieve bounds for solution of the model \eqref{uomega1}-\eqref{bBdyonOmega2}, then to obtain error estimates for an implicit-explicit FEM scheme which is more efficient and easy to parallelize, since equations \eqref{uomega1}-\eqref{ueomega2} are fully decoupled for each time step. Unlike the methods used in \cite{Amar2005,Calabro2006,Matano2011,Veneroni2006}, our proof of wellposedness is inspired by Picard's existence theorem, based on the fundamental solution of parabolic operator. The FEM scheme is an Euler method implicit for the parabolic operator and explicit for the nonlinear parts. A key part of the error analysis is to construct the elliptic projection, similar to the work in Cangiani\cite{Cangiani2013} and Douglas\cite{Douglas1973}, but incorporating the dynamic ODE system. So that the wellposedness of the associated nonlinear problem needs to be addressed. 

The paper is arranged as follows: Necessary lemmas are provided in Section \ref{Lemmas}. The existence, uniqueness and boundedness of the solution for the model equations \eqref{uomega1}-\eqref{bBdyonOmega2} are proved in Section \ref{E&U&B}.
Section \ref{errorGP} is devoted to the $H^1$ error estimate for the Galerkin projection which is based on a nonlinear problem in variational form. In Section \ref{SD}, the error analysis of the semi-discrete Galerkin method is carried out. The convergence rate in $H^1$ norm for a fully discrete implicit-explicit FEM scheme is obtained in Section \ref{FD}. Numerical tests in Section \ref{Num} validate the theoretical results. Conclusions are drawn in Section \ref{con}.


\section{Lemmas}\label{Lemmas}
In this section, we introduce some notations and lemmas. We define $Q_c = (0,T]\times {\Omega}_c$, $Q_e = (0,T]\times {\Omega}_e$. In the following $M$ is a positive constant and function $u$ being bounded by $M$ means $|u|\leq M$.

We then show that the open probability function $P(t,u)$ (see eq.~(\ref{ptu})) is continuous with respect to $u$ if $u\in C(\Upsilon\times[0,T])$ and $0\leq u\leq M$. 
\begin{lemma}\label{pu1u2}
	Let $u_1$, $u_2$ $\in C(\Upsilon\times[0,T])$ be non-negative functions bounded by $M$ and  $P(t,u)$ be defined as \eqref{ptu}-\eqref{matrixf1}. We then have
	\begin{equation}
		|P(t,u_1)-P(t,u_2)|\leq 
		K\left(|u_1(0)-u_2(0)|+\int_0^t|u_1(s)-u_2(s)|{\rm d}s\right),
	\end{equation}
	where $K$ is a positive constant that depends on $M$, but doesn't depend on $u_1, u_2$.
\end{lemma}
\begin{proof}
	Let $\vec{q}(t) = [c_1,o,c_2]^T$, $\vec{f}(u) = [k_a^+,u^3k_b^-,k_c^-]^T$, and 
	${\bf{A}}(u)$ be the coefficient matrix in \eqref{matrixf1}. 
	The ODE system can then be written as 
	\begin{equation}
		\frac{{\rm d}\vec{q}}{{\rm d}t} = {\bf{A}}(u)\vec{q}(t)+\vec{f}(u).
	\end{equation}
	For given $u_1, u_2$, the corresponding integral equations are
	\begin{eqnarray*}
		\vec{q}_i(t) 
		= 
		\int_0^t{\bf{A}}(u_i)\vec{q}_i(s){\rm d}s+\int_0^t\vec{f}(u_i){\rm d}s+\vec{q}_i(0,u_i(0)),\quad i=1,2 .
	\end{eqnarray*}
	Let $\vec{e}(t) = \vec{q}_1(t)-\vec{q}_2(t),$ $\vec{f_e} = \vec{f}_1(u_1)-\vec{f}_2(u_2),$ ${\bf{A}_e}={\bf{A}}(u_1)-{\bf{A}}(u_2)$ and 
	$\vec{q}_{0,e} = \vec{q}_1(0,u_1(0))-\vec{q}_2(0,u_2(0))$, we have
	\begin{align*}
		\vec{e}(t) 
		&= 
		\int_0^t{\bf{A}}(u_1)\vec{e}(s){\rm d}s+\int_0^t\vec{f_e}{\rm d}s
		+\int_0^t{\bf{A}_e}\vec{q}_2(s){\rm d}s+\vec{q}_{0,e}\\	
		\|\vec{e}(t)\|_1 
		&\leq 
		\int_0^t\|{\bf{A}}(u_1)\|_1\|\vec{e}(s)\|_1{\rm d}s+E_h(t)
	\end{align*}
	where $\|\cdot\|_1$ is the 1-norm and 
	\begin{eqnarray*}
		E_h(t) = \int_0^t\|\vec{f_e}\|_1{\rm d}s
		+\int_0^t\|{\bf{A}_e}\|_1\|\vec{q}_2(s)\|_1{\rm d}s
		+\|\vec{q}_{0,e} \|_1.
	\end{eqnarray*}
	With $E_h(t)$ being non-negative and non-decreasing, by Gronwall's inequality, we have
	\begin{equation}
		\|\vec{e}(t)\|_1 \leq E_h(t)\int_0^t\|{\bf{A}}(u_1)\|_1 {\rm d}s
	\end{equation}
	where $u_1$, $u_2$ are bounded and the initial condition of the ODE system is Lipschitz continuous, so that $\vec{q}_1, \vec{q}_2$ are bounded and
	\begin{equation}
		\|\vec{e}(t)\|_1 \leq 
		K\left(
		|u_1(0)-u_2(0)|+\int_0^t|u_1(s)-u_2(s)|{\rm d}s
		\right).
	\end{equation}
	With 
	$|P(t,u_1)-P(t,u_2)|\leq \|\vec{e}(t)\|_1$, the proof is completed.
\end{proof}

For brevity, from equation \eqref{uBdyonOmega1} and \eqref{ueBdyonGamma1} we define 
$$
g_c(u) := D_c\partial_n u \  {\rm on\ }\partial\Omega, \quad g_e(u,u_e) := D_e\partial_n u_e \  {\rm on\ } \Upsilon.
$$
With Lemma \ref{pu1u2} we get  
\begin{lemma}\label{g12u1u2}
	Let $u_1$, $u_2$ $\in C(\partial\Omega_c\times[0,T])$ and $u_{e1}$, $u_{e2}$ $\in C(\Upsilon\times[0,T])$ be non-negative functions bounded by $M$, then we have
	\begin{align*}
		|g_c(u_1)-g_c(u_2)|
		&\leq K_1|u_1-u_2|,
		\\
		|g_e(u_1,u_{e1})-g_e(u_2,u_{e2})|
		&\leq K_2(|u_1-u_2|+|u_{e1}-u_{e2}|+|u_1(0)-u_2(0)|)
		\\ 
		&+K_2\int_0^t|u_1(s)-u_2(s)|{\rm d}s,
	\end{align*}
	where $K_1$, $K_2$ are positive constants that depend on $M$, but do not depend on $u_1, u_2$ and $u_{e1}$, $u_{e2}$.
\end{lemma}
\begin{lemma}\label{trace1} \cite{Douglas1973}
	Let $D$ be the domain with appropriate boundary, there exists a positive constant $C_T$ such that for $0<\epsilon\leq 1$
	\begin{equation}
		\|v\|_{L^2(\partial D)}\leq 
		C_T\left(\epsilon\|\nabla v\|_{L^2(D)}+\epsilon^{-1}\|v\|_{L^2(D)}\right), \quad v\in H^1(D).
	\end{equation}
\end{lemma}
\begin{remark}
	Lemma \ref{g12u1u2} together with smooth enough $g_c, g_e$ and \eqref{gcge1} to \eqref{gcge3}
	\begin{eqnarray}
		&&g_c(0)\geq 0, \ g_c(c_o)\leq 0, \ g_c(u)\leq C_c,\text{ for }  u\geq 0 , c_o, C_c>0 \label{gcge1}\\
		&&g_e(0,u_e)\leq 0, \ g_e(u,0)\geq 0, \text{ for }  u, u_e\geq 0 \\ 
		&&-K_5u_e-K_6\leq g_e(u,u_e)\leq K_3u+K_4, \text{ for }  u, u_e,K_i>0, i\geq 3 \label{gcge3}
	\end{eqnarray}
	can also be viewed as conditions for the analysis in following sections. So that the analysis and numerical method can be applied to models with more general interface/membrane fluxes. 
\end{remark}
\section{Existence, Uniqueness and Boundedness}\label{E&U&B}

\subsection{Fundamental Solution}\label{FuS}
For completeness, we recall the definition and properties of the fundamental solution of the parabolic operator, for more information see \cite{Fri1967, Pao1992}. We define the operator $L$ as
$
Lu = {\partial_t u} -\mathcal{D}\Delta u,
$
where $\mathcal{D}>0$ is some coefficient. Then the fundamental solution of $L$ is
$$
\Gamma(t,x;\tau,\xi) =[4\pi \mathcal{D}(t-\tau)]^{-n/2}e^{-\frac{|x-\xi|^2}{4\mathcal{D}|t-\tau|}}.
$$
For any $x, \xi$ in $\mathbb{R}^n, n=2,3$ and $0\leq \tau<t\leq T$, the fundamental solution has bounds
\begin{align}
	|\Gamma(t,x;\tau,\xi)| 
	&\leq
	\frac{K_0}{(t-\tau)^{\mu}}
	\frac{1}{|x-\xi|^{n-2+\mu}},
	\quad 0<\mu<1 \label{Gamma1}
	\\
	\left|
	\frac{\partial\Gamma(t,x;\tau,\xi)}{\partial v(t,x)}
	\right| 
	&\leq
	\frac{K_0}{(t-\tau)^{\mu}}
	\frac{1}{|x-\xi|^{n+1-2\mu-\gamma}},
	\quad 1-\gamma/2<\mu<1 \label{Gamma2}
\end{align} 
where $K_0$ is a constant independent of $(t,x)$ and $(\tau,\xi)$.
Let $D$ be an open bounded domain with $C^{1,\beta}$ boundary.  The second initial boundary value problem is given by
\begin{equation}\label{uomega1_b}
	\left\{
	\begin{aligned} 
		&Lu(t,x) = f(t,x)
		\quad (t,x)\in D\times(0,T]
		\\
		&{\partial_n u} = g(t,x)
		\quad\quad\ \ (t,x)\in \partial D\times (0,T] 
		\\
		&u(0,x) = u_0(x) \quad\quad  x\in \overline{D}
	\end{aligned}
	\right.
\end{equation}
where $f, g, u_0$ are any given functions. The solution to \eqref{uomega1_b} is:  
\begin{align}
	u(t,x) 
	&=
	\int_0^t\int_{\partial D}
	\Gamma(t,x;\tau,\xi)
	\psi(\tau,\xi) d\xi d\tau
	+
	\int_D
	\Gamma(t,x;0,\xi)
	u_0(\xi) d\xi
	\\
	&
	+
	\int_0^t\int_D
	\Gamma(t,x;\tau,\xi)
	f(\tau,\xi) d\xi d\tau,
	\nonumber
\end{align}
where $\psi$ can be obtained from solving the following integral equation: 
\begin{align*}
	\psi(t,x) 
	&=
	2\int_0^t\int_{\partial D}
	\frac{\partial\Gamma(t,x;\tau,\xi)}{\partial v(t,x)}
	\psi(\tau,\xi) d\xi d\tau
	+
	2H(t,x),\\
	H(t,x)
	&=
	\int_{D}
	\frac{\partial\Gamma(t,x;0,\xi)}{\partial v(t,x)}
	u_0(\xi) d\xi
	+
	\int_0^t\int_{D}
	\frac{\partial\Gamma(t,x;\tau,\xi)}{\partial v(t,x)}
	f(\tau,\xi)  d\xi d\tau
	+
	g(t,x). 
	\nonumber
\end{align*}
Then, to get an explicit expression of $\psi$, same as in \cite{Pao1992}, we define
\begin{eqnarray*}
	Q
	:=
	\frac{\partial\Gamma(t,x;\tau,\xi)}{\partial v(t,x)},\quad
	Q_{j+1}
	:=
	\int_0^t\int_{\partial D}
	Q(t,x;s,y)Q_j(s,y;\tau,\xi) dy ds,
\end{eqnarray*}
where $j\geq 1$ and $Q_1 = Q$. 
The solution $\psi(t,x)$ has the explicit form
\begin{eqnarray}\label{psi}
	\psi(t,x)
	=
	2H(t,x)+
	2
	\int_0^t\int_{\partial D}
	R(t,x;\tau,\xi) H(\tau,\xi)
	d\xi d\tau.
\end{eqnarray}
$R(t,x;\tau,\xi)$ is denoted by
\begin{equation}\label{R_01}
	R(t,x;\tau,\xi) = \sum\limits_{j=1}^{\infty}Q_j(t,x;\tau,\xi)
\end{equation}
and $R(t,x;\tau,\xi)$ has the following bound
\begin{equation}\label{R_02}
	|R(t,x;\tau,\xi)| 
	\leq 
	\frac{K_1}{(t-\tau)^{\mu}}
	\frac{1}{|x-\xi|^{n+1-2\mu-\gamma}},
	\quad 1-\gamma/2<\mu<1,
\end{equation}
while $K_1$ does not depend on the variables, see \cite{Fri1967} for a proof. 
\subsection{Wellposedness}\label{WPN}
For brevity we define the parabolic operators as:
$$L_c u = {\partial_t u}-{D_c} \Delta u, \
L_b b = {\partial_t b}-{D_b} \Delta b, \
L_e u_e = {\partial_t u_e}-{D_e}\Delta u_e,$$ 
the corresponding fundamental solutions are $\Gamma_c, \Gamma_b, \Gamma_e$, and $R$ in \eqref{psi} are represented by $R_{c}$, $R_{b}$, $R_{e}$. We denote $C_{1,2}({Q}_c)$ as the space of functions  with two continuous spatial derivatives and one continuous time derivative on $Q_c$. The boundary conditions for $u$ are defined as $B_c u := \partial_n u$ on $\partial \Omega$, $B_e u:=\partial_n u$ on $\Upsilon$.    
In this section, we derive the bounds for solution (if it exits) of \eqref{uomega1}-\eqref{bBdyonOmega2} in Theorem \ref{bounds3}, then show the existence and uniqueness of the solution in Theorem \ref{EU3}. 
\begin{theorem}\label{bounds3}
	Assume $\{u,b,u_e\}$ is a solution of \eqref{uomega1}-\eqref{bBdyonOmega2}, and $u, b \in C(\overline{Q}_c)\cap C_{1,2}({Q}_c)$, $u_e \in C(\overline{Q}_e)\cap C_{1,2}(Q_e)$ with initial conditions $u(0,x), u_{e}(0,x)>0$, $0\leq b(0,x)\leq b^0$. Then for $t\in [0,T]$,   $0\leq b \leq b^0$,  $u, u_e$ are positive and bounded by a constant $M$ which does not depend on $u,b,u_e$. 
\end{theorem}
\begin{proof} 
	Equation \eqref{bomega1} can be written as
	$L_b b +K_b^-b+K_b^+ bu= K_b^-b^0$,
	the first observation from it and \eqref{bBdyonOmega1} is that, by maximum principle (see Theorem 1.4-1.5, Chapter 2 in \cite{Pao1992}), and as long as $u\geq 0$, we have $b\geq 0$. Next, since  $b^0$ is the solution of 
	$L_b b = K_b^-(b^0-b)$,
	by comparison theorem, see \cite{Fri1967, Pao1992}, we  get
	$b\leq b^0$. The initial values of $u, u_e$ are positive, so if $u, u_e$ are not always positive on $\overline{Q}_c$, then suppose one of them, e.g., $u$ becomes $0$ no later than $u_e$ and define $\underline{t} = \min\{\tau\geq 0| u(\tau,\underline{x})=0, \underline{x}\in \overline{\Omega}_c\}$
	so that $u(\underline{t},\underline{x}) = 0$. Equation \eqref{uomega1} can be written as $L_c u +K_b^+ bu= K_b^-(b^0-b)$, where $0\leq b\leq b^0$ on $[0,\underline{t}]$, by maximum principle and $g_c(0)\geq 0$, $-g_e(0,u_e)\geq 0$, we conclude that $\underline{t}$ does not exist. The positivity of $u_e$ can be obtained similarly.
	
	As $u$, $u_e$, $b$ are non-negative and $b$ is bounded,  we can see that $u$ is bounded by $u_e$. The following equation \eqref{u-u_0} is used to get the bound of $u$:
	\begin{equation}\label{u-u_0}
		\left\{
		\begin{aligned} 
			&L_c w = 
			K_b^-b^0,  
			\\
			&B_cw = C_3c_o/D_c,\ B_ew = (C_1^{e}+C_3^e)u_e/D_c, 
			\\
			&
			w(0,x) = 0.
		\end{aligned}
		\right.
	\end{equation}
	We define a constant $w_0 > u(0,x)$ for $x\in\overline{\Omega}_c,$ then by comparison theorem, we have $u(t,x)\leq w(t,x)+w_0$ on $\overline{\Omega}_c\times[0,T]$.
	Just like \eqref{uomega1_b} in Section \ref{FuS}, the solution $w$ can be obtained as 
	\begin{align*}
		w(t,x) 
		&=
		\int_0^t\int_{\partial \Omega_c}
		\Gamma_c(t,x;\tau,\xi)
		\psi_w(\tau,\xi) d\xi d\tau
		+
		K_b^-b^0\int_0^t\int_{\Omega_c}
		\Gamma_c(t,x;\tau,\xi)
		d\xi d\tau, \nonumber
		\\
		\psi_w(t,x)
		&=
		2\int_0^t\int_{\partial \Omega_c}
		R_{c}(t,x;\tau,\xi) H_w(\tau,\xi)
		d\xi d\tau
		+ 2H_w(t,x),\nonumber\\
		H_w(t,x)
		&=
		K_b^-b^0\int_0^t\int_{\Omega_c}
		\frac{\partial\Gamma_c(t,x;\tau,\xi)}{\partial v(t,x)}
		d\xi d\tau 
		+
		g_w(t,x),
	\end{align*}
	where $g_w(t,x) = B_c w$ on $\partial \Omega$ and $g_w(t,x) = B_ew$ on ${\Upsilon}$.
	Let $t\leq h,$ $h>0$, and define the norm
	$\|u_e\|_h = \sup\{|u_e(t,x)|; 0\leq t\leq h, x\in \overline{\Omega}_e\}$, by the properties of $\Gamma_c$ as in \eqref{Gamma1}-\eqref{Gamma2}, $R_{c}$ as in  \eqref{R_01}-\eqref{R_02},  we have 
	\begin{align*}
		|H_w(t,x)|
		&\leq
		C_1^w
		t^{1-\mu}
		+
		(C_1^{e}+C_3^e)\|u_e\|_h/D_c
		+C_3c_o/D_c,
		\\
		|\psi_w(t,x)|
		&\leq
		C_2^w (t^{1-\mu}+t^{2-2\mu}+c_o)
		+
		C_3^w(1+t^{1-\mu})\|u_e\|_h,
		\\
		|w(t,x)| 
		&\leq
		C_4^w
		(t^{1-\mu}
		+t^{2-2\mu}
		+t^{3-3\mu}
		)
		+
		C_5^w(t^{1-\mu}+t^{2-2\mu})\|u_e\|_h,
	\end{align*}
	so that on $\overline{\Omega}_c\times [0,h]$, $u$ can be bounded as
	\begin{equation}\label{ubu_e}
		|u(t,x)|\leq C_4^w
		(t^{1-\mu}
		+t^{2-2\mu}
		+t^{3-3\mu}
		)
		+
		C_5^w(t^{1-\mu}+t^{2-2\mu})\|u_e\|_h+w_0.
	\end{equation}
	Here, $C_1^w$ to $C_5^w$ do not depend on $(t,x)$ and $u_e$.
	
	We then prove that $u_e$ is bounded by $u$ from the following equation \eqref{u_e-u_e0}:
	\begin{equation}\label{u_e-u_e0}
		\left\{
		\begin{aligned} 
			&L_e v = 
			0, 
			\\
			&\partial_n v = 
			\left(
			C_1^{e}+{2C_2^e}/{(K_s m)}+C_3^e
			\right)
			u /D_e,
			\\
			&
			v(0,x) = 0, 
		\end{aligned}
		\right.
	\end{equation}
	where we define $v_0$ as a constant and $v_0 > u_e(0,x)$ for $x\in\overline{\Omega}_e$. By the comparison theorem, $u_e(t,x)\leq v(t,x)+v_0$ on $\overline{\Omega}_e\times[0,T]$.
	The solution $v$ can be obtained as
	\begin{align*}
		v(t,x) 
		&=
		\int_0^t\int_{\Upsilon}
		\Gamma_e(t,x;\tau,\xi)
		\psi_v(\tau,\xi) d\xi d\tau,
		\\
		\psi_v(t,x)
		&=
		2\int_0^t\int_{\Upsilon}
		R_{e}(t,x;\tau,\xi) g_v(\tau,\xi)
		d\xi d\tau+2g_v(t,x),\nonumber
	\end{align*}
	where 
	$g_v(t,x) = \partial_n v$.
	When defining the norm
	$\|u\|_h = \sup\{|u(t,x)|; 0\leq t\leq h, x\in \overline{\Omega}_c\}$, and by the properties of $\Gamma_e(t,x;\tau,\xi)$ from \eqref{Gamma1}-\eqref{Gamma2}, $R_{e}$ from \eqref{R_01}-\eqref{R_02}, we get
	\begin{eqnarray*}
		|\psi_v(t,x)|
		\leq
		C_1^v(1+t^{1-\mu})\|u\|_h, 
		\quad
		|v(t,x)| 
		\leq
		C_2^v(t^{1-\mu}+t^{2-2\mu})\|u\|_h,
	\end{eqnarray*}
	so that on $\overline{\Omega}_e\times [0,h]$, $u_e$ can be bounded by
	\begin{equation}\label{u_ebu}
		|u_e(t,x)|\leq
		C_2^v(t^{1-\mu}+t^{2-2\mu})\|u\|_h+v_0,
	\end{equation}
	where $C_1^v$, $C_2^v$ do not depend on $(t,x)$ and $u$.
	With \eqref{ubu_e}, \eqref{u_ebu}, and $h$ small enough, we have
	$\|u\|_h
	\leq
	\frac12\|u_e\|_h+w_0+\frac34C_4^w$, 
	$\|u_e\|_h 
	\leq
	\frac12\|u\|_h+v_0
	$,
	so that 
	\begin{eqnarray*}
		\|u\|_h
		\leq
		\frac23v_0+\frac43w_0+C_4^w,
		\quad
		\|u_e\|_h 
		\leq
		\frac43v_0+\frac23w_0+\frac12C_4^w.
	\end{eqnarray*}
	Then, on time interval $[h,2h]$, let 
	$
	w_1
	=
	\frac23v_0+\frac43w_0+C_4^w$, 
	$
	v_1 
	=
	\frac43v_0+\frac23w_0+\frac12C_4^w,
	$
	by \eqref{u-u_0} and \eqref{u_e-u_e0}, we have
	$u(t,x)\leq w(t,x)+w_1$ on $\overline{\Omega}_c\times[h,2h]$ and
	$u_e(t,x)\leq v(t,x)+v_1$ on $\overline{\Omega}_e\times[h,2h]$.
	With the same coefficients as in the previous step, we get similar bounds of $u, u_e$ for $t\in [h,2h]$. Since $h$ is fixed, with finite steps, we can reach the conclusion that $u$ and $u_e$ are bounded for $t\in [0,T]$.
\end{proof}

\begin{theorem}\label{EU3}
	Suppose $u(0,x), u_e(0,x) >0$ and $b^0\geq b(0,x)\geq 0$ are continuously differentiable in $\Omega_c$ or $\Omega_e$, then there exists a unique solution $\{u, b, u_e\}$ for \eqref{uomega1}-\eqref{bBdyonOmega2}. 
\end{theorem}
\begin{proof}
	From Theorem \ref{bounds3}, we can get the bounds of $u, u_e$ and $b$ if the solution $\{u, b, u_e\}$ exits and its components are smooth. Let $M$ be the upper bound of those three. We define $\phi(x) = \max(0,\min(x,M))$, but $\phi$ can be smoother if needed, see Section \ref{NLP1LA}.	Then we change $u, b$ to be $\phi(u), \phi(b)$ in $f(b, u)$ for equations \eqref{uomega1}-\eqref{bomega1}, and replace $u, u_e$ by $\phi(u), \phi(u_e)$ in the right hand sides of \eqref{uBdyonOmega1}-\eqref{ueBdyonGamma1}. We assert that the modified problem has a unique solution and the solution has the same bounds. Thus, it is the same solution of the original system \eqref{uomega1}-\eqref{bBdyonOmega2}.  Define the map $T$ as $\{u,b,u_e\} = T\{w,w_b,w_e\}$ for the modified problem 
	\begin{equation}\label{ube_w}
		\left\{
		\begin{aligned} 
			&L_c u = 
			f(\phi(w_b),\phi(w)),  \\
			&L_b b  =
			f(\phi(w_b),\phi(w)), \\
			&L_e u_e = 0, \\
			&B_cu  = g_c(\phi(w))/D_c, \ B_eu = -g_e(\phi(w),\phi(w_e))/D_c,\\
			&\partial_n u_e= g_e(\phi(w),\phi(w_e))/D_e,\ \partial_n b=0,
		\end{aligned} 
		\right.
	\end{equation}
	where $\{w,w_b,w_e\}$ are given functions, $\{u,b,u_e\}$ is the solution of \eqref{ube_w}.
	Next we show $T$ is a contraction map if $t\leq h$ for $h$ small enough. To prove this, let the entries of $\{w_1,w_{b,1},w_{e,1}\}$, $\{w_2,w_{b,2},w_{e,2}\}$ be Hölder continuous functions and $\{u_1,b_1,u_{e,1}\} = T\{w_1,w_{b,1},w_{e,1}\}$, 
	$\{u_2,b_2,u_{e,2}\} = T\{w_2,w_{b,2},w_{e,2}\}$. Then we define the norms
	$$
	\|v\|_h = \sup\{|v(t,x)|; 0\leq t\leq h, x\in \overline{D}\}, \
	\|\{v,v_{b},v_{e}\}\|_h =  \|v\|_h+\|v_b\|_h+\|v_e\|_h,
	$$ 
	where $\overline{D}$ can be $\overline{\Omega}_c$ or $\overline{\Omega}_e$. Let $v = u_1-u_2, v_b = b_1-b_2, v_e = u_{e,1}-u_{e,2}$, and $q=w_1-w_2$, $q_b=w_{b,1}-w_{b,2}$, $q_e=w_{e,1}-w_{e,2}$. Since the system \eqref{ube_w} is fully decoupled, $v, v_b, v_e$ can be solved separately.  We start with $v$ in \eqref{v_c}
	\begin{equation}\label{v_c}
		\left\{
		\begin{aligned} 
			&L_c v = 
			- K_b^+q_b-K_b^+(\phi(w_{b,1})\phi(w_1) - \phi(w_{b,2})\phi(w_2)),
			\\
			&B_cv  = ( g_c(\phi(w_1))-g_c(\phi(w_2)) )/D_c,
			\\ 
			&B_ev  = ( g_e(\phi(w_1),\phi(w_{e,1}))-g_e(\phi(w_2),\phi(w_{e,2}))
			)/D_c,
			\\
			&v(0,x) = 0,
		\end{aligned}
		\right.
	\end{equation}
	and by Lemma \ref{g12u1u2}, we have
	\begin{align*}
		|g_c(\phi(w_1))-g_c(\phi(w_2))| &\leq  K_1|q|, \\
		|g_e(\phi(w_1),\phi(w_{e,1}))-g_e(\phi(w_2),\phi(w_{e,2}))|&\leq K_2\left(|q|+|q_e|+\int_0^t|q(s)|{\rm d}s\right).
	\end{align*}
	Further, let $f_v(t,x) = - K_b^+q_b-K_b^+(\phi(w_{b,1})\phi(w_1) - \phi(w_{b,2})\phi(w_2))$, $g(t,x) = B_cv$ on $\partial \Omega$ and $g(t,x) = B_ev$ on ${\Upsilon}$. The solution of \eqref{v_c} is
	\begin{align*}
		&v(t,x) 
		=
		\int_0^t\int_{\partial \Omega_c}
		\Gamma_c(t,x;\tau,\xi)
		\psi_v(\tau,\xi) d\xi d\tau
		+
		\int_0^t\int_{\Omega_c}
		\Gamma_c(t,x;\tau,\xi)
		f_v(\tau,\xi) d\xi d\tau,
		\\
		&\psi_v(t,x)
		=
		2H_v(t,x)+
		2\int_0^t\int_{\partial \Omega_c}
		R_{c}(t,x;\tau,\xi) H_v(\tau,\xi)
		d\xi d\tau,\\
		&H_v(t,x)
		=
		\int_0^t\int_{\Omega_c}
		\frac{\partial\Gamma_c(t,x;\tau,\xi)}{\partial v(t,x)}
		f_v(\tau,\xi)  d\xi d\tau
		+
		g(t,x).
	\end{align*}
	By the properties of $\Gamma_c, R_{c}$, and $t\leq h$, we have
	\begin{align*}
		|H_v(t,x)|
		&\leq
		C_1^v
		(t^{1-\mu}
		+1
		+
		t)(\|q\|_h+\|q_b\|_h+\|q_e\|_h),
		\\
		|\psi_v(t,x)|
		&\leq
		C_2^v
		(t^{1-\mu}
		+1
		+t
		+t^{2-\mu}
		+t^{2-2\mu}
		)(\|q\|_h+\|q_b\|_h+\|q_e\|_h),\\
		|v(t,x)| 
		&\leq
		C_3^v
		(t^{1-\mu}
		+t^{2-\mu}
		+t^{2-2\mu}
		+t^{3-2\mu}
		+t^{3-3\mu}
		)(\|q\|_h	+\|q_b\|_h+\|q_e\|_h), \nonumber
	\end{align*}
	where $C_1^v, C_2^v, C_3^v$ do not depend on $q, q_b, q_e$ or $(t,x)$.
	So we can choose $h$ small enough such that
	$\|v\|_h\leq 1/6(\|q\|_h+\|q_b\|_h+\|q_e\|_h)$. 
	Similarly, we can obtain the bounds of $v_b, v_e$ as
	$\|v_b\|_h\leq 1/6(\|q\|_h+\|q_b\|_h)$, 
	$\|v_e\|_h\leq 1/6(\|q\|_h+\|q_e\|_h)$. 
	Summing these terms we can show that $T$ is a contraction map
	\begin{eqnarray*}
		\|T\{w_1,w_{b,1},w_{e,1}\}-T\{w_2,w_{b,2},w_{e,2}\}\|_h
		\leq 
		\frac12
		\|\{w_1, w_{b,1}, w_{e,1}\}
		-
		\{w_2, w_{b,2}, w_{e,2}\}\|_h,
	\end{eqnarray*}
	where $t\in [0,h].$ By iteration, we can get the unique solution $\{u,b,u_e\}$ on $[0,h]$. Then, choosing
	$u(h,x), b(h,x), u_e(h,x)$ as the initial value, and repeating the steps above, the existence and uniqueness of the solution on $[0,T]$ can be obtained. Following the proof in Theorem \ref{bounds3}, it's easy to see the bounds of the solution for the modified system are the same as the original system. For regularity of the solution, we refer to \cite{Fri1967} for more information.
\end{proof}  
\section{Galerkin Projection and the Error Estimates}\label{errorGP}
In this section, we define a nonlinear problem \eqref{NLP1} in variational form, which is used in Section \ref{GP} to define the Galerkin projection of $u, u_e$. The wellposedness of \eqref{NLP1} is proved In Section \ref{NLP1LA}. From Section \ref{WPN}, we get the bounds of the exact solution $\{u, b, u_e\}$ for \eqref{uomega1}-\eqref{bBdyonOmega2}, the bounds are used to define the function $\phi(\cdot)$, such that $\phi(u)=u$, $\phi(b)=b$ and $\phi(u_e)=u_e$ when $u, b, u_e$ are within the bounds. Since $b$ is not coupled with $u_e$, we can use the normal Galerkin projection for it, which will be given later. We now replace $u$ in $J_R$, $J_S$, $J_N$ by $\phi(u)$ and $u_e$ in $J_R$ by $\phi(u_e)$, see Section \ref{NLP1LA} for a smoother $\phi$. For brevity, we denote 
$\bar{g}_c(u)$ and $\bar{g}_e(u,u_e)$ as the modified nonlinear boundary conditions. Then the variational form is given by   
\begin{equation}\label{NLP1}
	\left\{
	\begin{aligned} 
		&a_c(u,v)+\lambda(u,v) 
		= <\bar{g}_c(u),v>_{\partial\Omega} - <\bar{g}_e(u,u_e),v>_{\Upsilon},\\
		&a_e(u_e,v_e)+\lambda(u_e,v_e) 
		= <\bar{g}_e(u,u_e),v_e>_{\Upsilon},
	\end{aligned}
	\right.
\end{equation}
where $a_c(u,v) = (D_c\nabla u,\nabla v)_{\Omega_c}$, $a_e(u_e,v_e) = (D_e\nabla u_e,\nabla v_e)_{\Omega_e}$.  
The problem is to find solutions $u\in L^\infty(0,T;H^1(\Omega_c))$, $u_e\in L^\infty(0,T;H^1(\Omega_e))$, for any $v\in H^1(\Omega_c)$, $v_e\in H^1(\Omega_e)$ and $\lambda >0$ large enough. 

\subsection{Wellposedness of the Variational Problem}\label{NLP1LA}
The existence and uniqueness for the solution of \eqref{NLP1} are proved in this section, well-posedness of the Galerkin Projection \eqref{NLPJ} can be obtained similarly. 

We define a smoother function $\phi(\cdot)$ which is used in the variational problem
$$
\phi(x) = 
\left\{
\begin{aligned} 
	&-a, & {\rm \ if\ } &x\leq -a\\
	&\frac{3 x^5}{a^4}+\frac{7 x^4}{a^3}+\frac{4 x^3}{a^2}+x, 
	& {\rm \ if\ } &-a <x<0\\ 
	&x, & {\rm \ if\ } &0\leq x \leq M\\ 
	&\frac{3 (x-M)^5}{a^4}-\frac{7 (x-M)^4}{a^3}+\frac{4 (x-M)^3}{a^2}+x, & {\rm \ if\ } &M< x \leq M+a\\ 
	&M+a, & {\rm \ if\ } &x>M+a
\end{aligned}
\right.
$$
where $a>0, M >0$, $a$ is small enough. 
\begin{lemma}\label{NLP1L} 
	The nonlinear problem \eqref{NLP1} has a unique solution, provided the modified conditions \eqref{uBdyonOmega1}-\eqref{ueBdyonGamma1}.
\end{lemma}
\begin{proof}
	The method we use here is adapted from Appendix in \cite{calabro2013}. We try to get the exact solution by iteration, let $u^0\in H^1(\overline{Q}_c)$, $u_e^0\in H^1(\overline{Q}_e)$ be given functions and $n\geq 0$, then we have
	\begin{align}
		&a_c(u^{n+1},v)+\lambda(u^{n+1},v) 
		= <-\bar{g}_e(u^n,u_e^n),v>_{\Upsilon}+<\bar{g}_c(u^n),v>_{\partial\Omega},\quad \\
		&a_e(u_e^{n+1},v_e)+\lambda(u_e^{n+1},v_e) 
		= <\bar{g}_e(u^n,u_e^n),v_e>_{\Upsilon},
	\end{align}
	where $u^{n+1}$, $u_e^{n+1}$ share same initial values as $u^0,u_e^0$.\\
	Let $r^{n+1}=u^{n+1}-u^{n}$, $r_e^{n+1}=u_e^{n+1}-u_e^{n}$ and $n\geq 1$,  we have
	\begin{align*}
		D_c(\nabla r^{n+1},\nabla v)+\lambda(r^{n+1},v) 
		&= <-\bar{g}_e(u^n,u_e^n)+\bar{g}_e(u^{n-1},u_e^{n-1}),v>_{\Upsilon}\\
		&+<\bar{g}_c(u^n)-\bar{g}_c(u^{n-1}),v>_{\partial\Omega}, \\
		D_e(\nabla r_e^{n+1},\nabla v_e)+\lambda(r_e^{n+1},v_e) 
		&= <\bar{g}_e(u^n,u_e^n)-\bar{g}_e(u^{n-1},u_e^{n-1}),v_e>_{\Upsilon},
	\end{align*}
	then let $v=r^{n+1}$ and $v_e=r_e^{n+1}$, we obtain
	\begin{align*}
		D_c\|\nabla r^{n+1}\|^2 +\lambda\|r^{n+1}\|^2  
		&= <-\bar{g}_e(u^n,u_e^n)+\bar{g}_e(u^{n-1},u_e^{n-1}),r^{n+1}>_{\Upsilon}\\
		&+<\bar{g}_c(u^n)-\bar{g}_c(u^{n-1}),r^{n+1}>_{\partial\Omega}, \\
		D_e\|\nabla r_e^{n+1}\|^2 +\lambda\|r_e^{n+1}\|^2 
		&= <\bar{g}_e(u^n,u_e^n)-\bar{g}_e(u^{n-1},u_e^{n-1}),r_e^{n+1}>_{\Upsilon}.
	\end{align*}
	Sum those two equations, we get \eqref{eg1}
	\begin{equation}\label{eg1}
		\begin{aligned}
			D_c\|\nabla r^{n+1}\|^2 &+\lambda\|r^{n+1}\|^2  
			+
			D_e\|\nabla r_e^{n+1}\|^2 +\lambda\|r_e^{n+1}\|^2 \\
			=& 
			<\bar{g}_c(u^n)-\bar{g}_c(u^{n-1}),r^{n+1}>_{\partial\Omega} \\
			&
			+
			<\bar{g}_e(u^n,u_e^n)-\bar{g}_e(u^{n-1},u_e^{n-1}),r_e^{n+1}-r^{n+1}>_{\Upsilon}. 
		\end{aligned}
	\end{equation}
	By Lemma \ref{g12u1u2}, the first term of the right hand side is bounded: 
	\begin{eqnarray}\label{eg2}
		<\bar{g}_c(u^n)-\bar{g}_c(u^{n-1}),r^{n+1}>_{\partial\Omega} 
		\leq
		\frac{K_1}{2}\|r^{n}\|^2_{L^2(\partial\Omega)}
		+
		\frac{K_1}{2}\|r^{n+1}\|^2_{L^2(\partial\Omega)},\quad
	\end{eqnarray}
	and the second term of the right hand side is bounded: 
\begin{equation}\label{eg3}
	\begin{aligned}
		<\bar{g}_e(u^n,u_e^n)&-\bar{g}_e(u^{n-1},u_e^{n-1}),r_e^{n+1}-r^{n+1}>_{\Upsilon}
		\\
		&\leq
		K_2<|r^n|+|r_e^n|+\int_0^t|r^n(s)|{\rm d}s, r_e^{n+1}-r^{n+1}>_{\Upsilon}\\
		&\leq
		\frac{K_2}{2}\|r^{n}\|^2_{L^2({\Upsilon})}+\frac{K_2}{2}\|r_e^{n}\|^2_{L^2({\Upsilon})}
		+
		\frac{K_2}{2}\int_0^t\|r^n(s)\|^2_{L^2({\Upsilon})}{\rm d}s \\
		&+
		K_2(2+T)
		\left(
		\|r^{n+1}\|^2_{L^2({\Upsilon})}
		+
		\|r_e^{n+1}\|^2_{L^2({\Upsilon})}
		\right).
	\end{aligned}
\end{equation}
	By equations \eqref{eg1} to \eqref{eg3}, we have
\begin{equation*}
	\begin{aligned}
		&
		D_c\|\nabla r^{n+1}\|^2 +\lambda\|r^{n+1}\|^2  
		+
		D_e\|\nabla r_e^{n+1}\|^2 +\lambda\|r_e^{n+1}\|^2 \\
		&\leq
		\frac{K_1}{2}\|r^{n}\|^2_{L^2(\partial\Omega)}
		+
		\frac{K_2}{2}\|r^{n}\|^2_{L^2({\Upsilon})}+
		\frac{K_2}{2}\|r_e^{n}\|^2_{L^2({\Upsilon})}
		+
		\frac{K_2}{2}\int_0^t\|r^n(s)\|^2_{L^2({\Upsilon})}{\rm d}s \\
		&
		+
		\frac{K_1}{2}\|r^{n+1}\|^2_{L^2(\partial\Omega)}
		+
		K_2(2+T)
		\left(
		\|r^{n+1}\|^2_{L^2({\Upsilon})}
		+
		\|r_e^{n+1}\|^2_{L^2({\Upsilon})}
		\right),
	\end{aligned}
\end{equation*}
take $L^{\infty}$ norm of $\|r^{n}\|^2_{L^2(\partial\Omega)}, \|r^{n}\|^2_{L^2({\Upsilon})}, \|r_e^{n}\|^2_{L^2({\Upsilon})},$ we obtain \eqref{eg4}
\begin{equation}\label{eg4}
\begin{aligned}
&
D_c\|\nabla r^{n+1}\|^2 +\lambda\|r^{n+1}\|^2  
+
D_e\|\nabla r_e^{n+1}\|^2 +\lambda\|r_e^{n+1}\|^2 \\
&\leq
C_1
\left(
\|r^{n}\|^2_{L^\infty(0,T;L^2(\partial\Omega))}
+
\|r^{n}\|^2_{L^\infty(0,T;L^2({\Upsilon}))}
+
\|r_e^{n}\|^2_{L^\infty(0,T;L^2({\Upsilon}))}
\right) \\
&
+
C_2
\left(
\|r^{n+1}\|^2_{L^2(\partial\Omega)}
+
\|r^{n+1}\|^2_{L^2({\Upsilon})}
+
\|r_e^{n+1}\|^2_{L^2({\Upsilon})}
\right). 
\end{aligned}
\end{equation}
	Let $\alpha >1$, add the same term to both sides of \eqref{eg4}, we have \eqref{eg5} 
\begin{equation}\label{eg5} 
	\begin{aligned}
	&
	D_c\|\nabla r^{n+1}\|^2 +\lambda\|r^{n+1}\|^2  
	+
	D_e\|\nabla r_e^{n+1}\|^2 +\lambda\|r_e^{n+1}\|^2 \\
	&
	+
	\alpha C_1
	\left(
	\|r^{n+1}\|^2_{L^2(\partial\Omega)}
	+
	\|r^{n+1}\|^2_{L^2({\Upsilon})}
	+
	\|r_e^{n+1}\|^2_{L^2({\Upsilon})}
	\right) \\
	&\leq
	C_1
	\left(
	\|r^{n}\|^2_{L^\infty(0,T;L^2(\partial\Omega))}
	+
	\|r^{n}\|^2_{L^\infty(0,T;L^2({\Upsilon}))}
	+
	\|r_e^{n}\|^2_{L^\infty(0,T;L^2({\Upsilon}))}
	\right) \\
	&
	+
	(\alpha C_1+C_2)
	\left(
	\|r^{n+1}\|^2_{L^2(\partial\Omega)}
	+
	\|r^{n+1}\|^2_{L^2({\Upsilon})}
	+
	\|r_e^{n+1}\|^2_{L^2({\Upsilon})}
	\right).
	\end{aligned}
\end{equation}
	From Lemma \ref{trace1} and $\lambda$ is large enough, we have
	\begin{equation}\label{eg6}
	\begin{aligned}
		(\alpha C_1&+C_2)
		\left(
		\|r^{n+1}\|^2_{L^2(\partial\Omega)}
		+
		\|r^{n+1}\|^2_{L^2({\Upsilon})}
		+
		\|r_e^{n+1}\|^2_{L^2({\Upsilon})}
		\right)\\
		&\leq 
		D_c\|\nabla r^{n+1}\|^2 +\lambda\|r^{n+1}\|^2  
		+
		D_e\|\nabla r_e^{n+1}\|^2 +\lambda\|r_e^{n+1}\|^2. \\
	\end{aligned}
\end{equation}
	So that combine \eqref{eg5} and \eqref{eg6}, we obtain \eqref{eg7}  
\begin{equation}\label{eg7}
	\begin{aligned}
		&
		\|r^{n+1}\|^2_{L^\infty(0,T;L^2(\partial\Omega))}
		+
		\|r^{n+1}\|^2_{L^\infty(0,T;L^2({\Upsilon}))}
		+
		\|r_e^{n+1}\|^2_{L^\infty(0,T;L^2({\Upsilon}))}
		\\
		&\leq
		\frac{1}{\alpha }
		\left(
		\|r^{n}\|^2_{L^\infty(0,T;L^2(\partial\Omega))}
		+
		\|r^{n}\|^2_{L^\infty(0,T;L^2({\Upsilon}))}
		+
		\|r_e^{n}\|^2_{L^\infty(0,T;L^2({\Upsilon}))}
		\right). 
	\end{aligned}
\end{equation}
	From \eqref{eg4} and \eqref{eg7}, let $\tau = \min\{D_c,D_e\}, C_\tau =\frac{C_1+C_2/\alpha}{\tau}$, we arrive at
	\begin{eqnarray*}\label{eg10}
		&&\|r^{n+1}\|^2_{L^\infty(0,T;H^1(\Omega_c))} 
		+
		\|r_e^{n+1}\|^2_{L^\infty(0,T;H^1(\Omega_e))}  \\
		&&\leq
		C_\tau
		\left(
		\|r^{n}\|^2_{L^\infty(0,T;L^2(\partial\Omega))}
		+
		\|r^{n}\|^2_{L^\infty(0,T;L^2({\Upsilon}))}
		+
		\|r_e^{n}\|^2_{L^\infty(0,T;L^2({\Upsilon}))}
		\right) \nonumber\\
		&&\leq
		\frac{C_\tau}{\alpha^n}
		\left(
		\|r^{1}\|^2_{L^\infty(0,T;L^2(\partial\Omega))}
		+
		\|r^{1}\|^2_{L^\infty(0,T;L^2({\Upsilon}))}
		+
		\|r_e^{1}\|^2_{L^\infty(0,T;L^2({\Upsilon}))}
		\right). \nonumber
	\end{eqnarray*}
	which implies that $u^n,u_e^n$ converge to a unique solution. 
\end{proof}

\subsection{Galerkin Projection}\label{GP}
We define $S_k(\Omega_c)$, $S_k(\Omega_e)$, $k\geq 1$ as the classical conforming $P_k$ finite element spaces, with shape regular triangular/tetrahedral meshes and the size of the corresponding mesh is $h$. For curved boundaries, we refer to isoparametric finite element methods, see \cite{09iso,iso1986} and references therein.  The Galerkin projection for the exact solutions $u,u_e$ can be obtained from \eqref{NLPJ}:
\begin{equation}\label{NLPJ}
	\left\{
	\begin{aligned} 
		&a_c(u-W,v)+\lambda(u-W,v)  = <-\bar{g}_e(u,u_e)+\bar{g}_e(W,W_e),v>_{\Upsilon} \\
		& \qquad\qquad\qquad\qquad\qquad\qquad\quad    +<\bar{g}_c(u)-\bar{g}_c(W),v>_{\partial\Omega}, \\
		&a_e(u_e-W_e,v_e)+\lambda(u_e-W_e,v_e) = <\bar{g}_e(u,u_e)-\bar{g}_e(W,W_e),v_e>_{\Upsilon}  .
	\end{aligned}
	\right.
\end{equation}
where $W(t), v\in S_k(\Omega_c)$, and $W_e(t),v_e\in S_k(\Omega_c)$, $\lambda$ is large enough.  

Also without further notice, $\eta, \eta_e$ are defined as $\eta := u-W$, $\eta_e :=u_e-W_e$, the following norms are denoted as $\|\cdot\|_{\infty; k} := \|\cdot\|_{L^\infty(0,T;H^k(D))}$, $\|\cdot\|_{\infty; 0} := \|\cdot\|_{L^\infty(0,T;L^2(D))}$,  and $\|\cdot\|:=\|\cdot\|_{L^2(D)}$, where $D$ can be $\Omega_c$ or $\Omega_e$ depending on the domain of the given function. Then we have the following theorem:
\begin{theorem}\label{p_teta}  
	With $P_k$ elements, $k\geq 1$, and appropriately chosen $W(0), W_e(0)$, we have 
	\begin{align}
		\|\eta\|_{\infty;1}
		+
		\|\eta_e\|_{\infty;1} 
		&\leq 
		C(
		\|u\|_{\infty; {k+1}}
		+
		\|u_e\|_{\infty; {k+1}}
		)h^k, \label{NLPJH1}
		\\
		\left\| {\partial_t \eta} \right\|_{\infty;{1}}
		+
		\left\| {\partial_t \eta_e} \right\|_{\infty;{1}} 
		&\leq 
		C(
		\|u\|_{\infty; {k+1}}
		+
		\left\| {\partial_t u} \right\|_{\infty; {k+1}}
		)h^k \label{dtNLPJH1}
		\\
		&
		+
		C
		(
		\|u_e\|_{\infty; {k+1}}
		+
		\left\| {\partial_t u_e} \right\|_{\infty; {k+1}}
		)
		h^k,
		\nonumber
	\end{align} 
	where $h$ is the mesh size and $C$ does not depend on $h$.
\end{theorem}
\begin{proof}
	First step, we prove \eqref{NLPJH1}. Letting $\delta u = w-u$, $\delta u_e = w_e-u_e$, 
	$w\in S_k(\Omega_c), w_e \in S_k(\Omega_e)$, \eqref{NLPJ} can be written as
	\begin{equation}\label{NLPJ2}
		\left\{
		\begin{aligned} 
			&a_c( \eta,v)+\lambda(\eta,v) 
			= <-\bar{g}_e(u,u_e)+\bar{g}_e(W,W_e),v>_{\Upsilon} \\
			&\qquad\qquad\qquad\qquad\quad
			+<\bar{g}_c(u)-\bar{g}_c(W),v>_{\partial\Omega}, \\
			&a_e(\eta_e,v_e)+\lambda(\eta_e,v_e) 
			= <\bar{g}_e(u,u_e)-\bar{g}_e(W,W_e),v_e>_{\Upsilon}.   
		\end{aligned}
		\right. 
	\end{equation} 
	Now set $v = w-W=\delta u+\eta$, $v_e = w_e-W_e=\delta u_e+\eta_e$, then
	\begin{equation*}\label{NLPJ3}
		\left\{
		\begin{aligned} 
			&a_c(\eta, \delta u+\eta)+\lambda(\eta,\delta u+\eta) 
			= <-\bar{g}_e(u,u_e)+\bar{g}_e(W,W_e),\delta u+\eta>_{\Upsilon}\\
			& \qquad\qquad\qquad\qquad\qquad\qquad\quad\   
			+<\bar{g}_c(u)-\bar{g}_c(W),\delta u+\eta>_{\partial\Omega}, \\
			&a_e(\eta_e, \delta u_e+\eta_e)+\lambda(\eta_e,\delta u_e+\eta_e) 
			= <\bar{g}_e(u,u_e)-\bar{g}_e(W,W_e),\delta u_e+\eta_e>_{\Upsilon},
		\end{aligned}
		\right. 
	\end{equation*} 
	moving the terms containing $\delta u$ to the right hand side leads to (with Schwarz inequality, Lemma \ref{g12u1u2} and trace theorem):
	\begin{align*}
		&D_c\|\nabla \eta\|^2 +\lambda\|\eta\|^2 
		+
		D_e\|\nabla \eta_e\|^2 +\lambda\|\eta_e\|^2 \\
		&\leq 
		\frac{D_c}{4}\|\nabla \eta\|^2 
		+
		\frac{\lambda}{4}\|\eta\|^2 
		+
		\frac{D_e}{4}\|\nabla \eta_e\|^2 
		+
		\frac{\lambda}{4}\|\eta_e\|^2 
		\nonumber\\
		&+
		C_1(\|\eta\|^2_{L^2({\Upsilon})}+\|\eta_e\|^2_{L^2({\Upsilon})}+\|\eta\|^2_{L^2(\partial \Omega)}) 
		+
		C_2\int_0^t\|\eta(s)\|^2_{L^2({\Upsilon})}{\rm d}s \nonumber\\
		&+
		C_3(
		\|\eta(0)\|^2_{L^2({\Upsilon})}+\|\eta_e(0)\|^2_{L^2({\Upsilon})} 
		+
		\inf\limits_{w\in S_k(\Omega_c)}\|\delta u\|^2_{H^1(\Omega_c)}\nonumber\\
		&+
		\inf\limits_{w_e\in S_k(\Omega_e)}\|\delta u_e\|^2_{H^1(\Omega_e)}
		),
		\nonumber
	\end{align*} 
	from Lemma \ref{trace1}, we have 
	$$C_2\int_0^t\|\eta(s)\|^2_{L^2({\Upsilon})}{\rm d}s \leq \int_0^tD_c/{(4T)}\|\nabla\eta(s)\|^2+C\|\eta(s)\|^2{\rm d}s,$$ 
	the bounds for $C_1(\|\eta\|^2_{L^2({\Upsilon})}+\|\eta_e\|^2_{L^2({\Upsilon})}+\|\eta\|^2_{L^2(\partial \Omega)}) $ can be obtained similarly.  Taking $L^{\infty}$ norm with respect to $t$ then to $s$ for the right-hand side, then taking the $L^{\infty}$ norm with respect to $t$ for the left hand side, cancelling the corresponding terms, we obtain
	\begin{align}\label{NLPJ5}
		\|\eta\|^2_{\infty;1}
		+
		\|\eta_e\|^2_{\infty;1} 
		&\leq 
		C(
		\|\eta(0)\|^2_{H^1(\Omega_c)}+\|\eta_e(0)\|^2_{H^1(\Omega_e)} \nonumber \\
		&+
		\sup\limits_{t\in[0,T]}\inf\limits_{w\in S_k(\Omega_c)}\|\delta u\|^2_{H^1(\Omega_c)}
		\nonumber \\
		&
		+
		\sup\limits_{t\in[0,T]}\inf\limits_{w_e\in S_k(\Omega_e)}\|\delta u_e\|^2_{H^1(\Omega_e)}
		),\nonumber
	\end{align} 
	so that \eqref{NLPJH1} is proved. 
	
	Next step, we prove \eqref{dtNLPJH1}.
	Differentiating $\bar{g}_c$ with respect to (w.r.t.) $t$ produces
	\begin{align*}
		\frac{{\rm d} \bar{g}_c(u)}{{\rm d} t} 
		&= 
		\frac{\partial \bar{g}_c}{\partial u}(u)
		\frac{\partial u}{\partial t}, 
		\\
		\frac{\partial \bar{g}_c}{\partial u}(u)
		&=
		-
		C_1^c
		\frac{2uK_p^2}{(K_p^2+u^2)^2}
		-
		C_2^c
		\frac{K_n}{(K_n+{\phi}(u))^2}\frac{\partial \phi}{\partial u}
		-C_3^c.
	\end{align*}
	The derivative of $\bar{g}_e(u,u_e)$ w.r.t. $t$ is
	\begin{align*}
		\frac{{\rm d} \bar{g}_e(u,u_e)}{{\rm d} t} 
		&= 
		C_1^{e}(c_1'(t)+c_2'(t))(\phi(u_e)-\phi(u))\\
		&-C_1^{e}P(t,\phi(u))
		\left(
		\frac{\partial \phi}{\partial u_e}
		\frac{\partial u_e}{\partial t}
		-
		\frac{\partial \phi}{\partial u}
		\frac{\partial u}{\partial t}
		\right)
		\nonumber \\
		&
		+
		\frac{\partial \psi}{\partial u}
		\frac{\partial u}{\partial t}
		+
		\frac{\partial \psi}{\partial u_e}
		\frac{\partial u_e}{\partial t},
		\nonumber 
	\end{align*}
	where ${\psi} = J_S(\phi(u),u_e)+J_{l,e}$,  and
	\begin{align*}
		\frac{\partial \psi}{\partial u}
		&=
		C_2^e
		\frac{K_s}{(K_s+{\phi}(u))^2\phi_{m}(u_e)}\frac{\partial \phi}{\partial u}
		-C_3^e,
		\\
		\frac{\partial \psi}{\partial u_e}
		&=
		-
		C_2^e\frac{\phi(u)}{(K_s+\phi(u))(\phi_{m}(u_e))^2}\frac{\partial \phi_m}{\partial u_e}
		+C_3^e.
	\end{align*}
	Then \eqref{NLPJ} can be differentiated w.r.t. $t$ as
	\begin{equation}\label{dtNLPJ}
		\left\{
		\begin{aligned}
			&a_c
			\left(
			\frac{\partial \eta}{\partial t}
			, v
			\right)
			+
			\lambda
			\left(
			\frac{\partial \eta}{\partial t}
			,v
			\right) 
			=
			<-
			\frac{{\rm d} \bar{g}_e(u,u_{e})}{{\rm d} t}  
			+
			\frac{{\rm d} \bar{g}_e(W,W_{e})}{{\rm d} t} 
			,v>_{\Upsilon} 
			\\
			&\qquad\qquad\qquad\qquad\qquad\qquad\ 		
			+<\frac{{\rm d} \bar{g}_c(u)}{{\rm d} t}
			-
			\frac{{\rm d} \bar{g}_c(W)}{{\rm d} t},v>_{\partial\Omega}, 
			\\
			&a_e
			\left(
			\frac{\partial \eta_e}{\partial t},  v_e
			\right)
			+
			\lambda
			\left(
			\frac{\partial \eta_e}{\partial t},v_e
			\right) 
			= <
			\frac{{\rm d} \bar{g}_e(u,u_e)}{{\rm d} t}
			-
			\frac{{\rm d} \bar{g}_e(W,W_e)}{{\rm d} t}
			,v_e>_{\Upsilon}.   
		\end{aligned}
		\right. 
	\end{equation} 
	All terms in the derivatives of $\bar{g}_c(u)$ and $\bar{g}_e(u,u_e)$ w.r.t. t, except $\frac{\rm d}{{\rm d} t} J_R(\phi(u),\phi(u_e))$,  are Lipschitz continuous. $\frac{\rm d}{{\rm d} t} J_R$ can be treated as in Lemma \ref{pu1u2}, which gives us
	\begin{align*}
		\left|
		\frac{{\rm d} \bar{g}_c(u)}{{\rm d} t} -\frac{{\rm d} \bar{g}_c(W)}{{\rm d} t} 
		\right|
		&\leq 
		C_1
		\left|\eta\right|
		\left|\frac{\partial u}{\partial t}\right|
		+
		\left|\frac{\partial \bar{g}_c}{\partial u}(W)\right|
		\left|\frac{\partial \eta}{\partial t}\right|,
	\end{align*}
	\begin{align*}
		\left|
		\frac{{\rm d} \bar{g}_e(u,u_{e})}{{\rm d} t} 
		-
		\frac{{\rm d} \bar{g}_e(W,W_e)}{{\rm d} t} 
		\right|
		&\leq
		C_2\left(|\eta|+|\eta_e|+|\eta(0)|+\int_0^t|\eta(s)|{\rm d}s\right)\\
		&
		+
		C_3\left(|\eta(0)|+\int_0^t|\eta(s)|{\rm d}s\right)
		\left|
		\frac{\partial \phi}{\partial u_{e}}
		\frac{\partial u_{e}}{\partial t}
		-
		\frac{\partial \phi}{\partial u}
		\frac{\partial u}{\partial t}
		\right|\\
		&+
		C_3
		\left|\eta\right|
		\left|\frac{\partial u}{\partial t}\right|
		+
		C_3
		\left|\frac{\partial \phi}{\partial u}(W)\right|
		\left|\frac{\partial \eta}{\partial t}\right|\\
		&
		+
		C_3
		\left|\eta_e\right|
		\left|\frac{\partial u_{e}}{\partial t}\right|
		+
		C_3
		\left|\frac{\partial \phi}{\partial u_e}(W_e)\right|
		\left|\frac{\partial \eta_e}{\partial t}\right|
		\\
		&
		+
		C_4
		\left(|\eta|+|\eta_e|\right)
		\left|\frac{\partial u}{\partial t}\right|
		+
		\left|\frac{\partial \psi}{\partial u}(W,W_e)\right|
		\left|\frac{\partial \eta}{\partial t}\right|
		\\
		&
		+
		C_5
		\left(|\eta|+|\eta_e|\right)
		\left|\frac{\partial u_{e}}{\partial t}\right|
		+
		\left|\frac{\partial \psi}{\partial u_e}(W,W_e)\right|
		\left|\frac{\partial \eta_e}{\partial t}\right|,
	\end{align*}
	where $\frac{\partial \phi}{\partial u}$, $\frac{\partial \phi}{\partial u_e}$, $\frac{\partial \psi}{\partial u}$, $\frac{\partial \psi}{\partial u_e}$ are bounded from the definitions of $\phi(\cdot)$ in  Section \ref{NLP1LA} and $\phi_m(\cdot)$ in Section \ref{intro}. Similar as the estimates for $\eta$, $\eta_e$, \eqref{dtNLPJH1} can be obtained.
\end{proof}

\section{Error Estimate for the Semi-discrete Galerkin Method}\label{SD}
Without further notice, in the sequel, we use the modified system (change $f(b,u)$ to $f(\phi(b),\phi(u))$, $g_c(u)$ to $\bar{g}_c(u)$ and $g_e(u,u_e)$ to $\bar{g}_e(u,u_e)$ as in Section \ref{errorGP}) to obtain the error analysis since it has the same solution as the original system, see a similar proof in Theorem \ref{EU3}.  Letting $U(t), U_b(t) \in S_k(\Omega_c)$, $U_e(t) \in S_k(\Omega_e)$, for $t\in[0,T]$ and $v, v_b\in S_k(\Omega_c)$, $v_e \in S_k(\Omega_e)$, the semi-discrete form of the model is given as
\begin{equation}\label{bomegaU}
	\left\{
	\begin{aligned}
		&		\left(
		{\partial_t U},v
		\right)
		+
		a_c(U,v) 
		= 
		<g_c(U),v>_{\partial\Omega}
		-
		<g_e(U,U_e),v>_{\Upsilon}
		\\
		&	\qquad\qquad\qquad\qquad\qquad 		+
		(f(U_b,U),v),
		\\
		&		\left(
		{\partial_t U_b},v_b
		\right)
		+
		a_b( U_b, v_b) 
		=
		(f(U_b,U),v_b),
		\\
		&		\left(
		{\partial_t U_e},v_e
		\right)
		+
		a_e(U_e,v_e) 
		= 
		<g_e(U,U_e),v_e>_{\Upsilon}.
	\end{aligned}
	\right. 
\end{equation} 
Then from the model equations \eqref{uomega1}-\eqref{bBdyonOmega2} and the nonlinear projection \eqref{NLPJ}, we have
\begin{equation} 
	\left\{
	\begin{aligned}
		&		\left(
		{\partial_t W},v
		\right)
		+
		a_c(  W, v)  
		\\
		&\qquad\quad = 
		<g_c(W),v>_{\partial\Omega}
		-
		<g_e(W,W_e),v>_{\Upsilon}
		+
		(f(W_b,W),v)  
		\\
		&\qquad\quad\quad +\lambda (u-W,v)
		-\left(\partial_t\eta
		, v
		\right)  	
		+
		(f(b,u)-f(W_b,W),v),
		\\
		&		\left(
		{\partial_t W_b} ,v_b
		\right)
		+ 
		a_b( W_b, v_b)  \\
		&\qquad\quad = 
		(f(W_b,W),v_b)
		\label{bomegaW}
		\\
		&\qquad\quad \quad
		+(b-W_b,v_b)
		-\left(\partial_t\eta_b
		, v_b
		\right)  	
		+
		(f(b,u)-f(W_b,W),v_b),
		\\
		&		\left(
		{\partial_t W_e} ,v_e
		\right)
		+
		a_e( W_e, v_e) \\
		&\qquad\quad 		= 
		<g_e(W,W_e),v_e>_{\Upsilon}
		+\lambda (u_e-W_e,v_e)
		-\left(\partial_t\eta_e
		, v_e
		\right),  
	\end{aligned}
	\right. 
\end{equation} 
where $\eta= u-W,  \eta_b=b-W_b, \eta_e=u_e-W_e$, and $W_b(t)\in S_k(\Omega_c)$ is the Galerkin Projection of $b$ from
$({D_b} \nabla (b-W_b),\nabla v_b) +(b-W_b,v_b) = 0.$
Letting
$e=u-U,e_b=b-U_b, e_e = u_e-U_e,$
we get the following theorem:
\begin{theorem} If the exact solutions $u, b, u_e$ for \eqref{uomega1}-\eqref{bBdyonOmega2} are smooth enough, then with $P_k$ elements, $k\geq 1$, and appropriately chosen $U(0),\, U_e(0),\, U_b(0)$, we have 
	\begin{align}
		\|e\|^2_{ \infty;0}
		&+
		\|e_b\|^2_{ \infty;0}
		+
		\|e_e\|^2_{ \infty;0} \nonumber \\
		&+
		\int_0^T
		\|\nabla e(s)\|^2 
		+
		\|\nabla e_b(s)\|^2 
		+
		\|\nabla e_e(s)\|^2 
		{\rm d}s
		\nonumber 
		\\
		&\leq
		Ch^{2k}
		\left(
		\|u\|^2_{\infty;k+1}
		+
		\left\|
		{\partial_t u} 
		\right\|^2_{\infty;k+1} 	
		+
		\|b\|^2_{\infty;k+1}
		+
		\left\|
		{\partial_t b}
		\right\|^2_{\infty;k+1} \right. \nonumber \\
		&	
		+
		\left.
		\|u_e\|^2_{\infty;k+1}
		+
		\left\|
		{\partial_t u_e} 
		\right\|^2_{\infty;k+1} 	
		\right),
		\nonumber
	\end{align}
	where $h$ is the mesh size, $C$ does not depend on $h$.
\end{theorem}
\begin{proof}
	Letting 
	$\xi = W-U,\,  \xi_b = W_b-U_b,\,  \xi_e = W_e-U_e$
	and 
	$\eta= u-W,\,  \eta_b=b-W_b,\,  \eta_e=u_e-W_e$, 
	subtracting \eqref{bomegaU} from \eqref{bomegaW}, we have
	\begin{equation}\label{bomegaXi} 
		\left\{
		\begin{aligned}
			\left(
			{\partial_t \xi},v
			\right)
			+&
			a_c(\xi,v)  \\
			&=
			<g_c(W)-g_c(U),v>_{\partial\Omega}
			-
			<g_e(W,W_e)-g_e(U,U_e),v>_{\Upsilon}
			\\
			&  +
			(f(W_b,W)-f(U_b,U),v)  \\
			&  
			+\lambda (\eta,v)
			-\left(
			{\partial_t \eta} 
			, v
			\right)  	
			+
			(f(b,u)-f(W_b,W),v),
			\\
			\left(
			{\partial_t \xi_b} ,v_b
			\right)
			&+
			a_b(\xi_b,v_b) \\
			&=
			(f(W_b,W)-f(U_b,U),v_b)
			+(\eta_b,v_b)
			\\
			&  
			-\left(
			{\partial_t \eta_b} 
			, v_b
			\right)  	
			+
			(f(b,u)-f(W_b,W),v_b),
			\\
			\left(
			{\partial_t \xi_e} ,v_e
			\right)
			&+
			a_e(\xi_e,v_e) \\
			&=
			<g_e(W,W_e)-g_e(U,U_e),v_e>_{\Upsilon}
			+\lambda (\eta_e,v_e)
			-\left(
			{\partial_t \eta_e} 
			, v_e
			\right). 
		\end{aligned}
		\right. 
	\end{equation} 
	From Lemma \ref{g12u1u2} and letting $W(0)=U(0)$, we have
	\begin{align}
		|g_c(W)-g_c(U)|
		&\leq 
		K_1|\xi|,
		\\
		|g_e(W,W_e)-g_e(U,U_e)|
		&\leq 
		K_2\left(|\xi|+|\xi_e|+\int_0^t|\xi(s)|{\rm d}s\right),
		\\
		|f(W_b,W)-f(U_b,U)|
		&\leq
		K_3(|\xi_b|+|\xi_e|).
	\end{align}
	Then in \eqref{bomegaXi}, let
	$
	v = \xi,\, v_b=\xi_b,\, v_e = \xi_e
	$, we obtain \eqref{bomegaNorm1} - \eqref{bomegaNorm3}:
	\begin{align}
		\frac{1}{2}\frac{{\rm d}}{{\rm d} t}
		\|\xi\|^2 
		&+
		{D_c} 
		\|\nabla \xi\|^2 \nonumber \\
		&\leq
		C_1
		\left(
		\|\xi\|^2_{L^2(\partial\Omega)}
		+
		\|\xi\|^2_{L^2({\Upsilon})}
		+
		\|\xi_e\|^2_{L^2({\Upsilon})}
		+
		\int_0^t\|\xi(s)\|^2_{L^2({\Upsilon})}{\rm d}s
		\right)
		\nonumber\\
		&
		+
		C_1
		\left(
		\|\xi_b\|^2 
		+
		\|\xi\|^2 
		\right)
		\label{bomegaNorm1}\\
		&
		+\frac{\lambda+K_3}{2} \|\eta\|^2 
		+\frac{1}{2}
		\left\|
		{\partial_t \eta} 
		\right\|^2  	
		+
		\frac{K_3}{2} \|\eta_b\|^2 ,
		\nonumber
	\end{align}
	\begin{align}
		\frac{1}{2}\frac{{\rm d}}{{\rm d} t}
		\|\xi_b\|^2 
		+
		{D_b} 
		\|\nabla \xi_b\|^2 
		&\leq
		C_1
		\left(
		\|\xi_b\|^2 
		+
		\|\xi\|^2 
		\right) 
		\label{bomegaNorm2}
		\\
		&
		+\frac{1+K_3}{2}\|\eta_b\|^2 
		+\frac{1}{2}
		\left\|
		{\partial_t \eta_b} 
		\right\|^2  	
		+
		\frac{K_3}{2} \|\eta\|^2 ,
		\nonumber
	\end{align}
	\begin{align}
		\frac{1}{2}\frac{{\rm d}}{{\rm d} t}
		\|\xi_e\|^2 
		&+
		{D_e} 
		\|\nabla \xi_e\|^2 \nonumber \\
		&\leq 
		C_3
		\left(
		\|\xi\|^2_{L^2({\Upsilon})}
		+
		\|\xi_e\|^2_{L^2({\Upsilon})}
		+
		\int_0^t\|\xi(s)\|^2_{L^2({\Upsilon})}{\rm d}s
		+
		\|\xi_e\|^2 
		\right)	\label{bomegaNorm3}
		\\
		&
		+\frac{\lambda}{2}\|\eta_e\|^2 
		+\frac{1}{2}
		\left\|
		{\partial_t \eta_e} 
		\right\|^2 . 		\nonumber
	\end{align}
	The sum of \eqref{bomegaNorm1}, \eqref{bomegaNorm2} and \eqref{bomegaNorm3} gives the following inequality: 
	\begin{align}\label{bomegaNorm510}
		\frac{1}{2}\frac{{\rm d}}{{\rm d} t}
		(
		\|\xi\|^2 
		&+
		\|\xi_b\|^2 
		+
		\|\xi_e\|^2 
		)
		+
		{D_c}
		\|\nabla \xi\|^2 
		+
		{D_b} 
		\|\nabla \xi_b\|^2 
		+
		{D_e} 
		\|\nabla \xi_e\|^2 
		\nonumber 
		\\ 
		&\leq
		C_1
		\left(
		\|\xi\|^2_{L^2(\partial\Omega)}
		+
		\|\xi\|^2_{L^2({\Upsilon})}
		+
		\|\xi_e\|^2_{L^2({\Upsilon})}
		+
		\int_0^t\|\xi(s)\|^2_{L^2({\Upsilon})}{\rm d}s
		\right)
		\nonumber
		\\
		&
		+
		C_2
		\left(
		\|\xi_b\|^2 
		+
		\|\xi\|^2 
		+
		\|\xi_e\|^2 
		\right)
		\\
		&
		+
		C_3
		\left( 
		\|\eta\|^2 
		+
		\left\|
		{\partial_t \eta}
		\right\|^2  	
		+
		\|\eta_b\|^2 
		+
		\left\|
		{\partial_t \eta_b} 
		\right\|^2  	
		\right.
		\nonumber \\
		&+
		\left.
		\|\eta_e\|^2 
		+
		\left\|
		{\partial_t \eta_e} 
		\right\|^2 	
		\right).
		\nonumber
	\end{align}
	By Lemma \ref{trace1} and 
	$$
	\int_0^t\int_0^s\|\xi(w)\|^2_{L^2({\Upsilon})}{\rm d}w{\rm d}s
	\leq
	\int_0^t\int_0^t\|\xi(w)\|^2_{L^2({\Upsilon})}{\rm d}w{\rm d}s
	\leq
	T\int_0^t\|\xi(s)\|^2_{L^2({\Upsilon})}{\rm d}s
	$$
	from \eqref{bomegaNorm510}, we have
	\begin{align}\label{bomegaNorm511}
		\|\xi(t)\|^2 
		&+
		\|\xi_b(t)\|^2 
		+
		\|\xi_e(t)\|^2 
		+
		\int_0^t
		\|\nabla \xi(s)\|^2 
		+
		\|\nabla \xi_b(s)\|^2 
		+
		\|\nabla \xi_e(s)\|^2 
		{\rm d}s
		\nonumber 
		\\ 
		&\leq
		C_1
		\int_0^t
		\|\xi_b(s)\|^2 
		+
		\|\xi(s)\|^2 
		+
		\|\xi_e(s)\|^2 
		{\rm d}s
		\\
		&
		+
		C_2
		\int_0^t
		\|\eta\|^2 
		+
		\left\|
		{\partial_t \eta} 
		\right\|^2  	
		+
		\|\eta_b\|^2 
		+
		\left\|
		{\partial_t \eta_b} 
		\right\|^2  +
		\|\eta_e\|^2 
		+
		\left\|
		{\partial_t \eta_e} 
		\right\|^2 	
		{\rm d}s.
		\nonumber
	\end{align}
	From \eqref{bomegaNorm511}, by Gronwall's Lemma, the following estimate can be obtained:
	\begin{align}
		\|\xi(t)\|^2 
		&+
		\|\xi_b(t)\|^2 
		+
		\|\xi_e(t)\|^2 \nonumber 
		\\
		&+
		\int_0^t
		\|\nabla \xi(s)\|^2 
		+
		\|\nabla \xi_b(s)\|^2 
		+
		\|\nabla \xi_e(s)\|^2 
		{\rm d}s
		\nonumber 
		\\ 
		&\leq
		C
		\int_0^t
		\|\eta\|^2 
		+
		\left\|
		{\partial_t \eta} 
		\right\|^2  	
		+
		\|\eta_b\|^2 
		+
		\left\|
		{\partial_t \eta_b} 
		\right\|^2  	+
		\|\eta_e\|^2 
		+
		\left\|
		{\partial_t \eta_e} 
		\right\|^2 	
		{\rm d}s,
		\nonumber
	\end{align}
	then by $\xi=e-\eta,\, \xi_b=e_b-\eta_b,\, \xi_e=e_e-\eta_e,$ we have
	\begin{align}
		\|e\|^2_{ \infty;0}
		&+
		\|e_b\|^2_{ \infty;0}
		+
		\|e_e\|^2_{ \infty;0}
		+
		\int_0^T
		\|\nabla e(s)\|^2 
		+
		\|\nabla e_b(s)\|^2 
		+
		\|\nabla e_e(s)\|^2 
		{\rm d}s
		\nonumber 
		\\
		&\leq
		C
		\left(
		\|\eta\|^2_{ \infty;1}
		+
		\left\|
		{\partial_t \eta} 
		\right\|^2_{ \infty;1} 	
		+
		\|\eta_b\|^2_{ \infty;1}
		+
		\left\|
		{\partial_t \eta_b} 
		\right\|^2_{ \infty;1} 	
		\right.
		\nonumber \\
		&
		+
		\|\eta_e\|^2_{ \infty;1}
		+\left.
		\left\|
		{\partial_t \eta_e} 
		\right\|^2_{ \infty;1} 	
		\right)\\
		&+
		C\left(
		\|\eta(0)\|^2_{H^1(\Omega_c)}
		+
		\|\eta_b(0)\|^2_{H^1(\Omega_c)}
		+
		\|\eta_e(0)\|^2_{H^1(\Omega_e)}
		\right),
		\nonumber
	\end{align}
	and with Theorem \ref{p_teta},  the proof is completed. 
\end{proof}

\section{Error Estimate for the Fully Discrete Implicit-Explicit Scheme}\label{FD}
We define $d_t U^{n+1} = (U^{n+1} -U^{n} )/{\Delta t}$, where  $\Delta t = T/N$, $N$ is a positive integer, $n=0,1,2,\cdots, N-1$, $U^n \in S_k(\Omega_c)$. Also $d_t U_b^{n+1}$, $d_t U_e^{n+1}$ can be defined similarly, where $U_b^{n} \in S_k(\Omega_c)$, $U_e^{n} \in S_k(\Omega_e)$.  Then let $v, v_b\in S_k(\Omega_c)$, $v_e \in S_k(\Omega_e)$ and $t_n=n\Delta t$, we have the fully discrete form of the model:
\begin{equation} \label{FDFEM}
	\left\{
	\begin{aligned}
		\left(
		d_t{U^{n+1} } ,v
		\right)
		+
		a_c( U^{n+1},  v) 
		=&
		<g_c(U^n),v>_{\partial\Omega}
		-
		<g_e(U^n,U_e^n),v>_{\Upsilon} 
		\\
		&+
		(f(U_b^n,U^n),v)
		\\
		\left(
		d_t{U^{n+1}_b} 
		,v_b
		\right)
		+
		a_b(  U_b^{n+1},  v_b) 
		&=
		(f(U_b^n,U^n),v_b) 
		\\
		\left(
		d_t{U^{n+1}_e } 
		,v_e
		\right)
		+
		a_e(  U_e^{n+1}, v_e) 
		&=
		<g_e(U^n,U_e^n),v_e>_{\Upsilon},
	\end{aligned}
	\right. 
\end{equation} 
for the ODE part, we employ
$U_h(t,x) = \sum\limits_{i=0}^{n}\phi_i(t)U^i(x)$ for $0\leq t\leq t_n,$
where $\phi_i$ is the one-dimensional hat function and $\phi_i(t_i)=1$. 
Let $e^n= u(t_n)-U^n$, $e_b^n=b(t_n)-U_b^n,$  $e_e^n=u_e(t_n) - U_e^n$, where $u, b, u_e$ are the exact solutions for \eqref{uomega1}-\eqref{bBdyonOmega2}, then we have the following theorem: 
\begin{theorem} If the exact solutions $u, b, u_e$ are sufficiently smooth, then with $P_k$ elements, $k\geq 1$, appropriately chosen $U^0, U_b^0, U_e^0$, and sufficiently small ${\Delta t}$ which doesn't depend on the spatial mesh size, for any $n\leq N-1$, we have 
	\begin{align*}
		\|e^{n+1}\|^2 
		&+
		\|e_b^{n+1}\|^2 
		+
		\|e_e^{n+1}\|^2 
		\\
		&+
		\Delta t 
		\sum\limits_{l=1}^{n+1}
		\|\nabla e^{l}\|^2
		+
		\Delta t 
		\sum\limits_{l=1}^{n+1}
		\|\nabla e_b^{l}\|^2
		+
		\Delta t 
		\sum\limits_{l=1}^{n+1}
		\|\nabla e_e^{l}\|^2
		\leq
		C
		(
		\Delta t^2+h^{2k}
		),
		\nonumber
	\end{align*}
	where $\Delta t$ is the time-step size, $h$ is the mesh size, $C$ does not depend on $n, h$ and $\Delta t$.
\end{theorem}	
\begin{proof}	
	Let
	$u^n=u(t_n),$  $d_t u^{n+1} = ({u^{n+1}-u^n})/{\Delta t}$ and similarly we have $d_t \eta^{n+1}$, $d_t \eta_b^{n+1},$ $d_t \eta_e^{n+1}$, where $\eta^n= u(t_n)-W(t_n)$, $\eta_b^n=b(t_n)-W_b(t_n),$  $\eta_e^n=u_e(t_n)-W_e(t_n)$, $W, W_b, W_e$ are the previously defined Galerkin projections for $u, b, u_e$.
	Then we have
	\begin{equation}\label{FDFEM2} 
		\left\{
		\begin{aligned}
			\left(
			d_tW^{n+1},v
			\right)
			+
			a_c( W^{n+1}, v) 
			= &
			<g_c(W^{n}),v>_{\partial\Omega}\\
			&-
			<g_e(W^{n},W_e^{n}),v>_{\Upsilon} \\
			& 
			+
			(f(W_b^{n},W^{n}),v) 
			+
			E^{n+1}(v),
			\\
			\left(
			d_t W_b^{n+1},v_b
			\right)
			+
			a_b( W_b^{n+1}, v_b) 
			&=
			(f(W_b^{n},W^{n}),v_b)+E_b^{n+1}(v_b),
			\\
			\left(
			d_t W_e^{n+1},v_e
			\right)
			+
			a_e(  W_e^{n+1}, v_e) 
			&= 
			<g_e(W^{n},W_e^{n}),v_e>_{\Upsilon}
			+E_e^{n+1}(v_e),
		\end{aligned}
		\right. 
	\end{equation}
	where
	\begin{align*}
		E^{n+1}(v)
		&= 
		<g_c(W^{n+1})-g_c(W^{n}),v>_{\partial\Omega}\\
		&-
		<g_e(W^{n+1},W_e^{n+1})-g_e(W^{n},W_e^{n}),v>_{\Upsilon}
		\nonumber\\ 
		&+
		(f(W_b^{n+1},W^{n+1})-f(W_b^{n},W^{n}),v) 
		-(d_t\eta^{n+1}, v)  \nonumber\\ 
		&+
		\left(
		d_tu^{n+1}-
		{\partial_t u^{n+1}},v
		\right)
		+\lambda (\eta^{n+1},v)	\\
		&
		+
		(f(b^{n+1},u^{n+1})-f(W_b^{n+1},W^{n+1}),v),
		\nonumber
		\\
		E_b^{n+1}(v_b)
		&=
		(f(W_b^{n+1},W^{n+1})-f(W_b^{n},W^{n}),v_b)
		-(d_t\eta^{n+1}_b, v_b)  \nonumber\\
		&+
		\left(
		d_tb^{n+1}-
		{\partial_t b^{n+1}},v_b
		\right)
		+(\eta_b^{n+1},v_b)
		\\
		&
		+
		(f(b^{n+1},u^{n+1})-f(W_b^{n+1},W^{n+1}),v_b),
		\nonumber
		\\
		E_e^{n+1}(v_e)
		&= 
		<g_e(W^{n+1},W_e^{n+1})-g_e(W^{n},W_e^{n}),v_e>_{\Upsilon}
		-(d_t\eta_e^{n+1}, v_e)  \nonumber \\
		&+
		\left(
		d_tu_e^{n+1}-
		{\partial_t u_e^{n+1}},v_e
		\right)
		+\lambda (\eta_e^{n+1},v_e).
	\end{align*}
	Subtracting \eqref{FDFEM} from \eqref{FDFEM2}, letting
	$\xi^{n+1} = W^{n+1}-U^{n+1},$ $\xi_b^{n+1} = W_b^{n+1}-U_b^{n+1},$ $\xi_e^{n+1} = W_e^{n+1}-U_e^{n+1}$,
	and choosing 
	$v = \xi^{n+1}, v_b = \xi_b^{n+1}, v_e = \xi_e^{n+1}$, with the definition $W_h(t,x) := \sum\limits_{i=0}^{n}\phi_i(t)W(t_i,x)$ and Lemma \ref{g12u1u2},
	we have the following equations \eqref{FDFEMer1}, \eqref{FDFEMer2} and \eqref{FDFEMer3} respectively
	\begin{align}
		&\frac{\|\xi^{n+1}\|^2}{2\Delta t}
		-
		\frac{\|\xi^{n}\|^2}{2\Delta t}
		+
		{D_c} 
		\|\nabla \xi^{n+1}\|^2\nonumber \\
		&\leq 
		C_1
		\left(
		\|\xi^n\|^2_{L^2(\partial\Omega)}
		+
		\|\xi^n\|^2_{L^2({\Upsilon})}
		+
		\|\xi_e^n\|^2_{L^2({\Upsilon})}
		+
		\int_0^{t_n}\|(W_h-U_h)(s)\|^2_{L^2({\Upsilon})}{\rm d}s
		\right)
		\nonumber		\\
		&
		+
		C_1
		(
		\|\xi_b^n\|^2 
		+
		\|\xi^n\|^2 
		+
		\|\xi^{n+1}\|^2 
		)
		\label{FDFEMer1}
		\\
		&
		+
		C_1
		(\|\xi^{n+1}\|^2_{L^2(\partial\Omega)}
		+
		\|\xi^{n+1}\|^2_{L^2({\Upsilon})}
		)
		+E^{n+1}(\xi^{n+1})
		\nonumber \\
		&
		+
		C_1
		\int_0^{t_n}\|(W-W_h)(s)\|^2_{L^2({\Upsilon})}{\rm d}s,\nonumber
	\end{align}
	\begin{align}
		\frac{\|\xi_b^{n+1}\|^2}{2\Delta t}
		&-
		\frac{\|\xi_b^{n}\|^2}{2\Delta t}
		+
		{D_b} 
		\|\nabla \xi_b^{n+1}\|^2\nonumber \\
		&\leq
		C_1
		\left(
		\|\xi_b^n\|^2 
		+
		\|\xi^n\|^2 
		+
		\|\xi_b^{n+1}\|^2 
		\right)
		+E_b^{n+1}(\xi_b^{n+1}),
		\label{FDFEMer2}
	\end{align}
	\begin{eqnarray}
		\frac{\|\xi_e^{n+1}\|^2}{2\Delta t}
		&-&
		\frac{\|\xi_e^{n}\|^2}{2\Delta t}
		+
		{D_e} 
		\|\nabla \xi_e^{n+1}\|^2\nonumber \\
		&\leq& 
		C_3
		\left(
		\|\xi^n\|^2_{L^2({\Upsilon})}
		+
		\|\xi_e^n\|^2_{L^2({\Upsilon})}
		+
		\int_0^{t_n}\|(W_h-U_h)(s)\|^2_{L^2({\Upsilon})}{\rm d}s
		\right)
		\nonumber
		\\
		&&
		+
		C_3
		\|\xi_e^{n+1}\|^2 
		+E_e^{n+1}(\xi_e^{n+1})
		\label{FDFEMer3} \\
		&&
		+
		C_3
		\int_0^{t_n}\|(W-W_h)(s)\|^2_{L^2({\Upsilon})}{\rm d}s.\nonumber
	\end{eqnarray}
	Adding the equations \eqref{FDFEMer1}, \eqref{FDFEMer2} and \eqref{FDFEMer3} as $n$ varies leads to equations \eqref{FDFEMer1S}, \eqref{FDFEMer2S} and \eqref{FDFEMer3S} correspondingly 
	\begin{align}
		&\frac{\|\xi^{n+1}\|^2}{2\Delta t}
		-
		\frac{\|\xi^{0}\|^2}{2\Delta t}
		+
		{D_c} 
		\sum\limits_{l=1}^{n+1}
		\|\nabla \xi^{l}\|^2\nonumber \\
		&\leq 
		C_1
		\sum\limits_{l=0}^{n}\left(
		\|\xi^l\|^2_{L^2(\partial\Omega)}
		+
		\|\xi^l\|^2_{L^2({\Upsilon})}
		+
		\|\xi_e^l\|^2_{L^2({\Upsilon})}
		+
		\int_0^{t_l}\|(W_h-U_h)(s)\|^2_{L^2({\Upsilon})}{\rm d}s
		\right)
		\nonumber\\
		&
		+
		C_1
		\sum\limits_{l=0}^{n}(
		\|\xi_b^l\|^2 
		+
		\|\xi^l\|^2 
		+
		\|\xi^{l+1}\|^2 
		)
		\label{FDFEMer1S}
		\\
		&
		+
		C_1
		\sum\limits_{l=0}^{n}
		(\|\xi^{l+1}\|^2_{L^2(\partial\Omega)}
		+
		\|\xi^{l+1}\|^2_{L^2({\Upsilon})}
		)
		+
		\sum\limits_{l=0}^{n}
		E^{l+1}(\xi^{l+1})
		\nonumber \\
		&
		+
		C_1
		\sum\limits_{l=0}^{n}\int_0^{t_l}\|(W-W_h)(s)\|^2_{L^2({\Upsilon})}{\rm d}s,	\nonumber
	\end{align}
	\begin{align}
		\frac{\|\xi_b^{n+1}\|^2}{2\Delta t}
		&-
		\frac{\|\xi_b^{0}\|^2}{2\Delta t}
		+
		{D_b} 
		\sum\limits_{l=1}^{n+1}
		\|\nabla \xi_b^{l}\|^2\nonumber \\
		&\leq
		C_1
		\sum\limits_{l=0}^{n}
		\left(
		\|\xi_b^l\|^2 
		+
		\|\xi^l\|^2 
		+
		\|\xi_b^{l+1}\|^2 
		\right)
		+
		\sum\limits_{l=0}^{n}
		E_b^{l+1}(\xi_b^{l+1}), 
		\label{FDFEMer2S}
	\end{align}
	\begin{align}
		&\frac{\|\xi_e^{n+1}\|^2}{2\Delta t}
		-
		\frac{\|\xi_e^{0}\|^2}{2\Delta t}
		+
		{D_e} 
		\sum\limits_{l=1}^{n+1}\|\nabla \xi_e^{l}\|^2\nonumber \\
		&\leq 
		C_3
		\sum\limits_{l=0}^{n}
		\left(
		\|\xi^l\|^2_{L^2({\Upsilon})}
		+
		\|\xi_e^l\|^2_{L^2({\Upsilon})}
		+
		\int_0^{t_l}\|(W_h-U_h)(s)\|^2_{L^2({\Upsilon})}{\rm d}s
		\right)
		\label{FDFEMer3S}
		\\
		&
		+
		C_3
		\sum\limits_{l=0}^{n}
		\|\xi_e^{l+1}\|^2 
		+
		\sum\limits_{l=0}^{n}
		E_e^{l+1}(\xi_e^{l+1})
		+
		C_3
		\sum\limits_{l=0}^{n}
		\int_0^{t_l}\|(W-W_h)(s)\|^2_{L^2({\Upsilon})}{\rm d}s.\nonumber 
	\end{align}
	In the right hand side of \eqref{FDFEMer1S}, the following estimates \eqref{estFD1} to \eqref{estFD1ev} can be obtained. Notice that each $\phi_i(s), i=0,\cdots,l$, has a compact support and the product
	$\phi_i(s) \phi_j(s) = 0$, if $|j-i|>1$, so that we have
	\begin{align}\label{estFD1}
		\begin{split}
			\int_0^{t_l}\|(W_h-U_h)(s)\|^2_{L^2({\Upsilon})}
			{\rm d}s
			&=
			\int_0^{t_l}\|\sum\limits_{i=0}^{l}\phi_i(s)\xi^i\|^2_{L^2({\Upsilon})}
			{\rm d}s
			\\
			&\leq
			\int_0^{t_l}3\sum\limits_{i=0}^{l}|\phi_i(s)|^2
			\|\xi^i\|^2_{L^2({\Upsilon})}{\rm d}s\\
			&\leq
			2{\Delta t}\sum\limits_{i=0}^{l}
			\|\xi^i\|^2_{L^2({\Upsilon})},
		\end{split}
	\end{align}
	where  $\int_0^{t_l}3|\phi_i(s)|^2{\rm d}s\leq 2{\Delta t}$. Then from \eqref{estFD1}, with the Trace Theorem, we get
	\begin{align}\label{estFD2}
		\begin{split}
			\sum\limits_{l=0}^{n}
			\int_0^{t_l}
			\|(W_h-U_h)(s)\|^2_{L^2({\Upsilon})}
			{\rm d}s
			&\leq
			\sum\limits_{l=0}^{n}
			2{\Delta t} 
			\sum\limits_{i=0}^{n}
			\|\xi^i\|^2_{L^2({\Upsilon})}
			\\
			&\leq
			2T 
			\sum\limits_{i=0}^{n}
			\|\xi^i\|^2_{L^2({\Upsilon})}
			\\
			&\leq
			C \sum\limits_{l=0}^{n}
			\left(
			\|\xi^l\|^2 
			+
			\|\nabla \xi^l\|^2 
			\right).
		\end{split}
	\end{align}
	By Lemma \ref{trace1}, we have
	\begin{align}\label{xin+1}
		\begin{split}
			\Delta t C_1
			\sum\limits_{l=0}^{n}
			(\|\xi^{l+1}\|^2_{L^2(\partial\Omega)}
			+
			\|\xi^{l+1}\|^2_{L^2({\Upsilon})}
			)  
			\leq&
			\frac{\Delta t D_c}{8}
			\sum\limits_{l=1}^{n+1}
			\|\nabla \xi^{l}\|^2 \\
			&+
			\Delta t C
			\|\xi^{n+1}\|^2 
			+
			\Delta t C
			\sum\limits_{l=1}^{n}
			\|\xi^{l}\|^2 ,
		\end{split}
	\end{align}
	the first and second terms in the right hand side of \eqref{xin+1} can be canceled in \eqref{FDFEMer1S} if  $\Delta t$ is sufficiently small, however, this small $\Delta t$ doesn't depend on the spatial mesh size. Then  with Lemma \ref{g12u1u2}, Lemma \ref{trace1}, by the estimates of $\partial_t\eta$, $\partial_t\eta_e$ in Theorem \ref{p_teta} and the estimates for $\partial_t\eta_b$, which is easier to obtain from the definition of $W_b$ in Section \ref{SD}, we know $\partial_t W$, $\partial_t W_e$ and $\partial_t W_b$ are bounded with $H^1$ norm, so that   
	\begin{align}\label{estFD1ev}
		\begin{split}
			\sum\limits_{l=0}^{n}
			E^{l+1}(\xi^{l+1})
			\leq&
			C\Delta t +C\frac{h^{2k}}{\Delta t} 
			+
			\frac{D_c}{8}
			\sum\limits_{l=1}^{n+1}
			\|\nabla \xi^{l}\|^2  \\
			&+
			C
			\|\xi^{n+1}\|^2 
			+
			C
			\sum\limits_{l=1}^{n}
			\|\xi^{l}\|^2 ,
		\end{split}
	\end{align}
	where $C$ is a positive constant and doesn't depend on $n$.
	Other terms on the right-hand side of \eqref{FDFEMer1S} can be treated similarly.\\
	So that from \eqref{FDFEMer1S} and the estimates \eqref{estFD2} to \eqref{estFD1ev}, we have \eqref{FDFEMer1S1}
	\begin{align}
		\frac{1}{2}
		\|\xi^{n+1}\|^2
		&+
		\frac{3D_c}{4}\Delta t 
		\sum\limits_{l=1}^{n+1}
		\|\nabla \xi^{l}\|^2\nonumber \\
		&\leq 
		C_{c1}\Delta t
		\sum\limits_{l=0}^{n} 
		\|\xi^l\|^2_{L^2(\Omega)}
		+
		C_{c2}\Delta t
		\sum\limits_{l=0}^{n}
		\|\xi_b^l\|^2 
		+
		C_{c3}\Delta t
		\sum\limits_{l=0}^{n}
		\|\xi_e^l\|^2  
		\label{FDFEMer1S1}\\
		&
		+
		\frac{D_e}{8}\Delta t
		\sum\limits_{l=0}^{n}
		\|\nabla\xi_b^l\|^2 
		+
		C_{c4}
		\|\xi^{0}\|_{H^1(\Omega_c)}^2
		+
		C_c(\Delta t^2+h^{2k}).\nonumber
	\end{align}
	From \eqref{FDFEMer2S}, with similar estimates as \eqref{FDFEMer1S}, we have \eqref{FDFEMer2S1}
	\begin{align}
		&\frac{1}{2}
		\|\xi_b^{n+1}\|^2
		+
		{D_b} \Delta t
		\sum\limits_{l=1}^{n+1}
		\|\nabla \xi_b^{l}\|^2	\label{FDFEMer2S1}\\
		&\leq
		C_{b1}\Delta t
		\sum\limits_{l=0}^{n}
		\|\xi_b^l\|^2 
		+
		C_{b2}\Delta t
		\sum\limits_{l=0}^{n}
		\|\xi^l\|^2 
		+
		C_{b3}
		\|\xi_b^{0}\|^2_{H^1(\Omega_c)}
		+
		C_b(\Delta t^2+h^{2k}).\nonumber 	
	\end{align}
	From \eqref{FDFEMer3S}, with similar estimates as \eqref{FDFEMer1S}, we have \eqref{FDFEMer3S1}
	\begin{equation}\label{FDFEMer3S1}
		\begin{aligned}
		\frac{1}{2}\|\xi_e^{n+1}\|^2
		&+
		\frac{3D_e}{4}\Delta t 
		\sum\limits_{l=1}^{n+1}\|\nabla \xi_e^{l}\|^2  \\
		&\leq
		C_{e1}\Delta t
		\sum\limits_{l=0}^{n}
		\|\xi^l\|^2 
		+
		C_{e2}\Delta t
		\sum\limits_{l=0}^{n}
		\|\xi_e^l\|^2 
		+
		\frac{D_c}{8}\Delta t
		\sum\limits_{l=0}^{n}
		\|\xi^l\|^2 \\
		&
		+
		C_{e3}
		\|\xi_e^{0}\|^2_{H^1(\Omega_e)}
		+
		C_e(\Delta t^2+h^{2k}). 
		\end{aligned}
	\end{equation}
	$C_c, C_b, C_e$ in \eqref{FDFEMer1S1},\eqref{FDFEMer2S1} and \eqref{FDFEMer3S1} do not depend on $n$.
	
	Summing \eqref{FDFEMer1S1}, \eqref{FDFEMer2S1} and \eqref{FDFEMer3S1}, we have
	\begin{align*}
		\|\xi^{n+1}\|^2
		&+
		\|\xi_b^{n+1}\|^2
		+
		\|\xi_e^{n+1}\|^2
		+
		\Delta t 
		\sum\limits_{l=1}^{n+1}
		\|\nabla \xi^{l}\|^2
		+
		\Delta t 
		\sum\limits_{l=1}^{n+1}
		\|\nabla \xi_b^{l}\|^2
		+
		\Delta t 
		\sum\limits_{l=1}^{n+1}
		\|\nabla \xi_e^{l}\|^2
		\nonumber
		\\
		& \leq
		C_1\Delta t
		\sum\limits_{l=0}^{n} 
		\left(
		\|\xi^l\|^2
		+
		\|\xi_b^l\|^2 
		+
		\|\xi_e^l\|^2 
		\right)\\
		&\quad +
		C_{2}
		(
		\|\xi^{0}\|^2_{H^1(\Omega_c)}
		+
		\|\xi_b^{0}\|^2_{H^1(\Omega_c)}
		+
		\|\xi_e^{0}\|^2_{H^1(\Omega_e)}
		+
		\Delta t^2+h^{2k}
		).
	\end{align*}
	If $\Delta t$ is small enough, by discrete Gronwall's Inequality,  the estimate below follows
	\begin{align*}
		\|\xi^{n+1}\|^2
		&+
		\|\xi_b^{n+1}\|^2
		+
		\|\xi_e^{n+1}\|^2\\
		&+
		\Delta t 
		\sum\limits_{l=1}^{n+1}
		\|\nabla \xi^{l}\|^2
		+
		\Delta t 
		\sum\limits_{l=1}^{n+1}
		\|\nabla \xi_b^{l}\|^2
		+
		\Delta t 
		\sum\limits_{l=1}^{n+1}
		\|\nabla \xi_e^{l}\|^2
		\leq
		C
		(
		\Delta t^2+h^{2k}
		),
	\end{align*}
	where $C$ does not depend on $n, \Delta t$ and $h$. 	Then by $e^n =\xi^n+\eta^n,\, e_b^n = \xi_b^n+\eta_b^n,\, e_e^n=\xi_e^n+\eta_e^n,$ and Theorem \ref{p_teta}, we complete the proof.
\end{proof}

\section{Numerical Experiments}\label{Num}
In this section, we illustrate the convergence theorem for the fully discrete scheme \eqref{FDFEM} using Examples 1 and 2 below, and then apply the methodology to show the existence of $Ca^{2+}$ wave propagation numerically in Example 3 and 4. The ODE system \eqref{matrixf1} plays the key role for calcium wave initiation and propagation, {\color{black} which is solved by backward Euler's method. The numerical schemes in this section are implemented in FreeFem++, see \cite{Hecht12}.} All presented examples are in 2D, but theorems and simulations are also valid in 3D.

{\bf Example }1. In this problem we cosider two coupled PDEs with the unknowns $u$ and $u_e$:
\begin{equation}
	\left\{
	\begin{aligned}
		&{\partial_t u} -\Delta u = 
		f_1(x,y,t) 
		\quad \text{on } \Omega_c\times(0,T]\\
		& {\partial_t u_e} -\Delta u_e = f_2(x,y,t)  \quad \text{on } \Omega_e\times(0,T]
	\end{aligned}
	\right. 
\end{equation}
where $T=1.3$ and the boundary conditions are: $\partial_n u 
= g_1(x,y,t) \text{ on } \partial\Omega\times(0,T]$ and
\begin{align} 
	\partial_n u 
	&= 
	P(t,u)(u_e-u)
	+ g_2(x,y,t) \ \text{ on } \Upsilon\times(0,T]\\
	\partial_n u_e
	&= 
	P(t,u)(u-u_e)
	+ g_3(x,y,t) \ \text{ on } \Upsilon\times(0,T] 
\end{align}
Here, {\color{black} let $(n_x,n_y)$ be the unit outer normal vector on $\partial\Omega_c$ for $g_1, g_2$;
on $\Upsilon$ for $g_3$, then the exact solution $u, u_e$ and corresponding functions are listed below:
\begin{center}
\begin{tabular}{ |c|c| } 
 \hline
 $u$   & $e^{\frac{x^2+y^2+4t}{4}}/10$ \\ \hline
 $u_e$ & $e^{x+y}(\sin t +2)/8$\\ \hline
 $f_1$ & $-e^{x^2/4+y^2/4+t}(x^2+y^2)/40$\\ \hline
 $f_2$ & $e^{x+y}(\cos(t)-2(\sin(t)+2))/8$\\ \hline
 $g_1$  & $n_x e^{\frac{x^2+y^2+4t}{4}}x/20+n_y e^{\frac{x^2+y^2+4t}{4}}y/20$ \\ \hline
 $g_2$  & $n_x e^{\frac{x^2+y^2+4t}{4}}x/20+n_y e^{\frac{x^2+y^2+4t}{4}}y/20-P(t,u)(u_e-u)$ \\ \hline
 $g_3$  & $n_x e^{x+y}(\sin t +2)/8+n_y e^{x+y}(\sin t +2)/8+P(t,u)(u_e-u)$ \\ \hline
\end{tabular}
\end{center}
}
The coefficients in ODE \eqref{matrixf1} are taken from \cite{Keizer1996}:
\begin{eqnarray}\label{odeparas}
	k_a^- = 28.8,\
	k_a^+ = 1500,\
	k_b^- = 385.9,\
	k_b^+ = 1500,\
	k_c^- = 0.1,\
	k_c^+ = 1.75
\end{eqnarray}
The initial conditions for \eqref{matrixf1} are chosen as:
$
c_1(0) = 0.5,\ o(0) = 0,\ c_2(0) = 0.5.
$

{\bf Example }2. In this example we consider three coupled PDEs with unknowns $u$, $b$ and $u_e$: 
\begin{equation}
	\left\{
	\begin{aligned}
		&{\partial_t u}-\Delta u = 
		f_1(x,y,t) -b u 
		\quad \text{on } \Omega_c\times(0,T] \\
		&{\partial_t b}-\Delta b = 
		f_2(x,y,t) -b u 
		\quad \text{on } \Omega_c\times(0,T] \\
		& {\partial_t u_e} -\Delta u_e = f_3(x,y,t)  \quad \text{on } \Omega_e\times(0,T] 
	\end{aligned}
	\right. 
\end{equation}
where $T=1.3$, and the boundary conditions are:
$\partial_n u = g_1(x,y,t)$ on $\partial\Omega\times(0,T]$, 
$\partial_n b = g_2(x,y,t)$  on  $\partial\Omega_c\times(0,T]$
and 
\begin{align} 
	\partial_n u
	&= 
	P(t,u)(u_e-u)-\frac{u}{(1+u)u_e} + g_3(x,y,t)\
	\text{ on } \Upsilon\times(0,T]\\
	\partial_n u_e
	&= 
	P(t,u)(u-u_e)+\frac{u}{(1+u)u_e} + g_4(x,y,t)\
	\text{ on } \Upsilon\times(0,T]
\end{align}
Here, {\color{black} let $(n_x,n_y)$ be the unit outer normal vector on $\partial\Omega_c$ for $g_1, g_2, g_3$;
on $\Upsilon$ for $g_4$, then the exact solution $u, b, u_e$ and corresponding functions are listed below:
\begin{center}
\begin{tabular}{ |c|c| } 
 \hline
 $b$   & $e^{xyt/16}$ \\ \hline
 $u$   & $e^{\frac{x^2+y^2+4t}{4}}/10$ \\ \hline
 $u_e$ & $e^{x+y}(\sin t +2)/16+1$\\ \hline
 $f_1$ & $-e^{x^2/4+y^2/4+t}(x^2+y^2)/40+bu$\\ \hline
 $f_2$ & $-e^{xyt/16}( -xy+t^2(x^2+y^2)/16 )/16+bu$\\ \hline
 $f_3$ &  $e^{x+y}(\cos(t)-2(\sin(t)+2))/8$\\ \hline
 $g_1$  & $n_x e^{\frac{x^2+y^2+4t}{4}}x/20+n_y e^{\frac{x^2+y^2+4t}{4}}y/20$ \\ \hline
 $g_2$  & $n_x e^{xyt/16}ty/16+n_y e^{xyt/16}tx/16$ \\ \hline
 $g_3$  & $n_x e^{\frac{x^2+y^2+4t}{4}}x/20+n_y e^{\frac{x^2+y^2+4t}{4}}y/20-P(t,u)(u_e-u)+\frac{u}{(1+u)u_e}$  \\ \hline
 $g_4$  & $n_x e^{x+y}(\sin t +2)/8+n_y e^{x+y}(\sin t +2)/8+P(t,u)(u_e-u)-\frac{u}{(1+u)u_e}$ \\ \hline
\end{tabular}
\end{center}
 }   
The ODE system and its initial conditions are the same as in Example 1.
Example 1 and 2 share the same space-time meshes (see Figure \ref{fig:2} for examples of spatial meshes). The spatial mesh sizes are: $\pi/8, \pi/16, \pi/32, \pi/64, \pi/128$. Letting $\Delta t$ be the time step size and $h$ the spatial mesh size, we have the relation between the two defined as $\Delta t = C h^2 $, where $C={32T}/(5\pi^2)$, $T=1.3$. Both examples have exact solutions. 
We define the error in the $L^2$ norm for $u$ as:
$\|u-u_h\| = \sum_{i=1}^{N}\|u(t_i)-u_h(t_i)\|/N$ 
and the error in $H^1$ semi-norm as:
$\|\nabla u-\nabla u_h\| = \sum_{i=1}^{N}\|\nabla u(t_i)-\nabla u_h(t_i)\|/N$ where $N=T/\Delta t$, $t_i = i\Delta t$, $u$ is the exact solution, $u_h$ is the numerical solution. Then let $b_h$, $u_e^h$ be the numerical solutions, the errors are defined similarly. In Figure \ref{tab3}, we show the 
convergence rates for $P_1$ elements in space, which are optimal.

{\bf Example }3. In this example we present a minimal system that produces $Ca^{2+}$ waves. Units were adjusted so that $t$ has unit $s$, $u, u_e$ have unit $\upmu$M:
\begin{equation}
	\left\{
	\begin{aligned}
		&{\partial_t u} -\Delta u = 0 
		\quad \text{on } \Omega_c\times(0,T]\\
		& {\partial_t u_e} -\Delta u_e = 0  \quad \text{on } \Omega_e\times(0,T]
	\end{aligned}
	\right. 
\end{equation}
where $T=12$ and the boundary conditions are: 
\begin{align}
	\partial_n u 
	&= C_3(1000-u)-C_2\frac{u}{1+u}-C_1\frac{u^2}{1+u^2}
	+f(x,y,t) \ {\text{ on } \partial\Omega}\times(0,T] \label{ueBdyonGamma441}  \\
	\partial_n u  
	&= 
	C_{1}^e P(t,u)(u_e-u)-C_2^e\frac{u}{(2+u)u_e}+C_3^e(u_e-u) 
	\ \text{ on } \Upsilon\times(0,T]\\
	\partial_n u_e
	&= 
	C_{1}^e P(t,u)(u-u_e)+C_2^e\frac{u}{(2+u)u_e}-C_3^e(u_e-u) 
	\ \text{ on } \label{ueBdyonGamma443} \Upsilon\times(0,T]
\end{align}
The initial conditions are $u(x,y,0) = 0.05$, $u_e(x,y,0)=180$, the ODE system is the same as in Example 1, and the initial conditions of the ODE are  $c_1(0) = 0.798$, $o(0) = 0,$ $c_2(0) = 0.202.$ 
In \eqref{ueBdyonGamma441}, the value $1,000$ is the extracellular Ca$^{2+}$ concentration, and $f$ is a calcium influx function: 
$f(x,y,t) = 3$ if $0.05\leq t\leq 0.65$ and $y-x\geq 2.5$; $f(x,y,t) = 0$ elsewhere. The coefficients in \eqref{ueBdyonGamma441} to \eqref{ueBdyonGamma443} are: $C_{1}^e = 0.17,$ $C_3^e = 1/150,$ $C_2^e = 8853.54,$ $C_3 = 1/540000,$ $C_1=C_2=19954C_3.$ 

Example 3 is constructed to show the initiation and propagation of a calcium wave which plays a critical role in neuronal signal processing, see Figure \ref{fig:ca2+wave}. For the computation, we use a similar geometry as in Examples 1 and 2. The radii of the two circles, with center (0,0), are 1 and 2, but with different meshes and elements, where $T=12$, $\Delta t=0.00375,$ and the spatial mesh size is $h=\pi/24$ with $P_3$ elements in space. 
With the help of Figure \ref{fig:ca2+wave}, a calcium wave can be described as follows. 
An extracellular stimulus produces Ca$^{2+}$ influx across the outer interface (the plasma membrane) raising the calcium concentration in parts of cytosol ($\Omega_c$) and ER ($\Omega_e$) see Figures \ref{fig:a}, \ref{fig:b}. In this example, the influx goes from $0.05s$ to $0.65s$. Then, an increased concentration activates the release of Ca$^{2+}$ from the ER at $t=0.72s$, see Figure \ref{fig:c}, which in turn generates the calcium spike and thus mediates global activation of the cell, see Figures \ref{fig:d} to \ref{fig:g}. The calcium concentration reaches its peak around $3.12s$ in Figure \ref{fig:h}, from the color bar \ref{fig:bar1}, $u$ varies from $0.05$ $\upmu$M to value greater than $1.5$ $\upmu$M. Meanwhile, $u_e$ decreases to the value around $176$ $\upmu$M. After reaching the peak, $u$ decreases and $u_e$ increases to the initial state, see \ref{fig:i} to \ref{fig:l}. The term \eqref{J_R} is essential for generating calcium wave, without the ODE system, there is no calcium wave.
{\color{black} Figure \ref{fig:po3} shows the open probability function $P(t,u)$ in equation \eqref{ptu} for RyR channels on the ER membrane. It ranges from 0 to 0.81.}
Instability of the scheme \eqref{FDFEM} can be observed with time step size larger than $0.00375$. 

\begin{figure}[ht!] 
	\centering
	\subfigure[][]{%
		\includegraphics[width=0.35\linewidth]{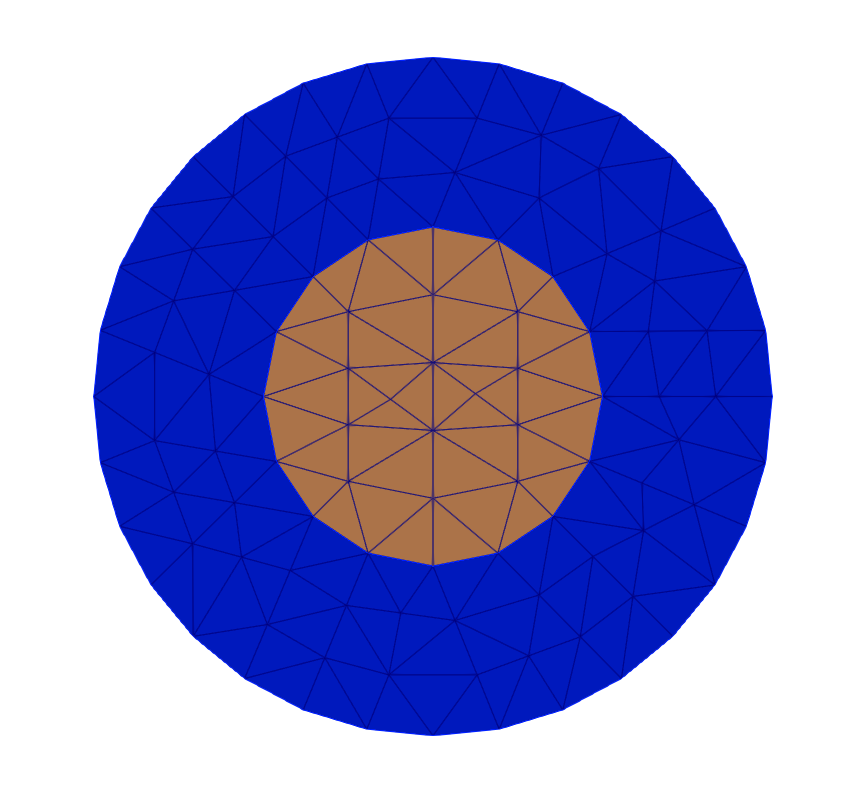}}%
	\hspace{4pt}%
	\subfigure[][]{%
		\includegraphics[width=0.35\linewidth]{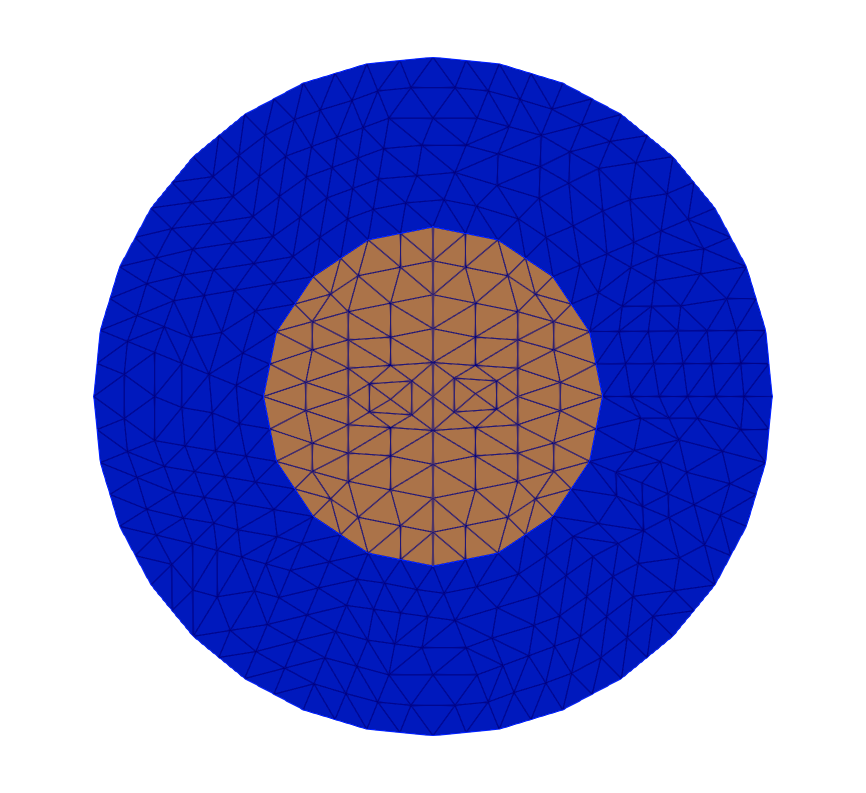}} 
	\caption[figures for meshes]{Meshes in which the black region is $\Omega_c$, the brown region is $\Omega_e$: (a) coarse with mesh size $\pi/8$; (b) refined mesh with mesh size $\pi/16$. The radius of the larger circle is 2, radius of the smaller one is 1.}
	\label{fig:2}%
\end{figure}

\begin{figure}[ht!]\label{tab3}
	\centering
	\subfigure[Convergence rates for Example 1]{%
		\includegraphics[width=0.48\linewidth]{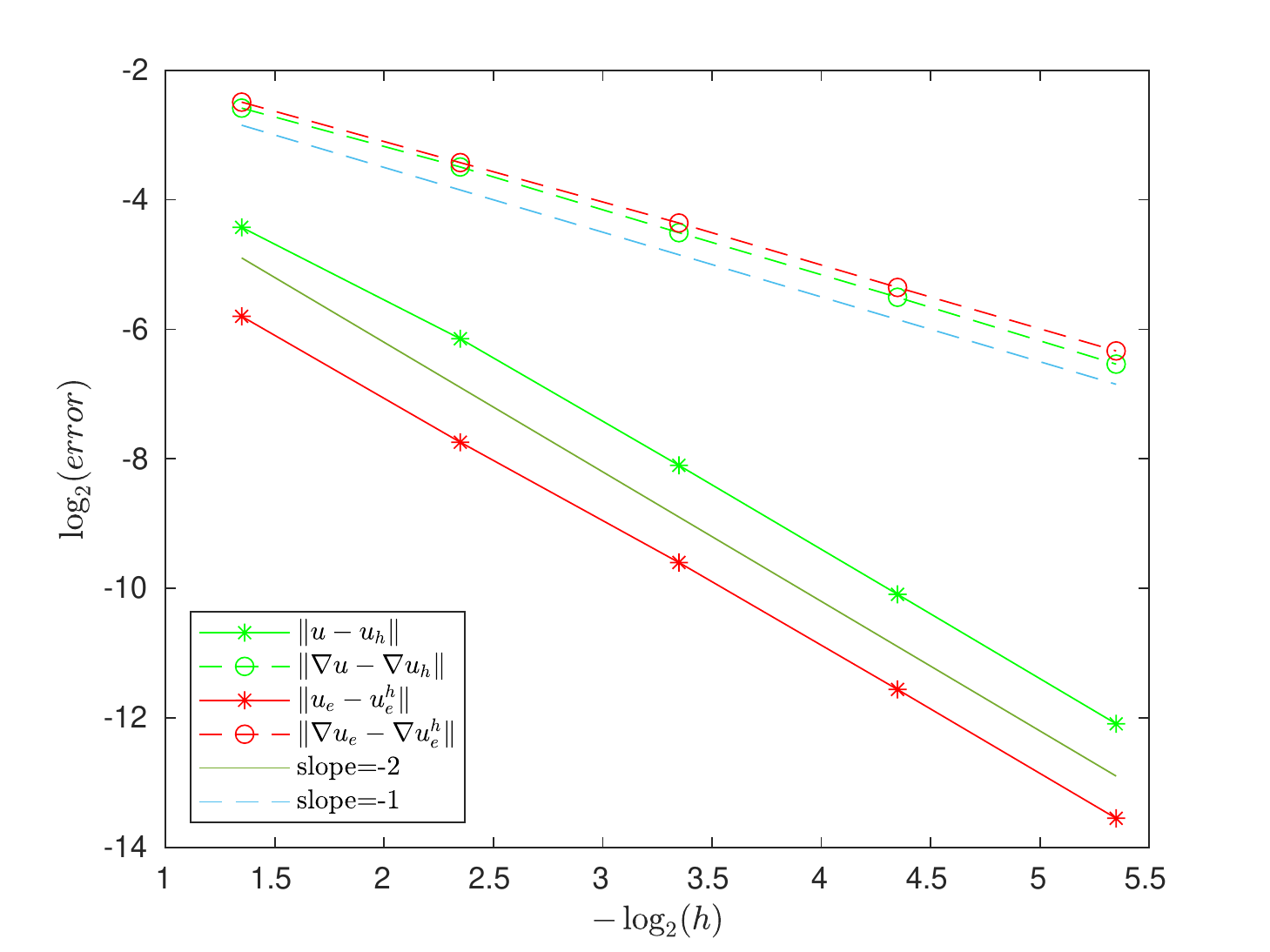}}%
	\hspace{4pt}%
	\subfigure[Convergence rates for Example 2]{%
		\includegraphics[width=0.48\linewidth]{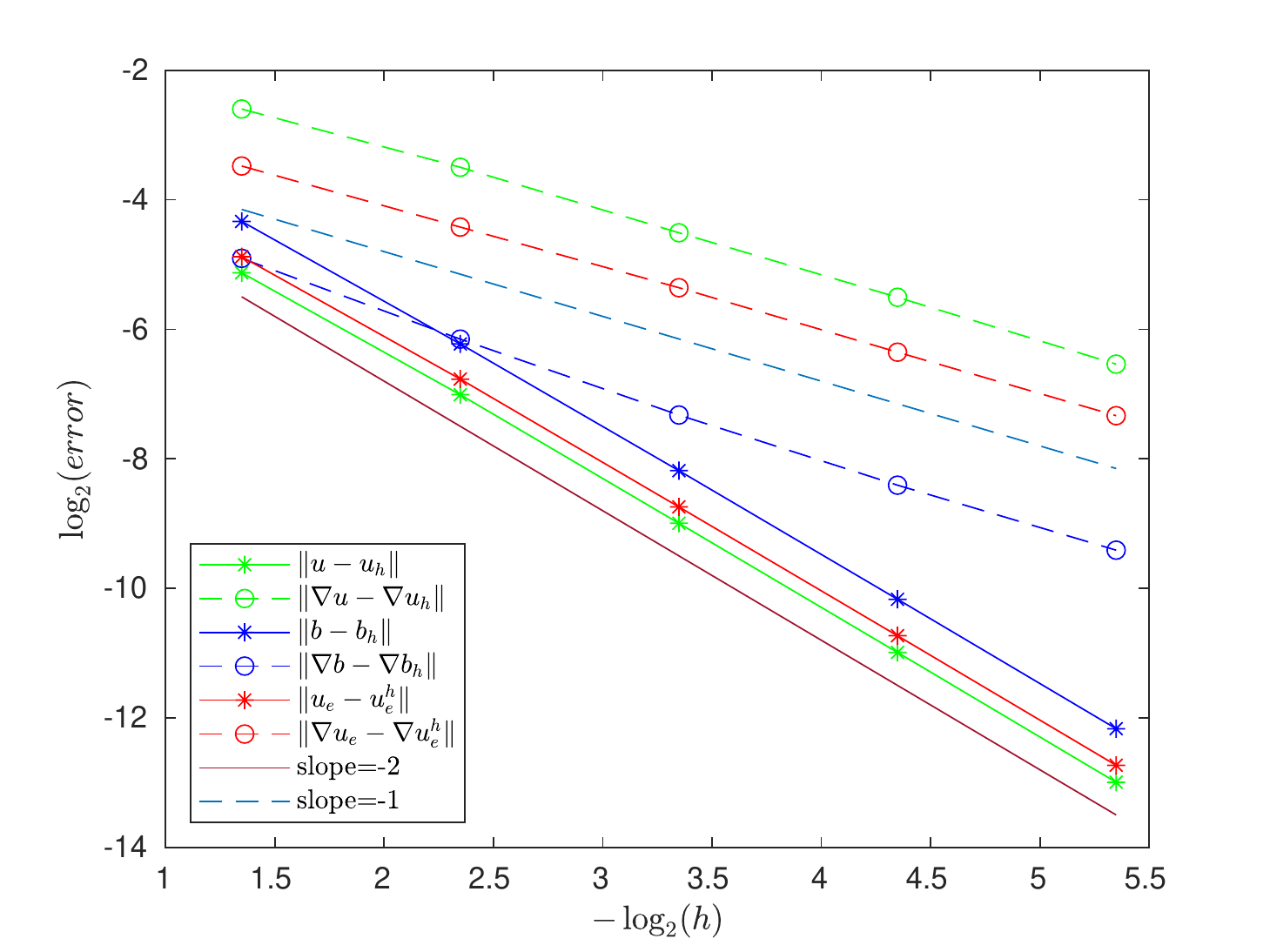}} 
	\caption[figures for meshes]{Both examples employ piecewise continuous $P_1$ element in space, for all solutions the errors with $L^2$ norm have rate 2, the errors with $H^1$ semi-norm have rate 1.}
	\label{fig:3}%
\end{figure}

\begin{figure}[ht!]
	\centering
	\subfigure[$t=0.12$]{\label{fig:a}
		\includegraphics[width=0.225\textwidth]{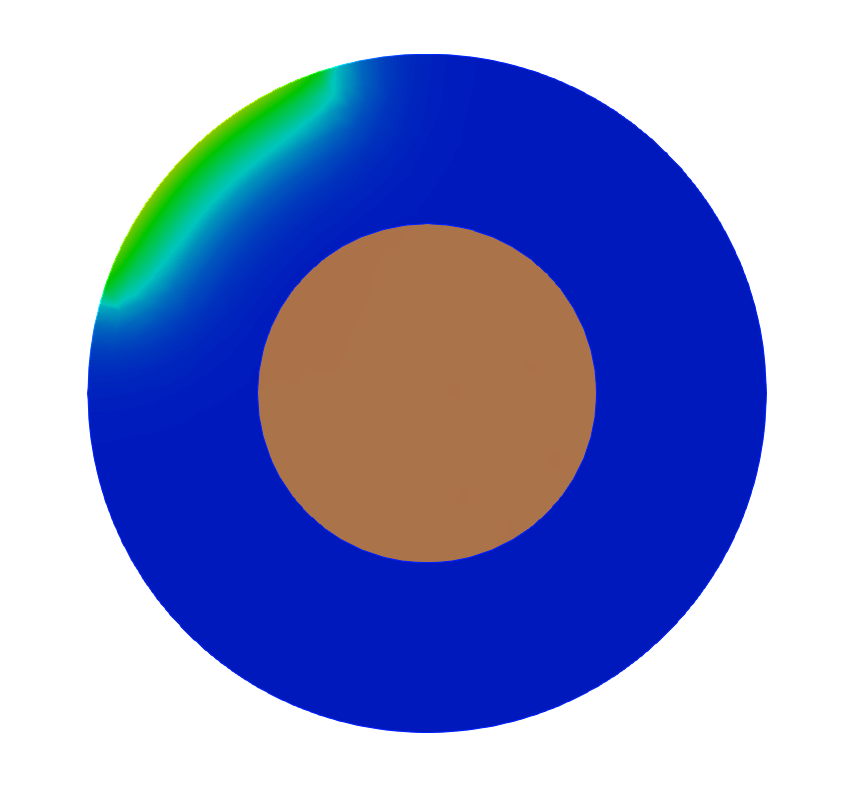}}
	\subfigure[$t=0.6$]{\label{fig:b}
		\includegraphics[width=0.225\textwidth]{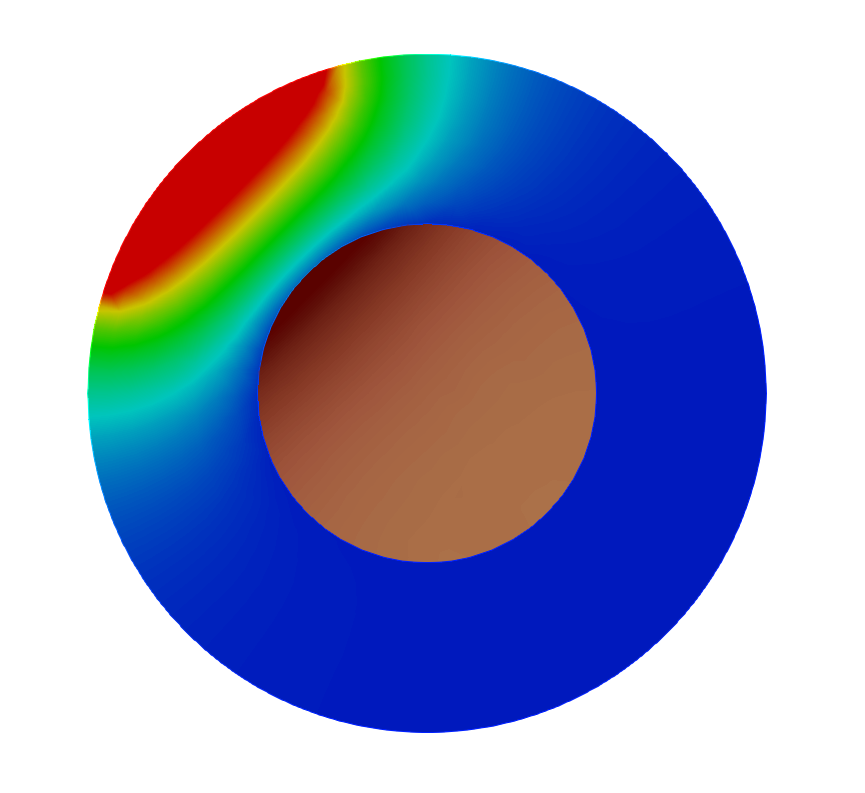}}
	\subfigure[$t=0.72$]{\label{fig:c}
		\includegraphics[width=0.225\textwidth]{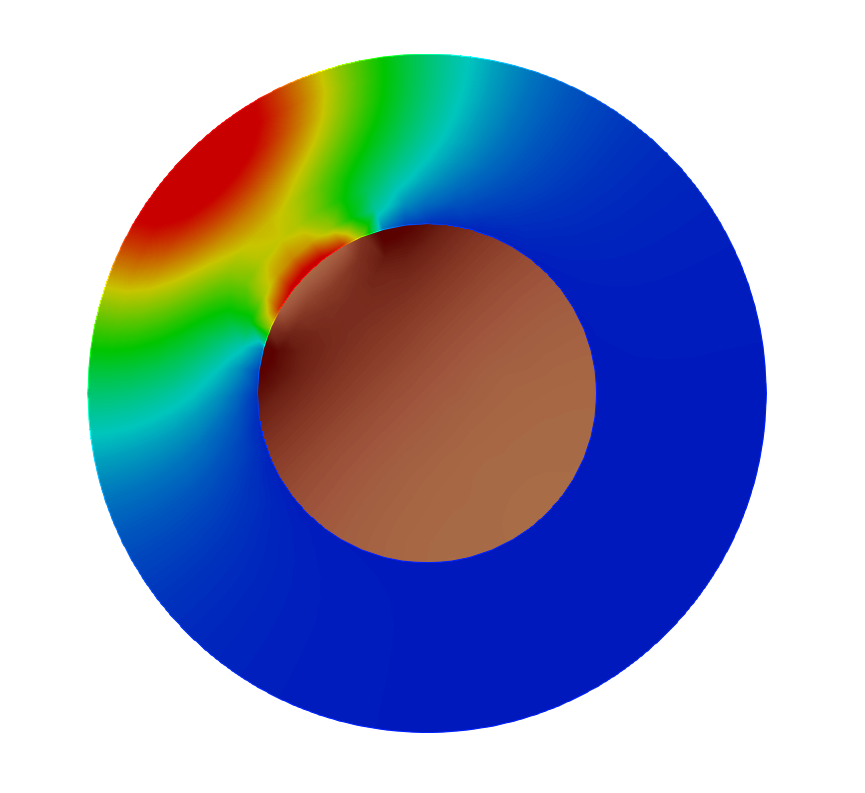}}
	\subfigure[$t=0.84$]{\label{fig:d}
		\includegraphics[width=0.225\textwidth]{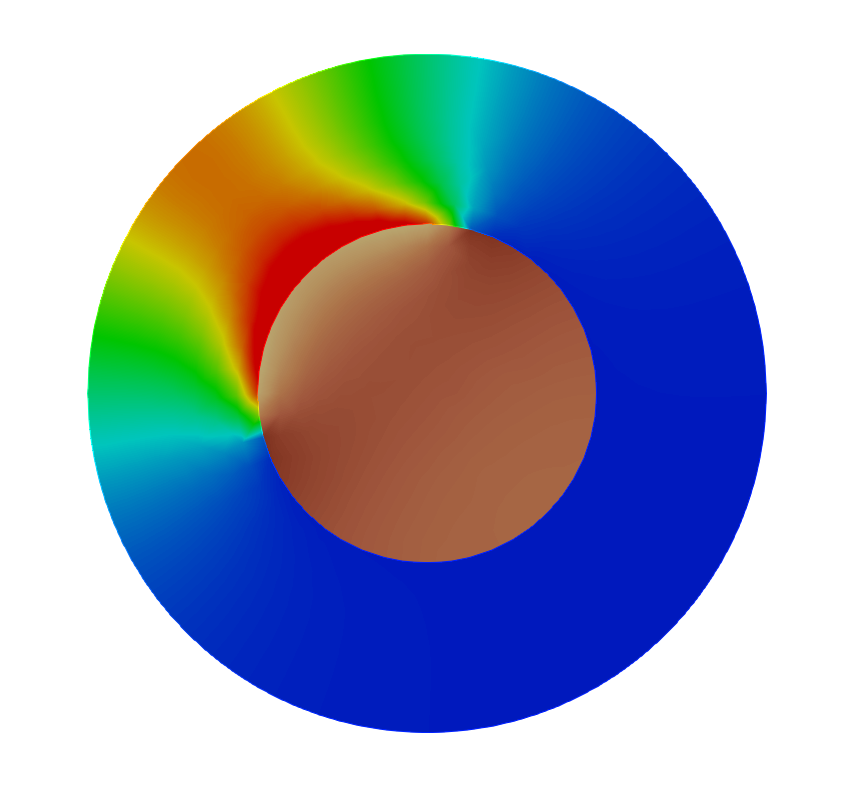}}
	\quad
	
	\subfigure[$t=1.08$]{\label{fig:e}
		\includegraphics[width=0.225\textwidth]{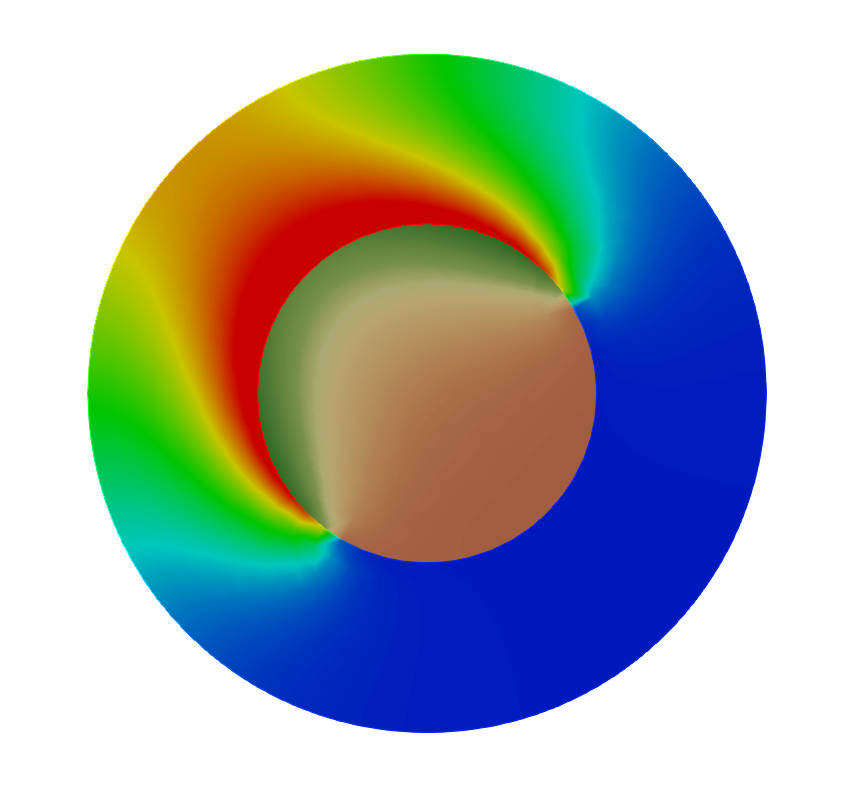}}
	\subfigure[$t=1.44$]{\label{fig:f}
		\includegraphics[width=0.225\textwidth]{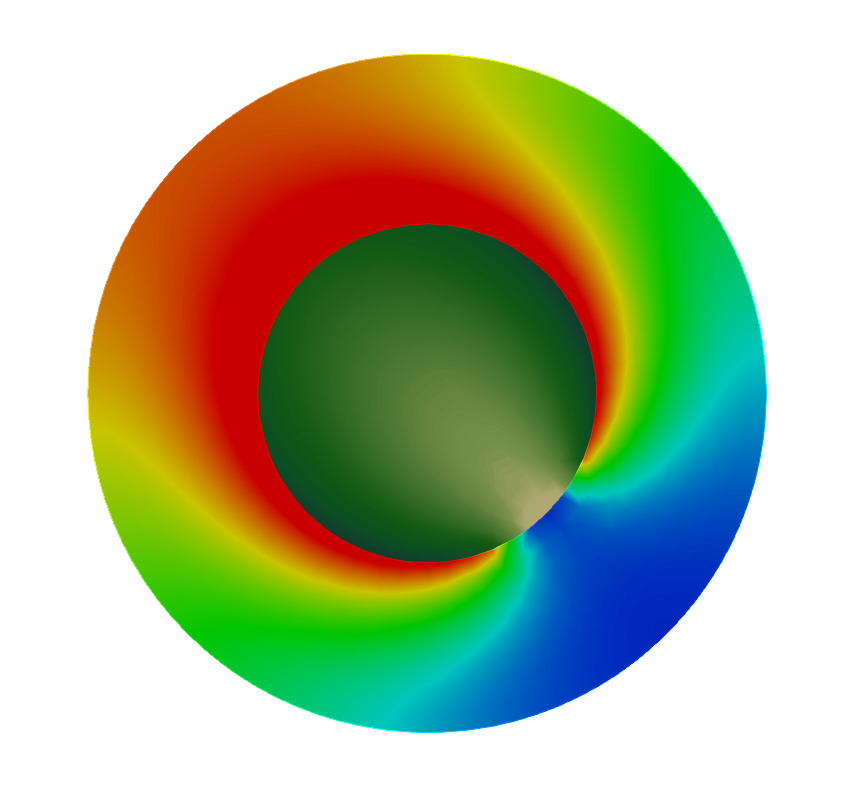}}
	\subfigure[$t=1.92$]{\label{fig:g}
		\includegraphics[width=0.225\textwidth]{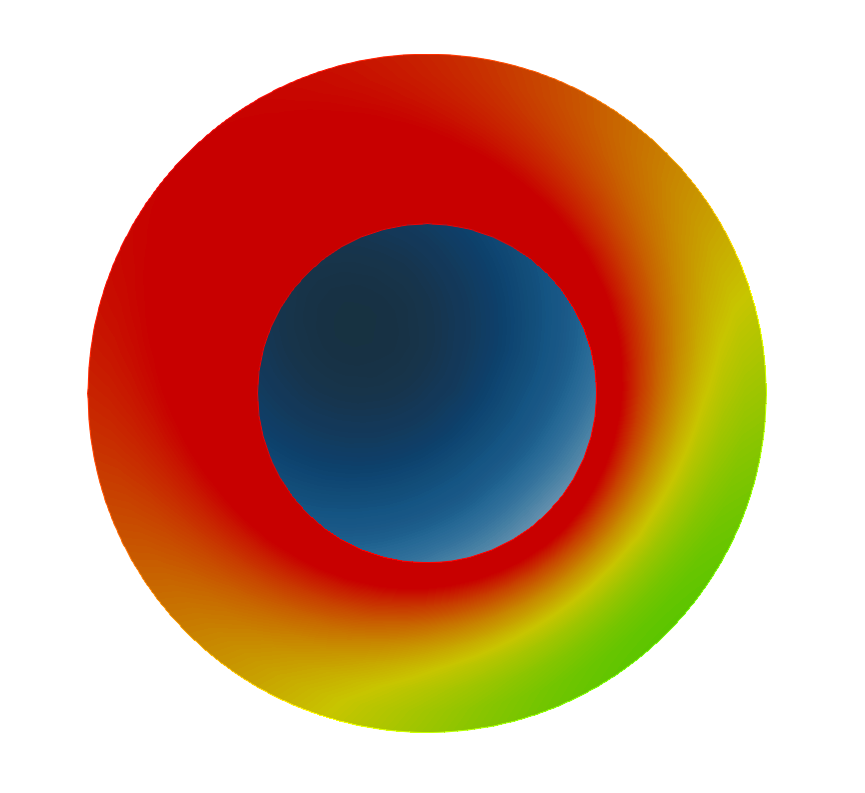}}
	\subfigure[$t=3.12$]{\label{fig:h}
		\includegraphics[width=0.225\textwidth]{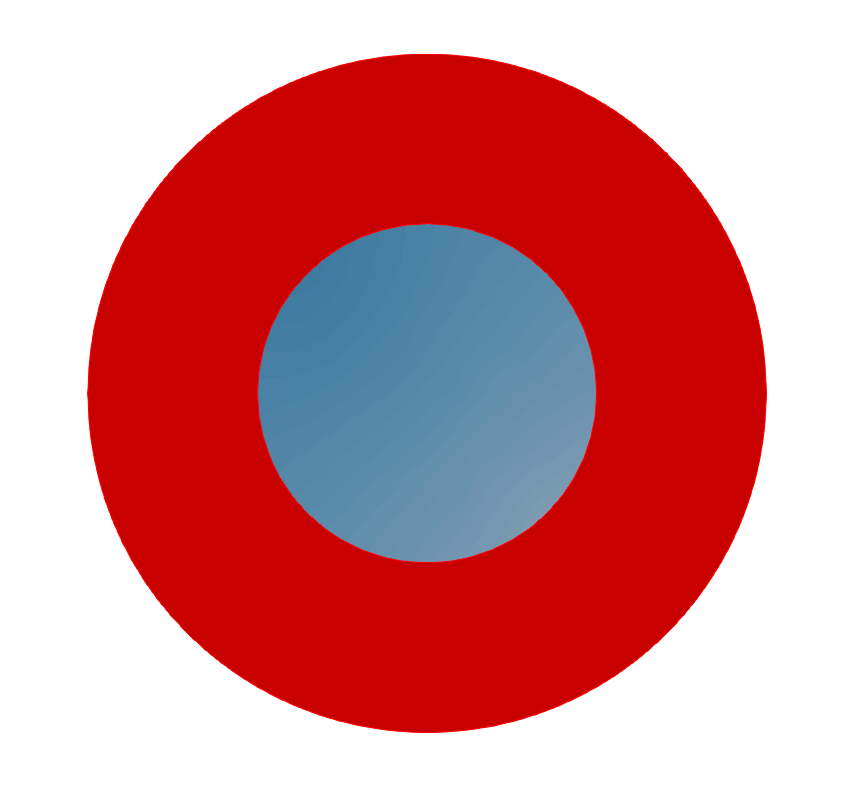}}
	\quad
	
	\subfigure[$t=4.32$]{\label{fig:i}
		\includegraphics[width=0.225\textwidth]{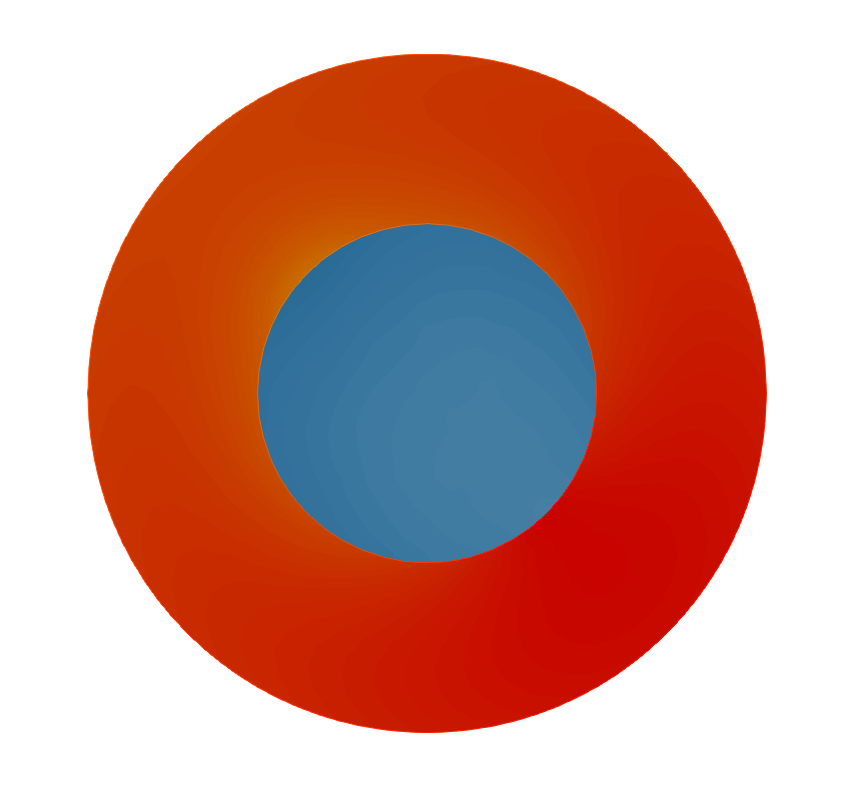}}
	\subfigure[$t=5.52$]{\label{fig:j}
		\includegraphics[width=0.225\textwidth]{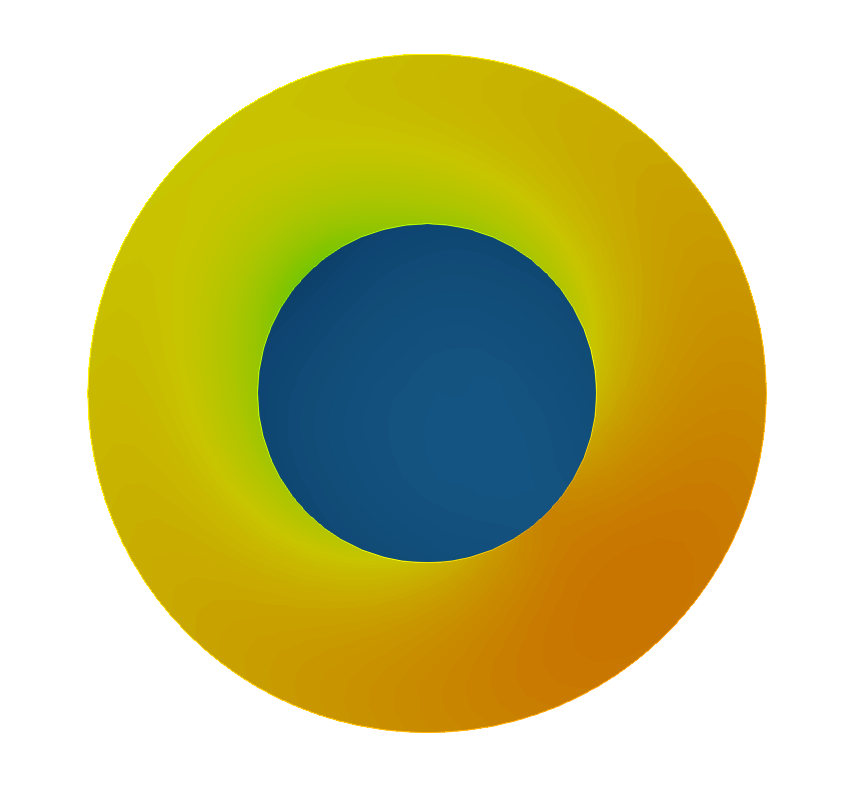}}
	\subfigure[$t=7.32$]{\label{fig:k}
		\includegraphics[width=0.225\textwidth]{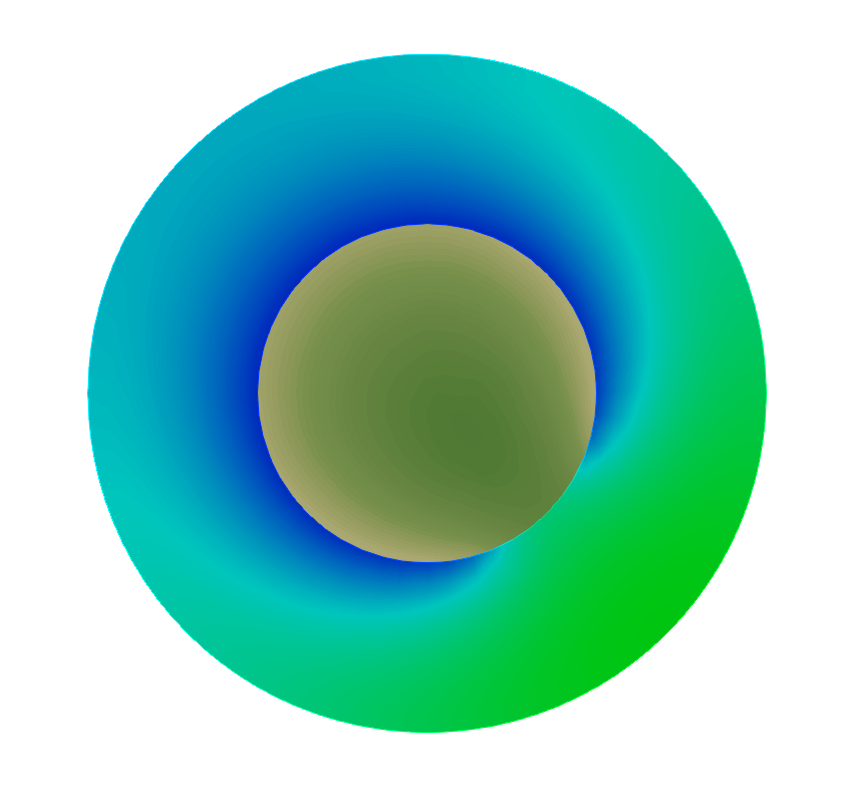}}
	\subfigure[$t=9.12$]{\label{fig:l}
		\includegraphics[width=0.225\textwidth]{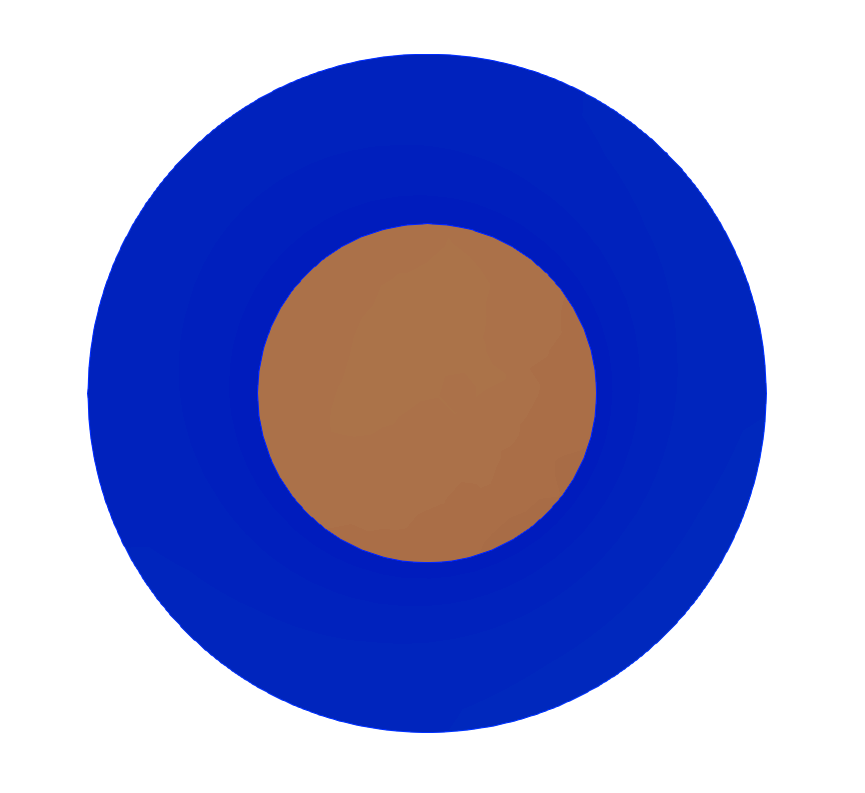}}
	\quad
	
	\subfigure[color bar for $u$ in $\Omega_c$]{\label{fig:bar1}
		\includegraphics[width=0.44\textwidth]{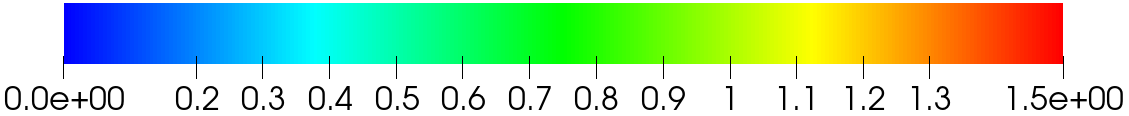}}
	\subfigure[color bar for $u_e$ in $\Omega_e$]{\label{fig:bar2}
		\includegraphics[width=0.44\textwidth]{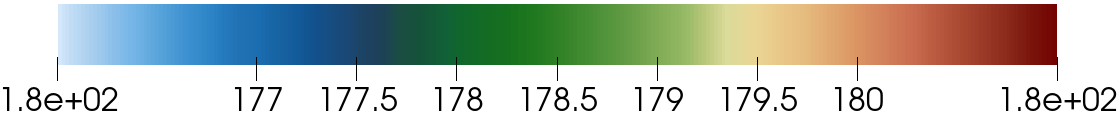}}
	\caption{Initiation (a)-(b), propagation (c)-(h) and recovery (i)-(l) of  the  calcium wave in a 2D cell. As in (l) - an equilibrium state, the black region is cytosol ($\Omega_c$), the brown region is ER ($\Omega_e$). $u$ and $u_e$ are the calcium concentrations in cytosol and ER respectively.}
	\label{fig:ca2+wave}
\end{figure}

\begin{figure}[ht!]
	\centering
	\subfigure[$t=0.72$]{\label{fig:ap}
		\includegraphics[width=0.225\textwidth]{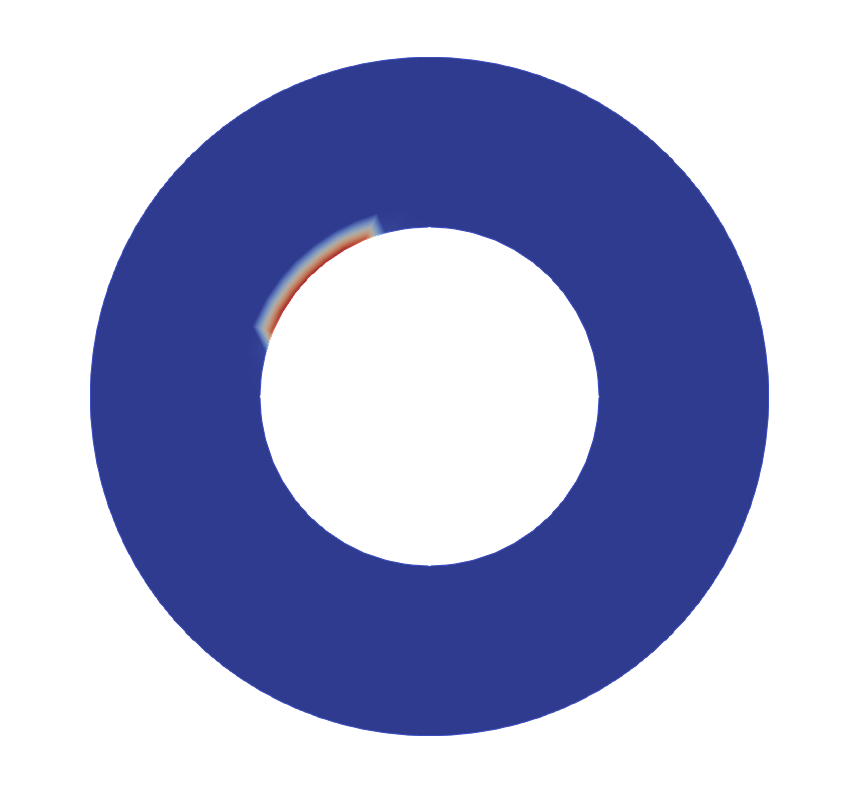}}
	\subfigure[$t=0.84$]{\label{fig:bp}
		\includegraphics[width=0.225\textwidth]{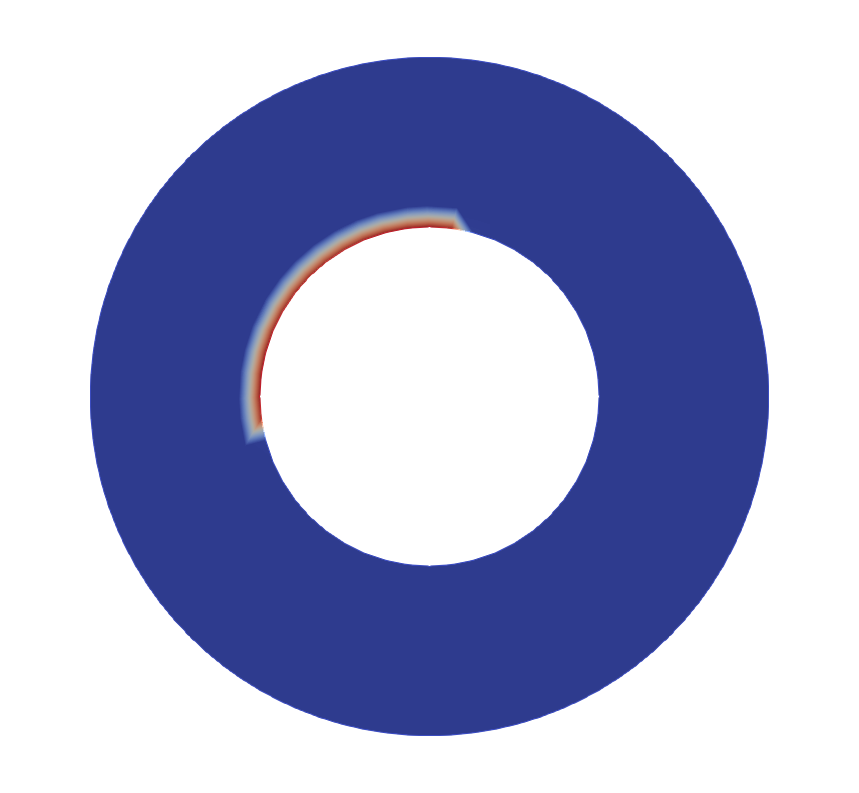}}
	\subfigure[$t=0.96$]{\label{fig:cp}
		\includegraphics[width=0.225\textwidth]{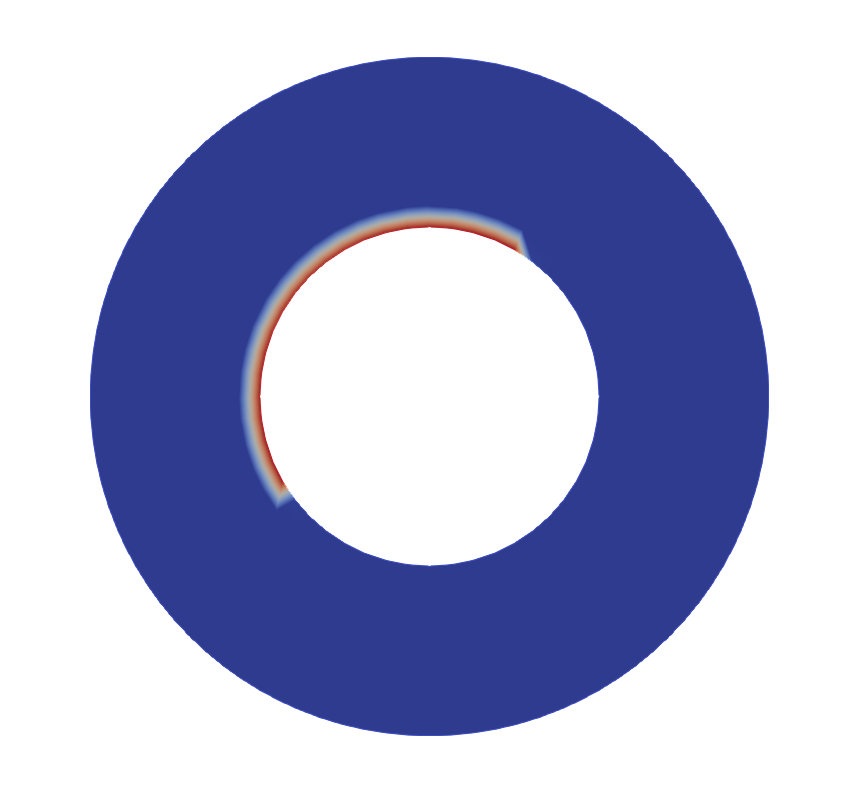}}
	\subfigure[$t=1.08$]{\label{fig:dp}
		\includegraphics[width=0.225\textwidth]{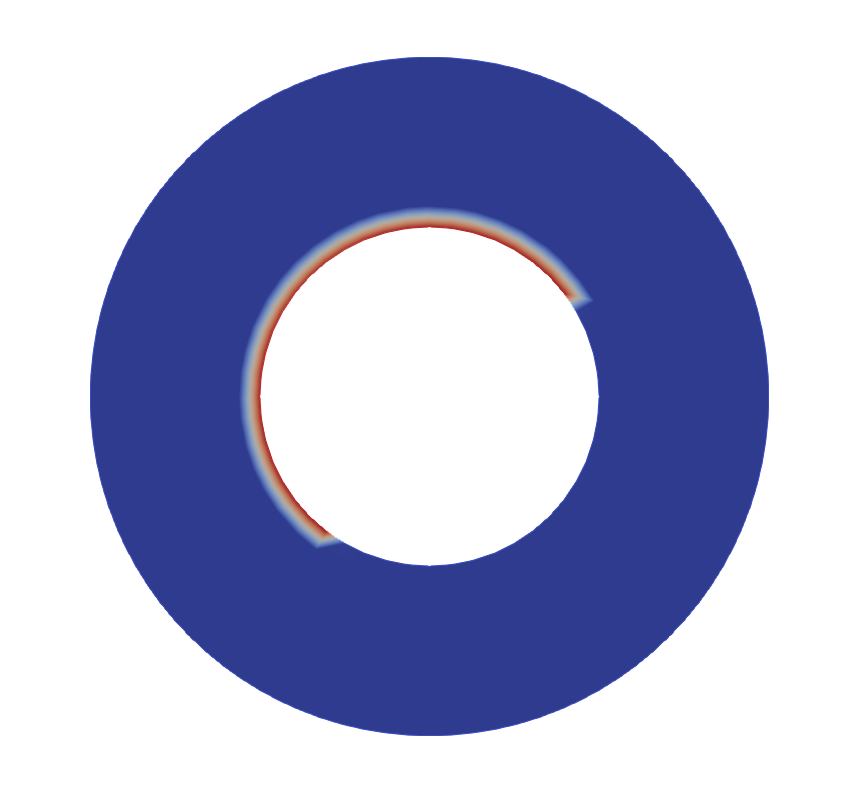}}
	\quad
	
	\subfigure[$t=1.20$]{\label{fig:ep}
		\includegraphics[width=0.225\textwidth]{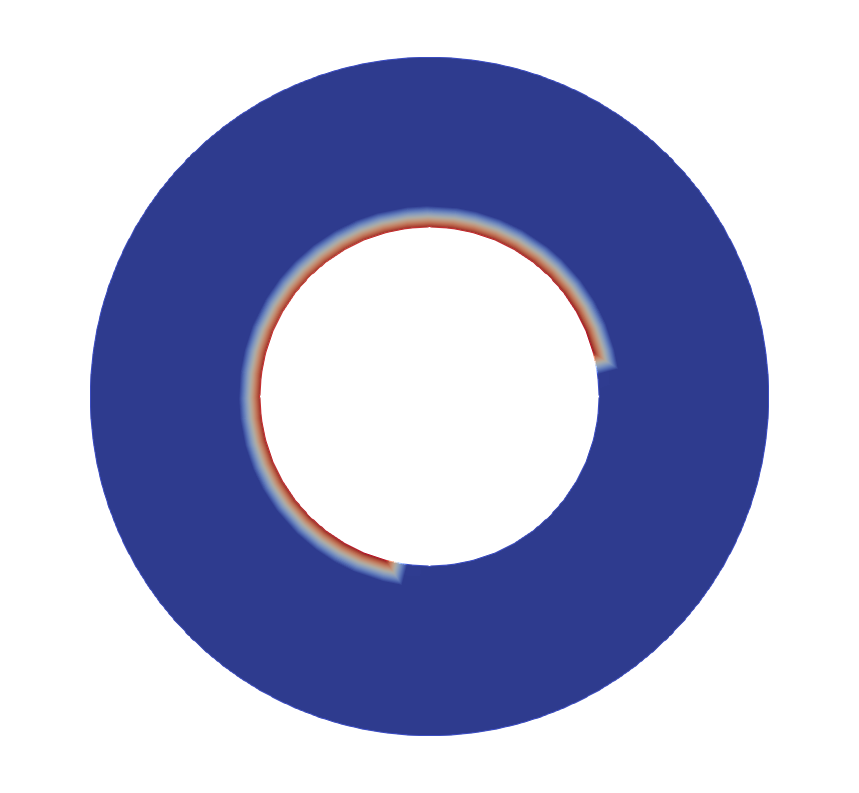}}
	\subfigure[$t=1.32$]{\label{fig:fp}
		\includegraphics[width=0.225\textwidth]{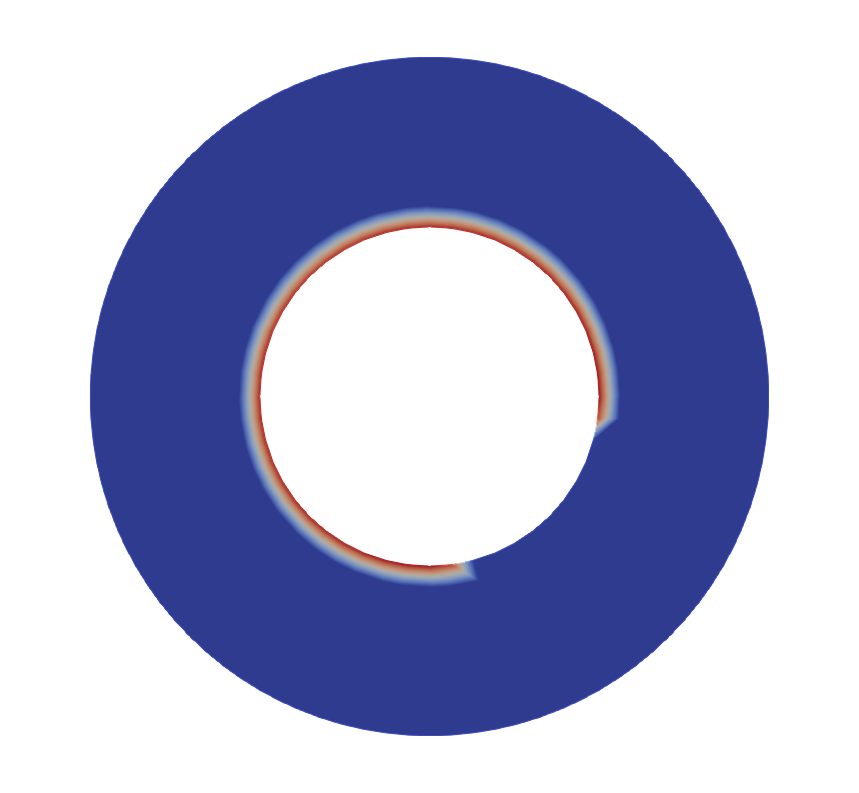}}
	\subfigure[$t=1.44$]{\label{fig:gp}
		\includegraphics[width=0.225\textwidth]{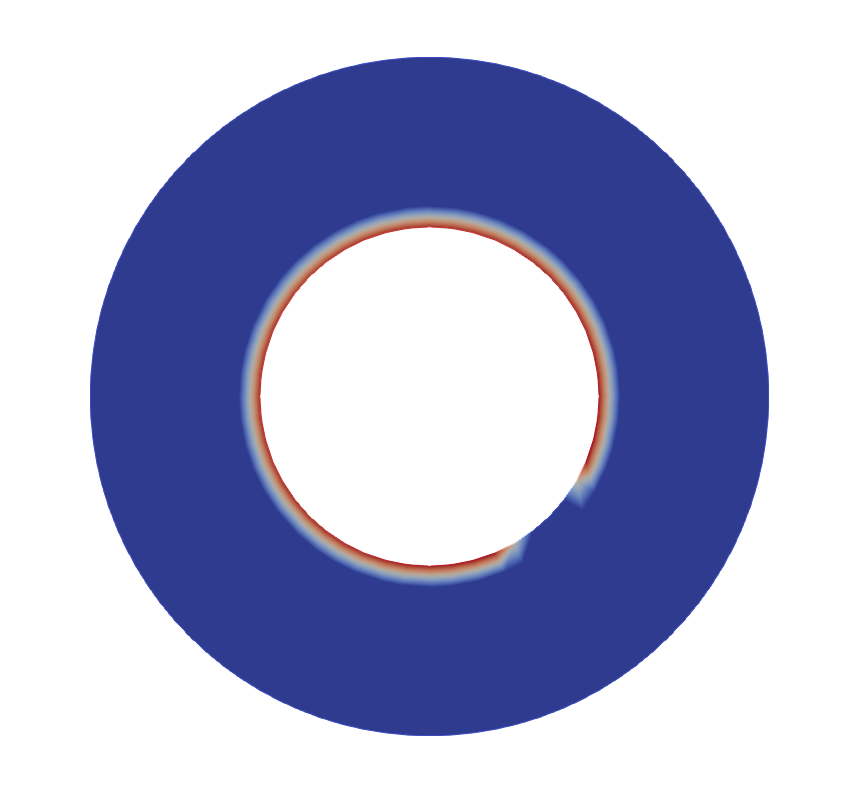}}
	\subfigure[$t=1.56$]{\label{fig:hp}
		\includegraphics[width=0.225\textwidth]{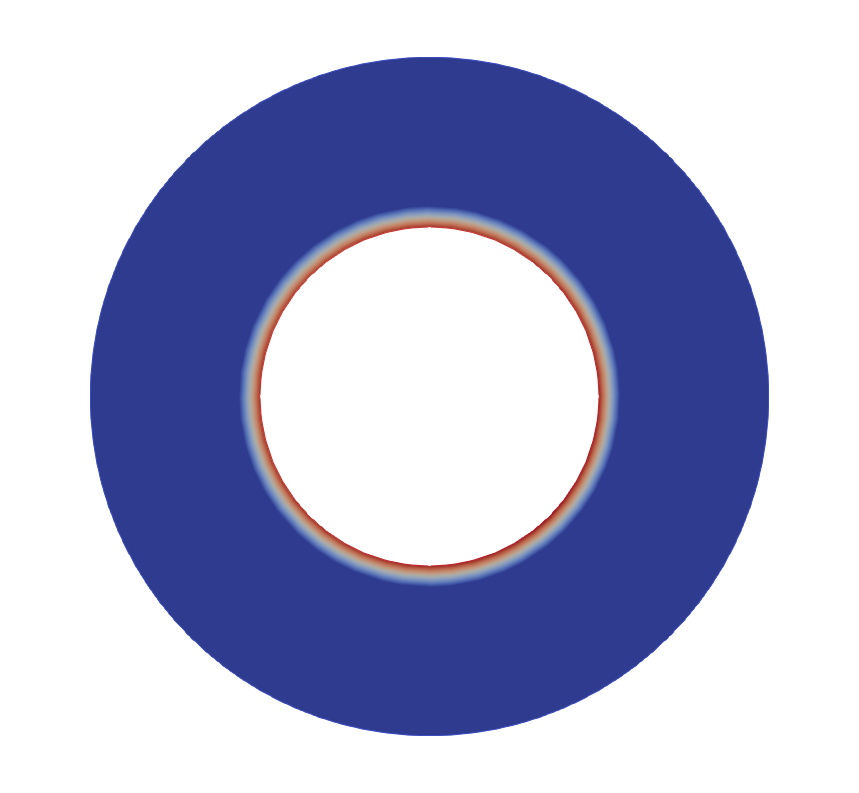}}
	\quad
	
	\subfigure[$t=4.80$]{\label{fig:ip}
		\includegraphics[width=0.225\textwidth]{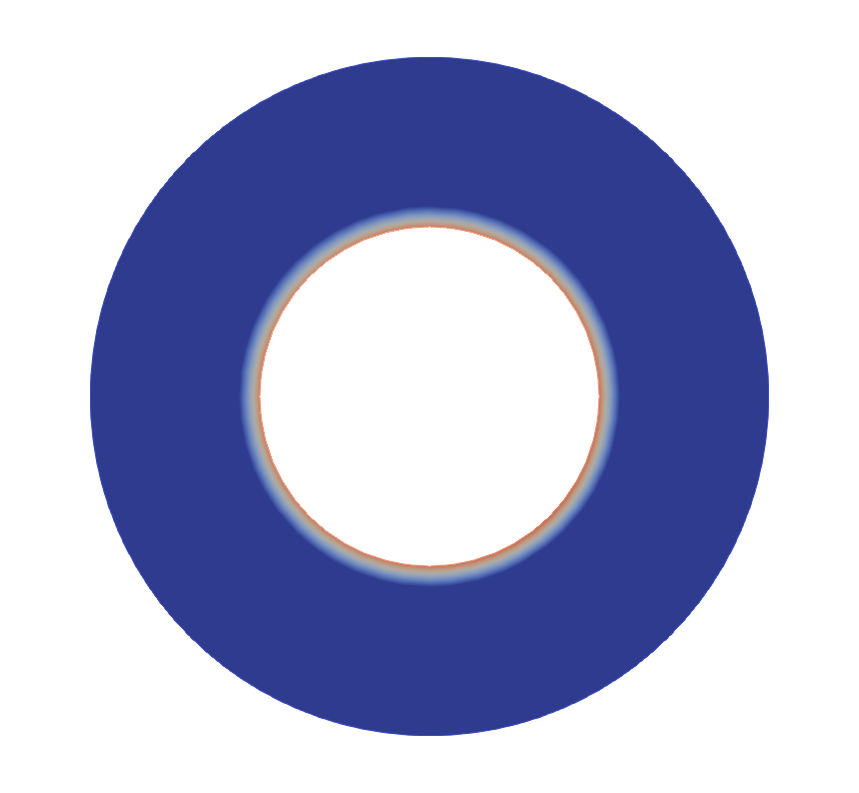}}
	\subfigure[$t=5.76$]{\label{fig:jp}
		\includegraphics[width=0.225\textwidth]{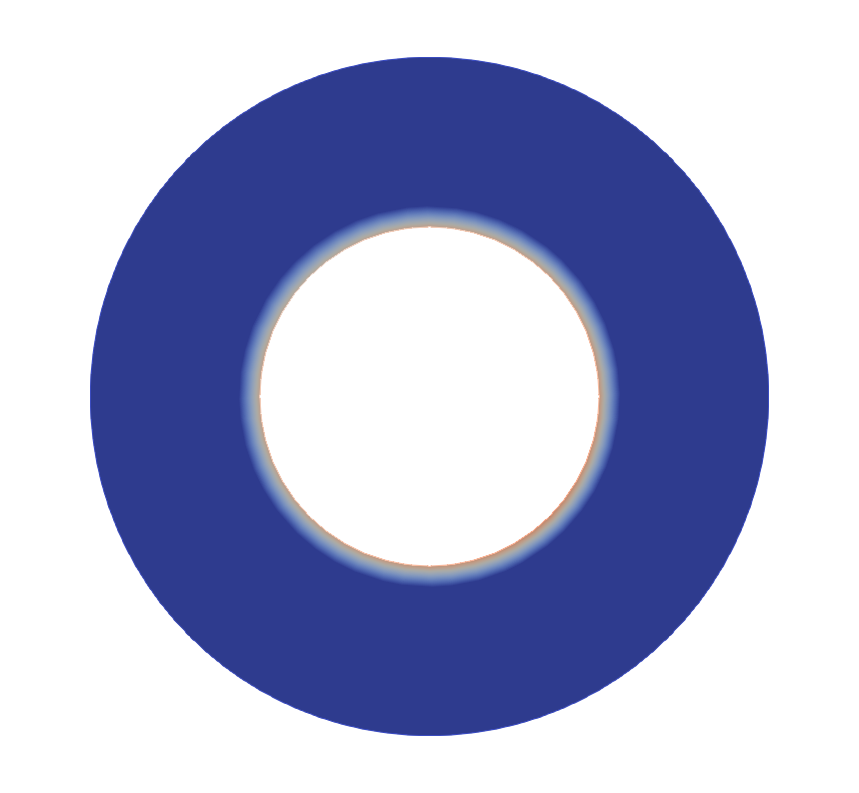}}
	\subfigure[$t=6.72$]{\label{fig:kp}
		\includegraphics[width=0.225\textwidth]{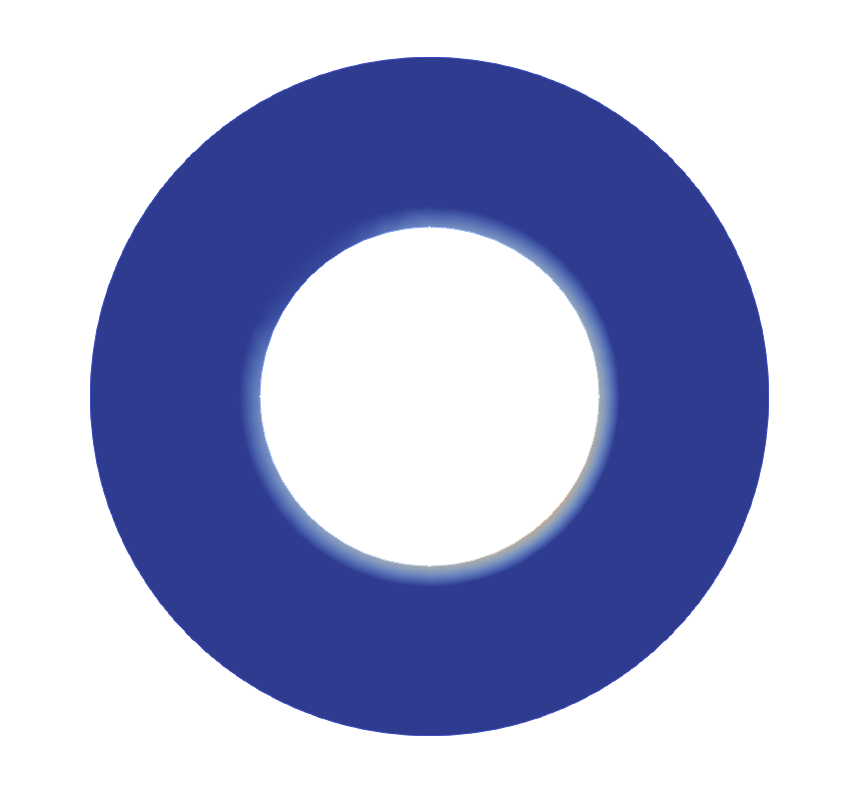}}
	\subfigure[$t=7.68$]{\label{fig:lp}
		\includegraphics[width=0.225\textwidth]{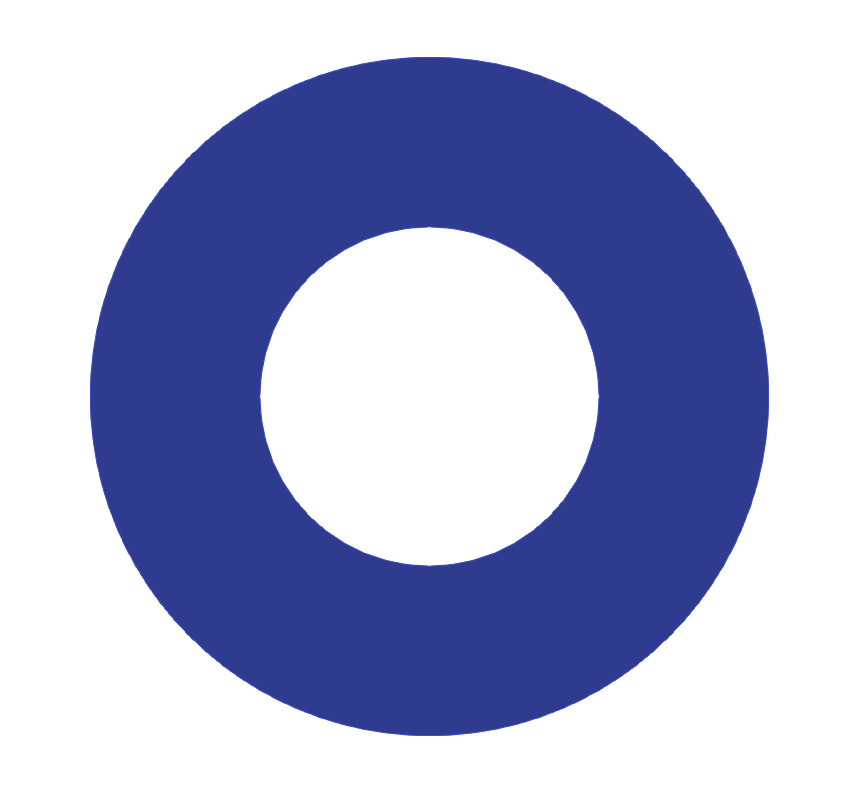}}
	\quad
	
	\subfigure[Color bar for the value of open probability on the inner 
	            circle -$\Upsilon$]{\label{fig:bar1p}
		\includegraphics[width=0.9\textwidth]{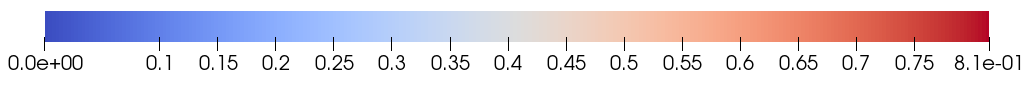}}
	\caption{The black region as in (l) is $\bar{\Omega}_c$, the value is always 0 on $\bar{\Omega}_c/\Upsilon$. Red color part on the inner circle (ER membrane) means the RyR channels are open. We can see that from (a)-(h), the open state propagates to the whole ER membrane which leads to the release of calcium from ER. From (i) to (l), the value of open probability decreases to the equilibrium state.}
	\label{fig:po3}
\end{figure}

{\color{black} {\bf Example }4. In this example we present a full system with different coefficients, taken from \cite{Breit2018}, which produces $Ca^{2+}$ waves. Units were adjusted so that $t$ has unit $s$, $u, u_e$ have unit $\upmu$M:
\begin{equation}
	\left\{
	\begin{aligned}
		&{\partial_t u} -220\Delta u = f(b,u) 
		\quad \text{on } \Omega_c\times(0,T]\\
		&{\partial_t b} -20\Delta b = f(b,u) 
		\quad \text{on } \Omega_c\times(0,T]\\
		& {\partial_t u_e} -220\Delta u_e = 0  \quad \text{on } \Omega_e\times(0,T]
	\end{aligned}
	\right. 
\end{equation}
where $T=80$, $f(b,u) = K_b^-(b^0-b)-K_b^+ bu$, $\partial_n b =0$ on $\partial\Omega_c$, and other boundary conditions are: 
\begin{align}
	\partial_n u 
	&= C_3(1000-u)-\frac{C_2u}{1.8+u}-\frac{C_1u^2}{0.06^2+u^2}
	+g(x,y,t) \ {\text{ on } \partial\Omega}\times(0,T] \label{ueBdyonGamma4414}  \\
	\partial_n u  
	&= 
	C_{1}^e P(t,u)(u_e-u)-C_2^e\frac{u}{(0.18+u)u_e}+C_3^e(u_e-u) 
	\ \text{ on } \Upsilon\times(0,T]\\
	\partial_n u_e
	&= 
	C_{1}^e P(t,u)(u-u_e)+C_2^e\frac{u}{(0.18+u)u_e}-C_3^e(u_e-u) 
	\ \text{ on } \label{ueBdyonGamma4434} \Upsilon\times(0,T]
\end{align}
The initial conditions are $u(x,y,0) = 0.05$, $b(x,y,0)=37$, $u_e(x,y,0)=250$, in $f(b,u)$, $b^0=40,$ $K_b^-=16.65,$ $K_b^+=27.$ The ODE system is the same as in Example 1, and the initial conditions of the ODE are  $c_1(0) = 0.994$, $o(0) =1.5721\times 10^{-7},$ $c_2(0) =5.6625\times 10^{-3}.$ 
In \eqref{ueBdyonGamma4414}, the value $1,000$ is the extracellular Ca$^{2+}$ concentration, and $g$ is a calcium influx function: 
$g(x,y,t) = 240 e^{-0.01/(0.01-(t-0.2)^2)+1}$ if $0.1<t<0.3$ and $y-x\geq 2.5$; $g(x,y,t) = 0$ elsewhere. The coefficients in \eqref{ueBdyonGamma4414} to \eqref{ueBdyonGamma4434} are: $C_{1}^e =0.829468,$ $C_2^e = 11000,$ $C_3^e =0.038,$ $C_1=8.5,$ $C_2=37.6,$ $C_3 =0.0045.$ 

Example 4 is constructed to show the initiation and propagation of the calcium wave in a full model, see Figure \ref{fig:ca2+wave4}. For computation, we use a similar geometry as in Examples 3. The radii of the two circles, with center (0,0), are 1.2 and 2, but with different meshes and $P_1$ elements, where $T=80$, $\Delta t=0.01/16,$ and the spatial mesh size is $h=\pi/32$. Due to larger diffusion coefficients and the buffer $b$, propagation of the calcium is much faster than Example 3, but the recovery is slower. Here, we don't show the graph of $b$ which varies from 2 to 38, since it's less important.  
Figure \ref{fig:ca2+wave4p} shows the open probability function $P(t,u)$ in equation \eqref{ptu} for RyR channels on the ER membrane. It ranges from 0 to 0.96.
Instability of the scheme \eqref{FDFEM} can be observed with time step size larger than $0.01/16$. 
}
\begin{figure}[ht!]
	\centering
	\subfigure[$t=0.04$]{\label{fig:a4}
		\includegraphics[width=0.225\textwidth]{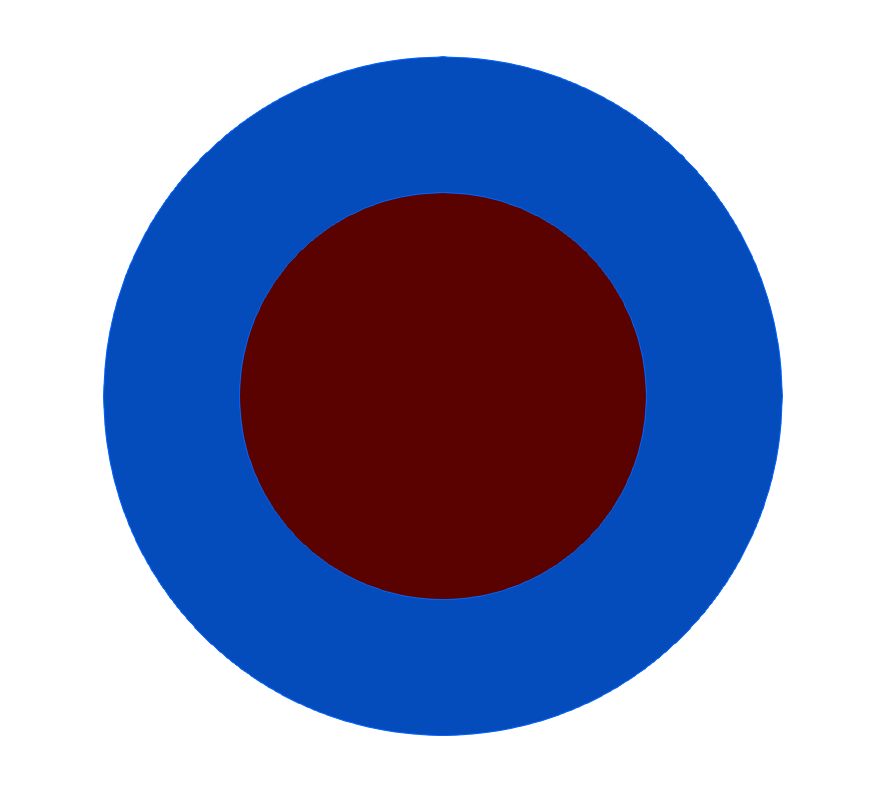}}
	\subfigure[$t=0.24$]{\label{fig:b4}
		\includegraphics[width=0.225\textwidth]{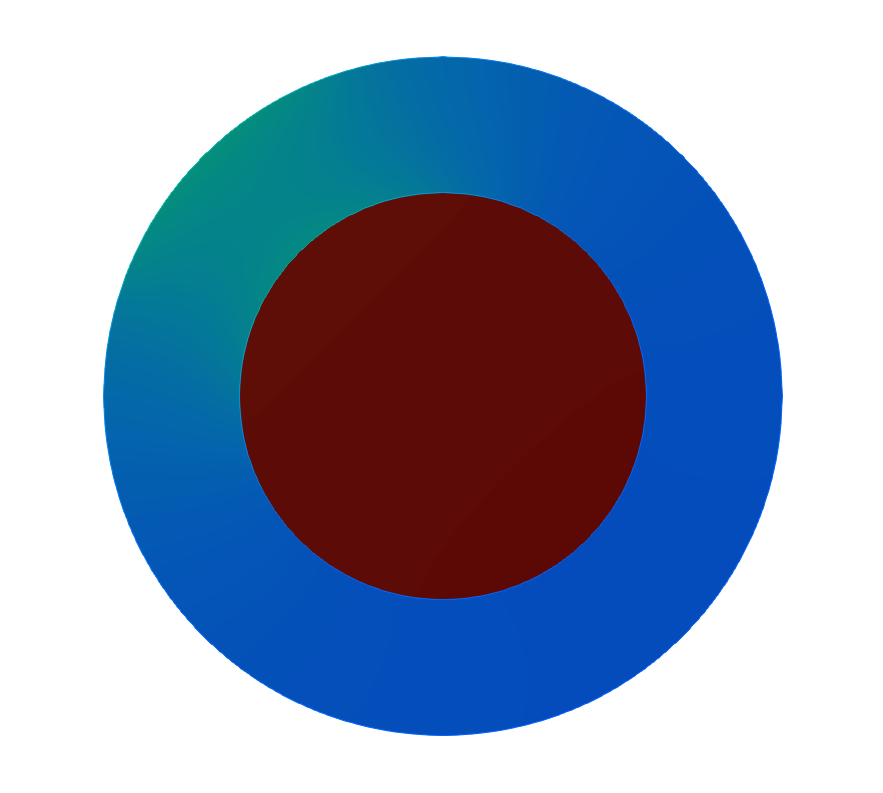}}
	\subfigure[$t=0.44$]{\label{fig:c4}
		\includegraphics[width=0.225\textwidth]{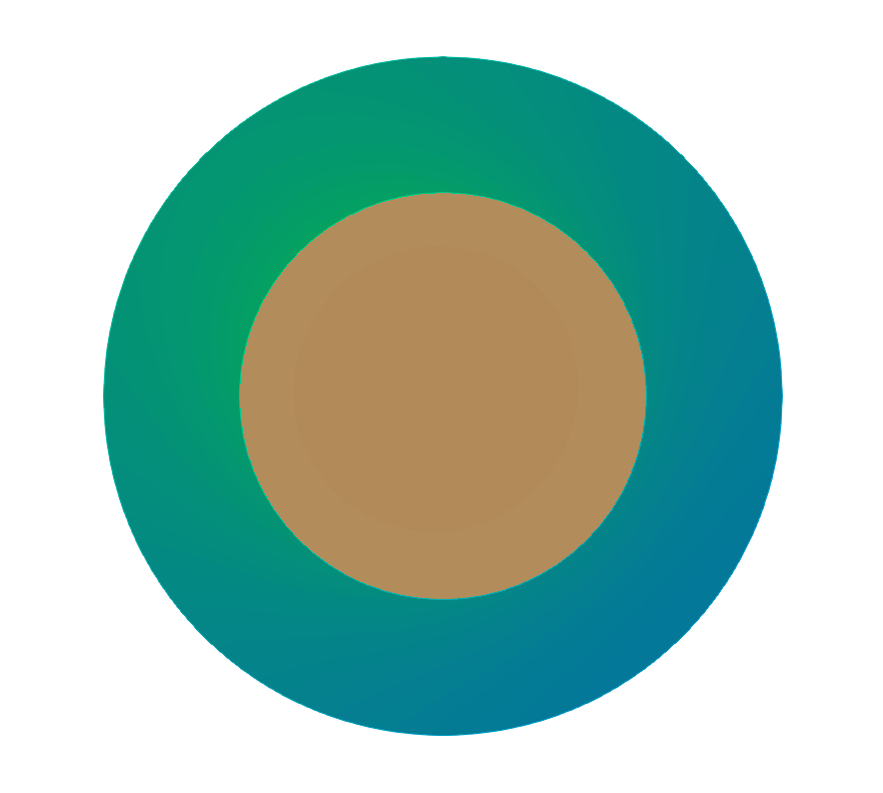}}
	\subfigure[$t=0.64$]{\label{fig:d4}
		\includegraphics[width=0.225\textwidth]{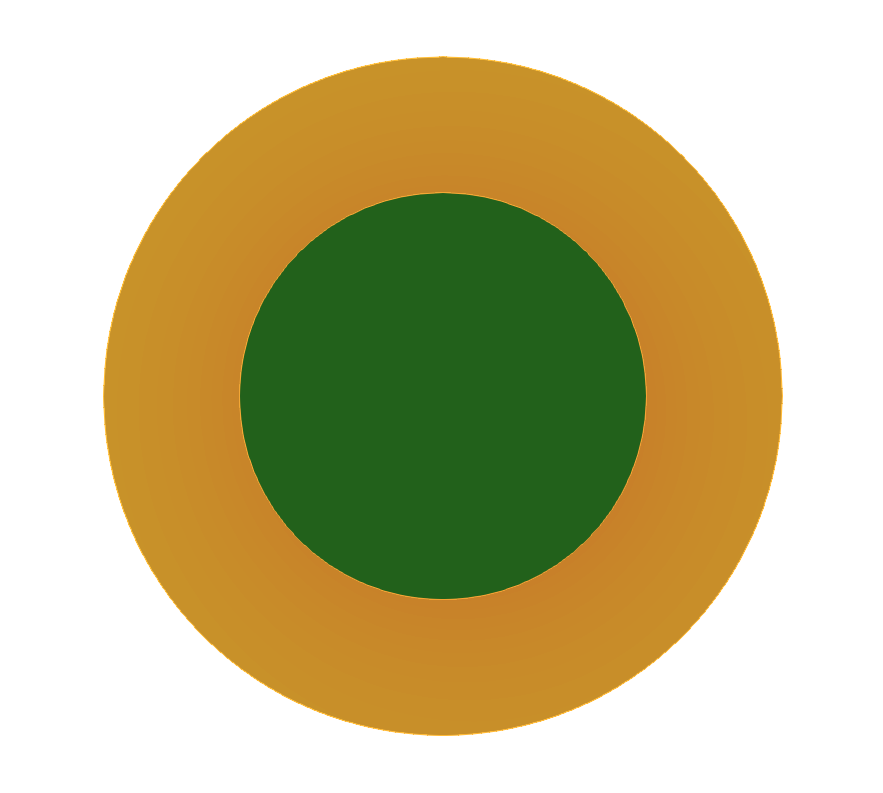}}
	\quad
	
	\subfigure[$t=0.84$]{\label{fig:e4}
		\includegraphics[width=0.225\textwidth]{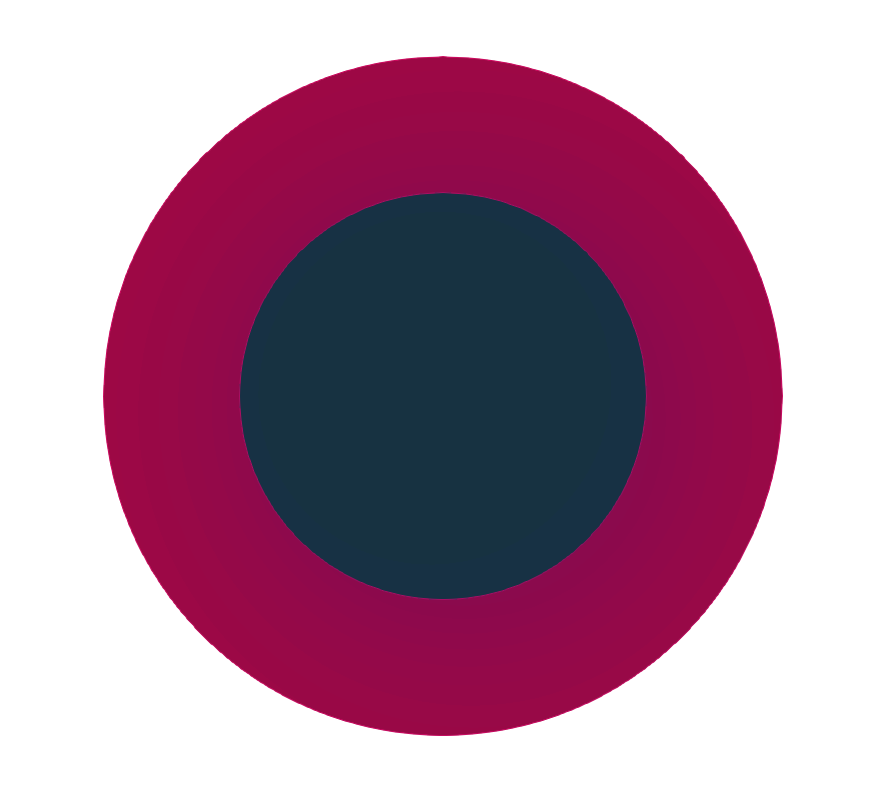}}
	\subfigure[$t=1.04$]{\label{fig:f4}
		\includegraphics[width=0.225\textwidth]{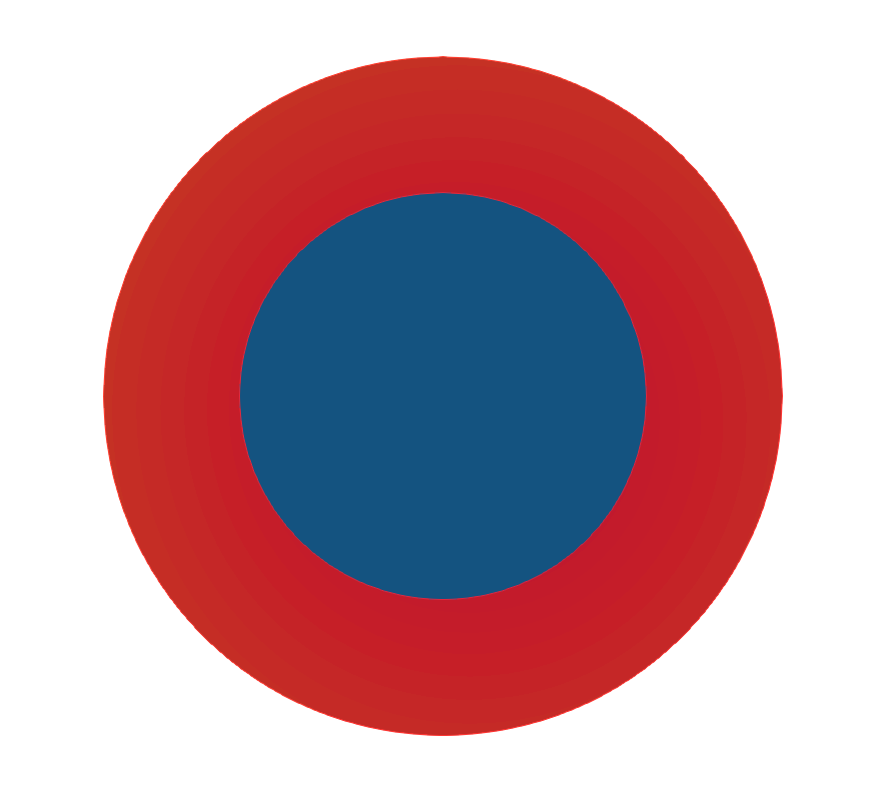}}
	\subfigure[$t=1.24$]{\label{fig:g4}
		\includegraphics[width=0.225\textwidth]{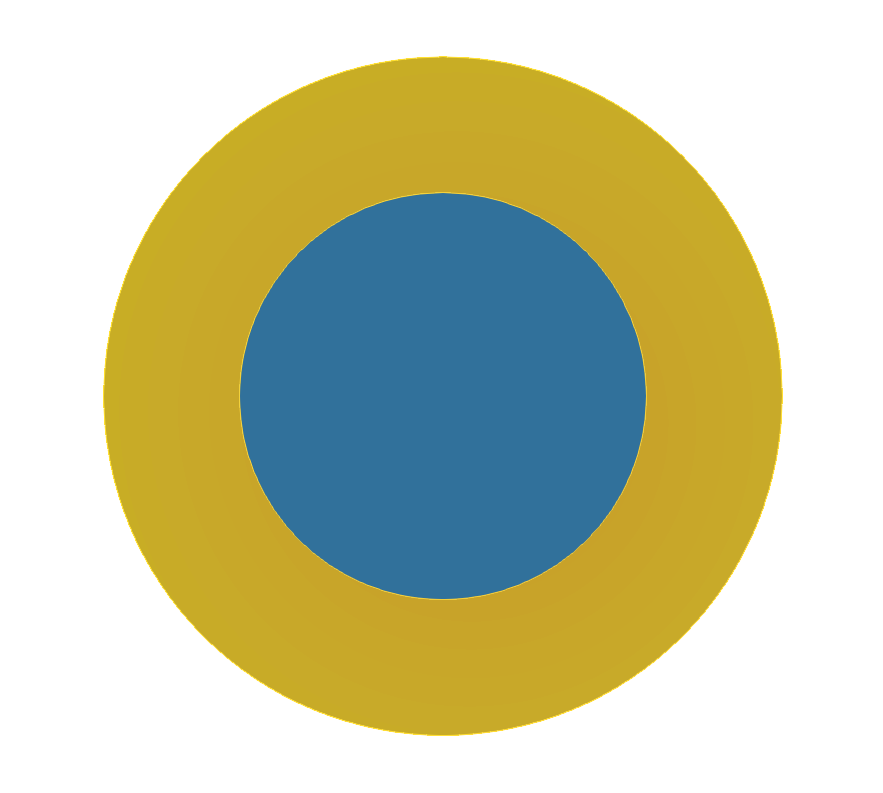}}
	\subfigure[$t=1.44$]{\label{fig:h4}
		\includegraphics[width=0.225\textwidth]{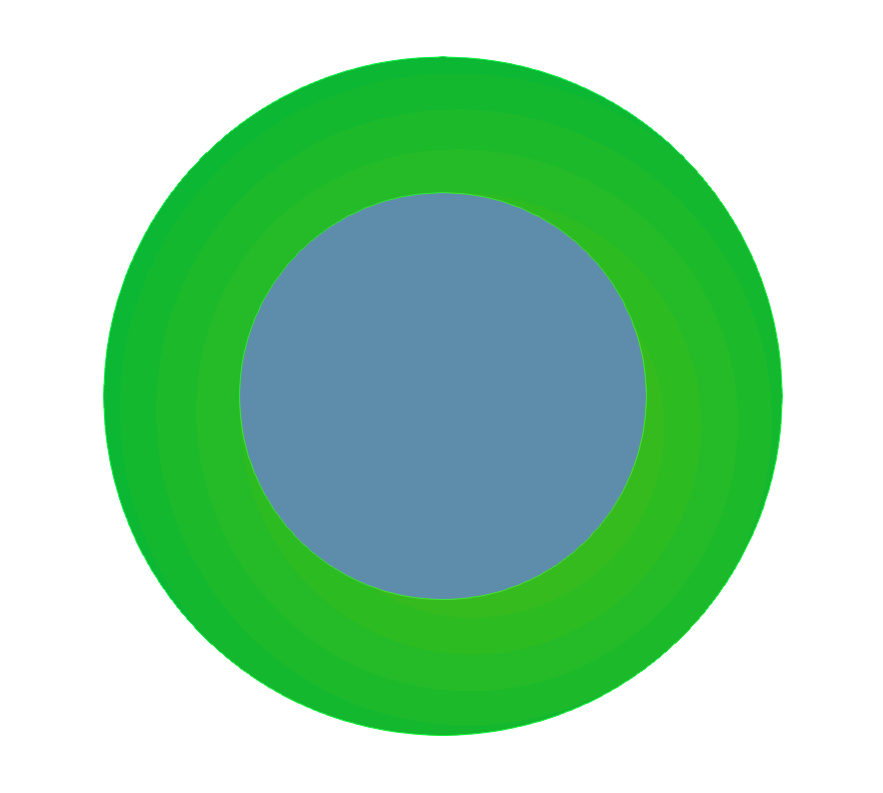}}
	\quad
	
	\subfigure[$t=1.64$]{\label{fig:i4}
		\includegraphics[width=0.225\textwidth]{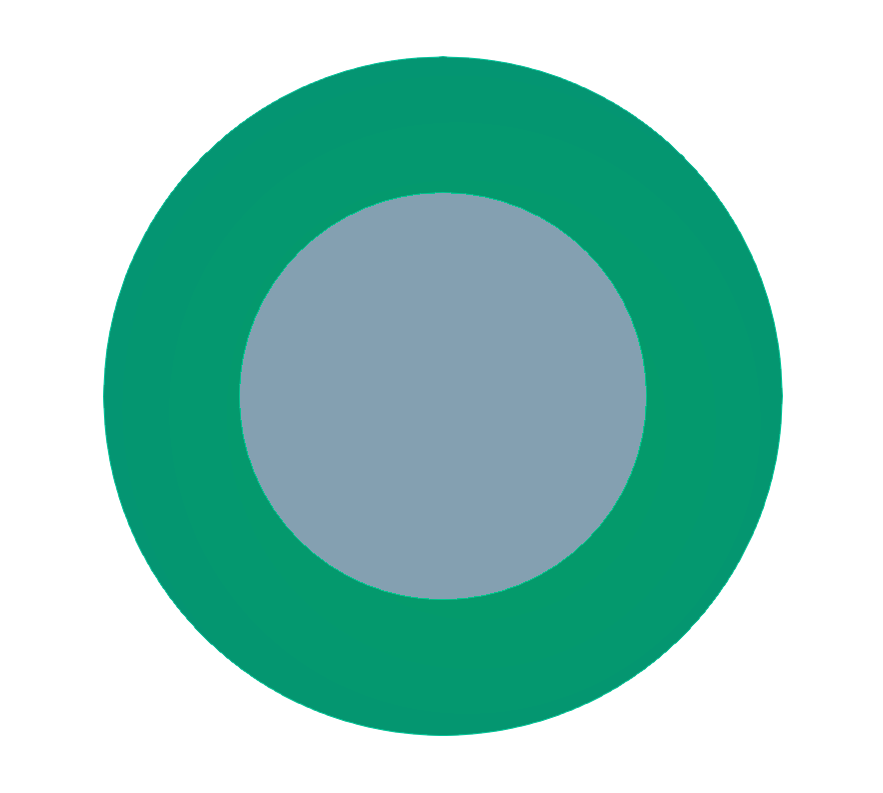}}
	\subfigure[$t=1.88$]{\label{fig:j4}
		\includegraphics[width=0.225\textwidth]{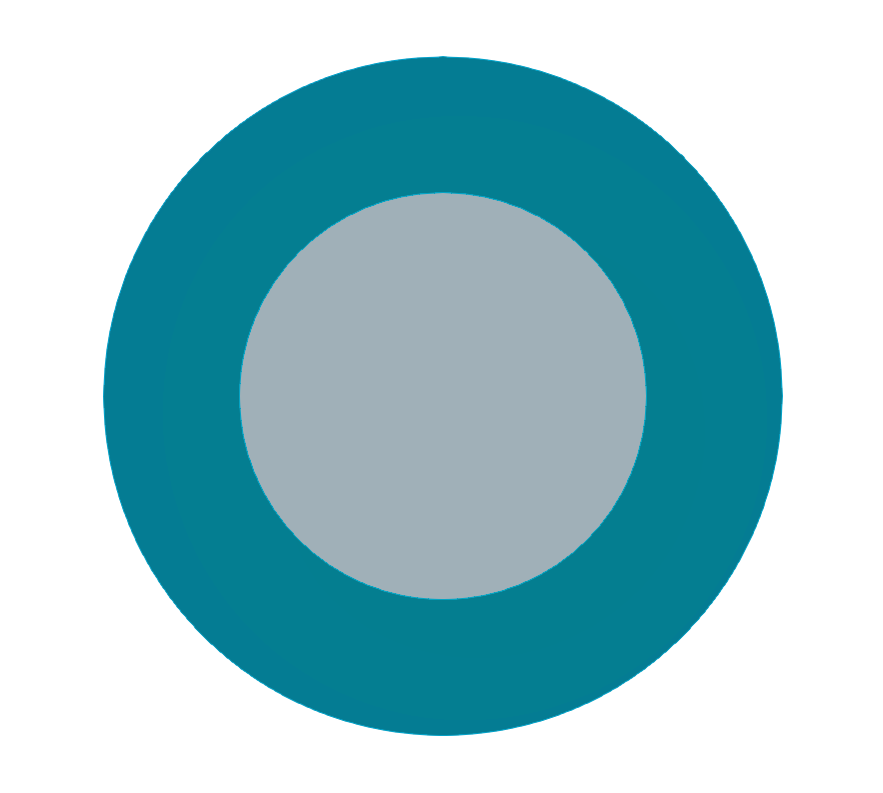}}
	\subfigure[$t=2.44$]{\label{fig:k4}
		\includegraphics[width=0.225\textwidth]{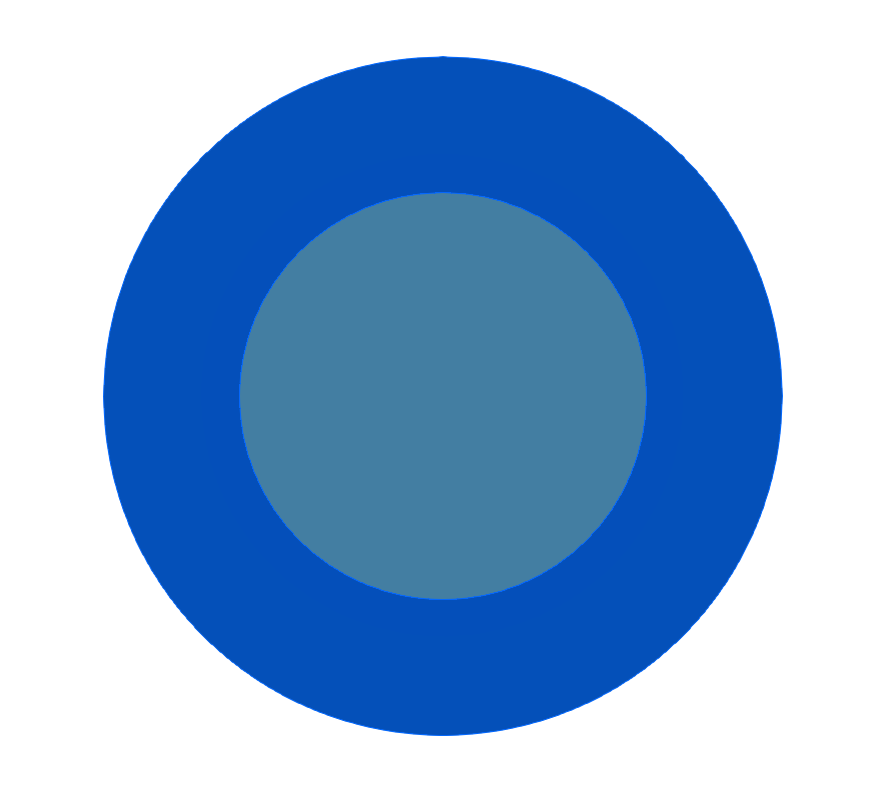}}
	\subfigure[$t=3.64$]{\label{fig:l4}
		\includegraphics[width=0.225\textwidth]{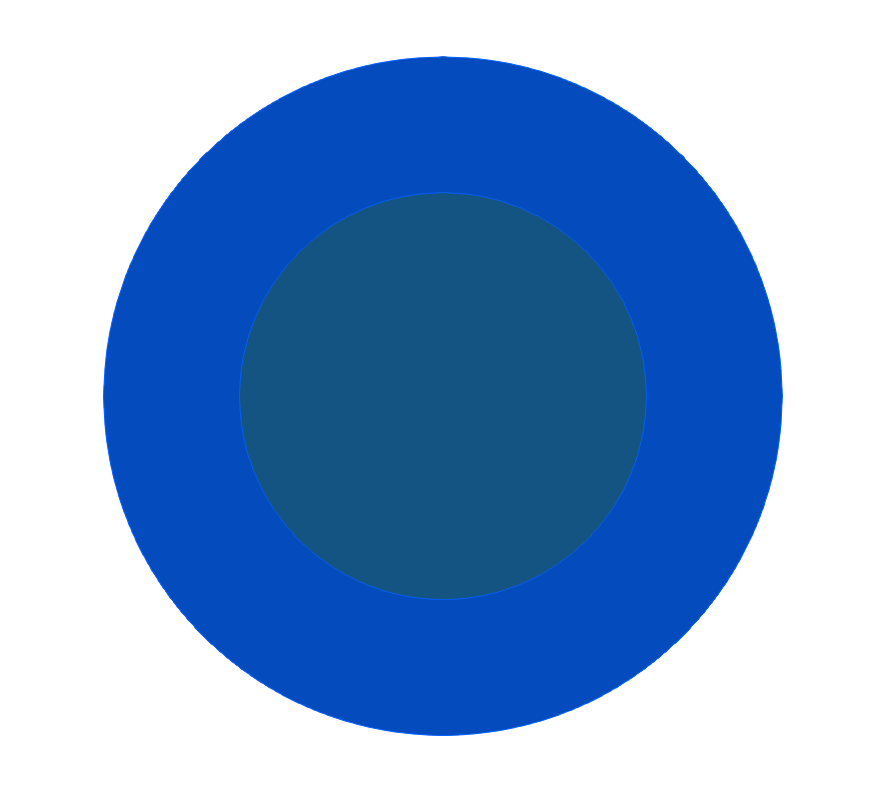}}
	\quad
	
	\subfigure[$t=4.8$]{\label{fig:m4}
		\includegraphics[width=0.225\textwidth]{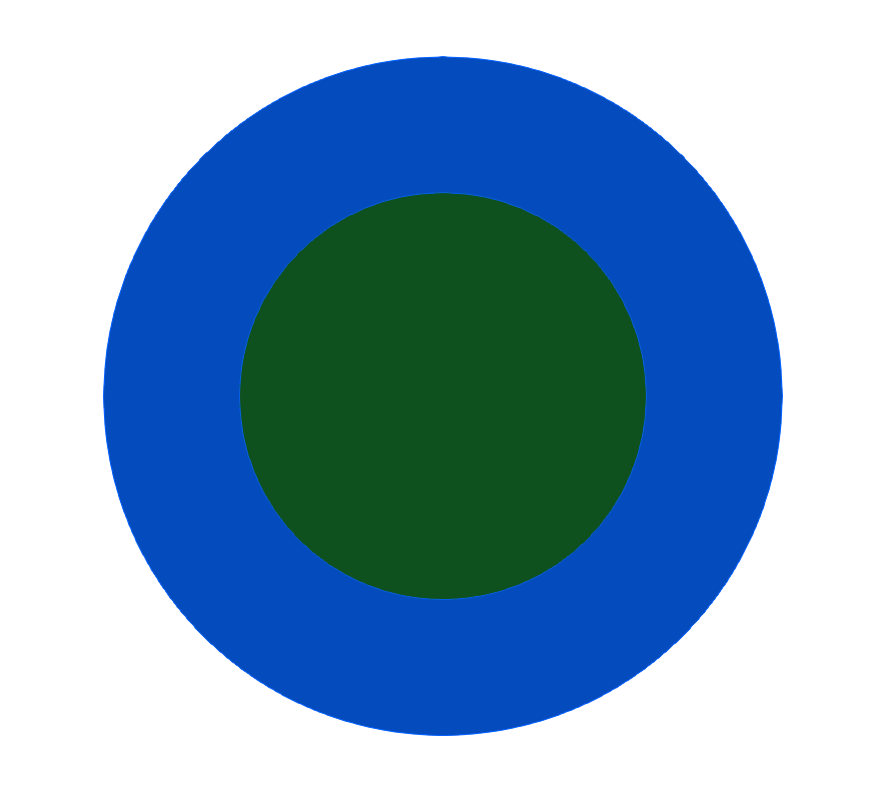}}
	\subfigure[$t=16.8$]{\label{fig:n4}
		\includegraphics[width=0.225\textwidth]{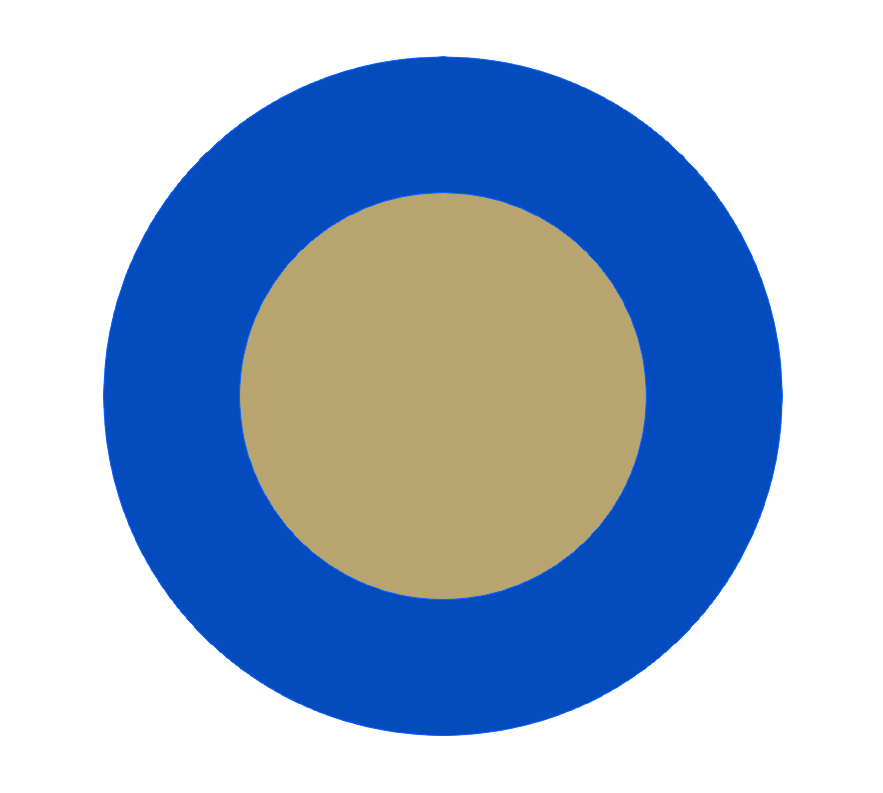}}
	\subfigure[$t=28.8$]{\label{fig:o4}
		\includegraphics[width=0.225\textwidth]{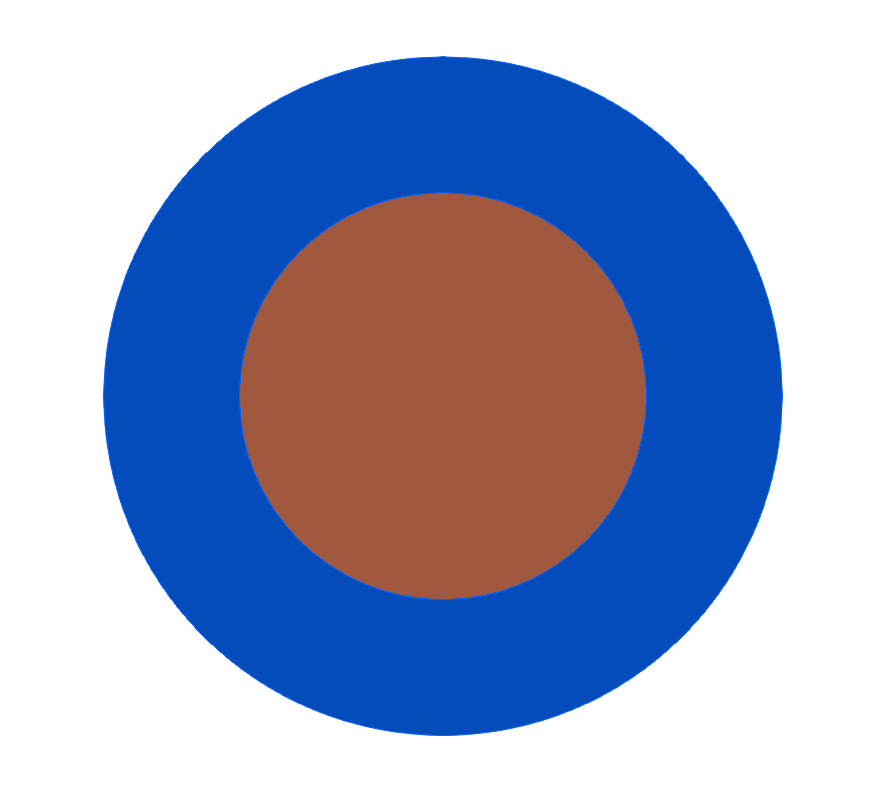}}
	\subfigure[$t=74.4$]{\label{fig:p4}
		\includegraphics[width=0.225\textwidth]{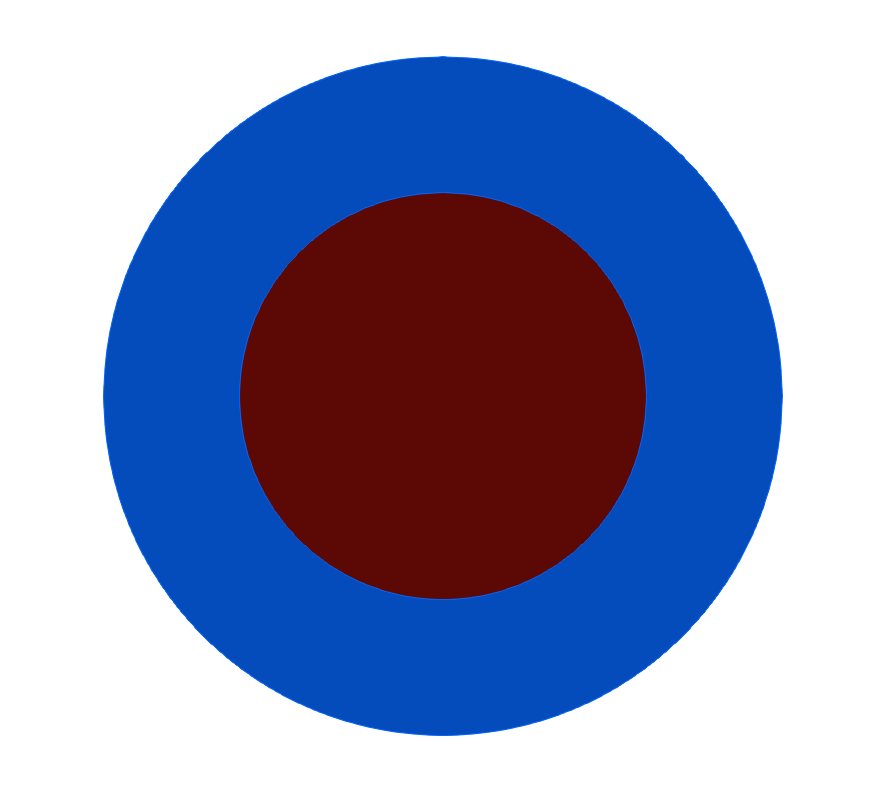}}
	\quad	
	\subfigure[color bar for $u$ in $\Omega_c$]{\label{fig:bar14}
		\includegraphics[width=0.44\textwidth]{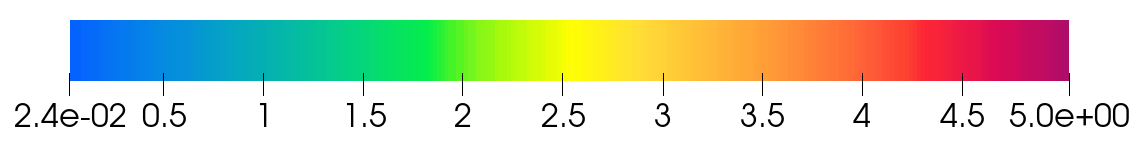}}
	\subfigure[color bar for $u_e$ in $\Omega_e$]{\label{fig:bar24}
		\includegraphics[width=0.44\textwidth]{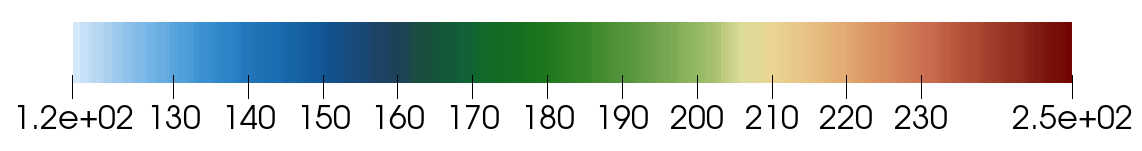}}
	\caption{Initiation (a)-(b), propagation (c)-(e) and recovery (f)-(p) of  the  calcium wave in a 2D cell. As in (a), (p) - the equilibrium state, the black region is cytosol ($\Omega_c$), the dark-red region is ER ($\Omega_e$). $u$ and $u_e$ are the calcium concentrations in cytosol and ER respectively.}
	\label{fig:ca2+wave4}
\end{figure}

\begin{figure}[ht!]
	\centering
	\subfigure[$t=0.16$]{\label{fig:a4p}
		\includegraphics[width=0.225\textwidth]{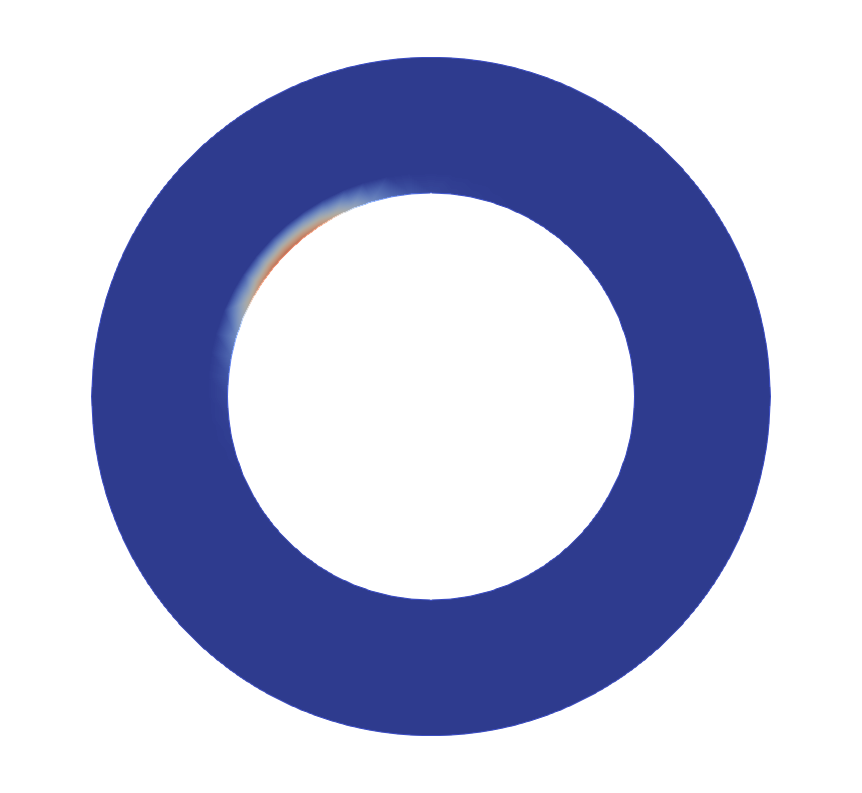}}
	\subfigure[$t=0.20$]{\label{fig:b4p}
		\includegraphics[width=0.225\textwidth]{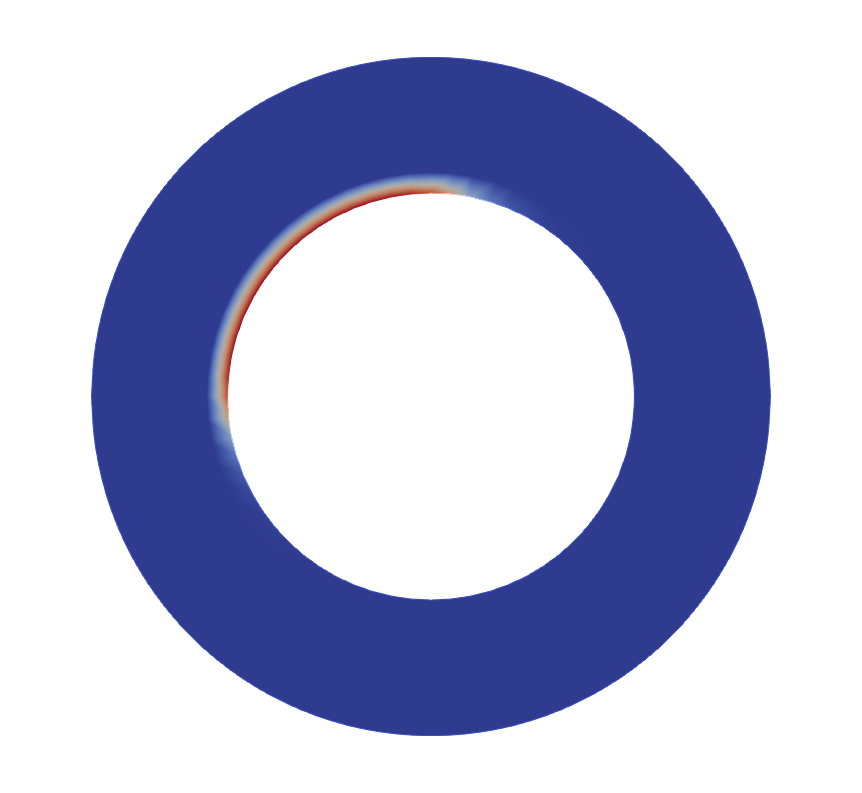}}
	\subfigure[$t=0.24$]{\label{fig:c4p}
		\includegraphics[width=0.225\textwidth]{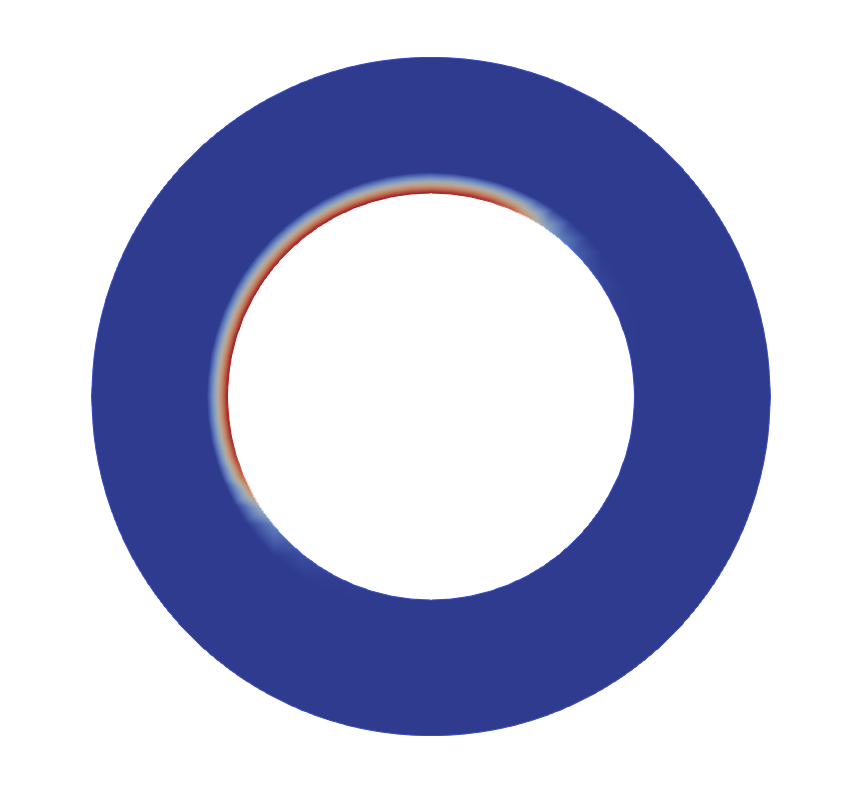}}
	\subfigure[$t=0.28$]{\label{fig:d4p}
		\includegraphics[width=0.225\textwidth]{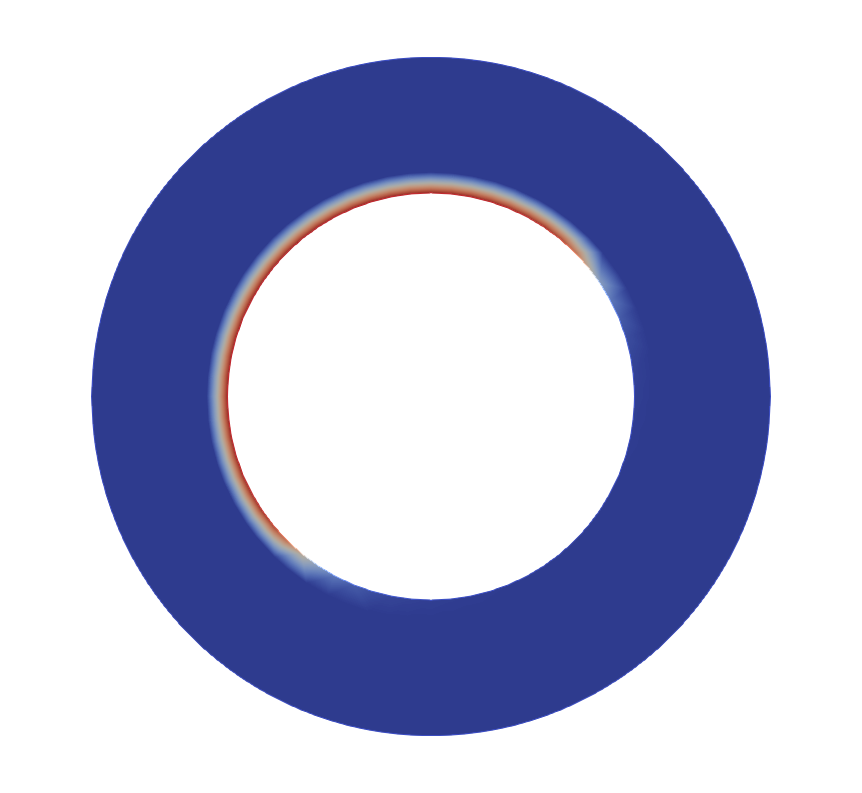}}
	\quad
	
	\subfigure[$t=0.32$]{\label{fig:e4p}
		\includegraphics[width=0.225\textwidth]{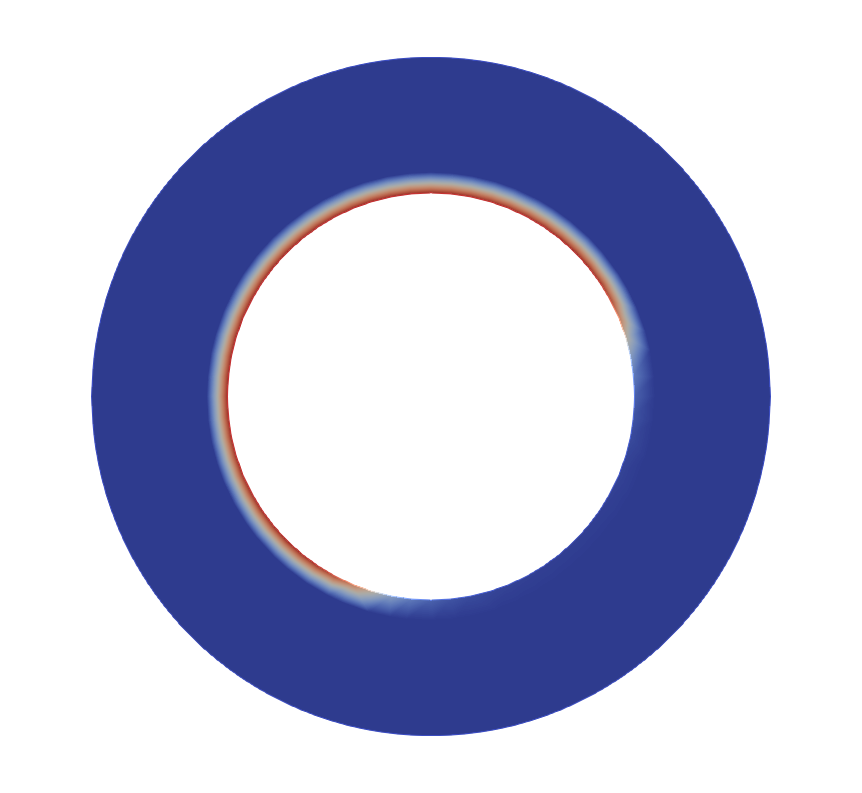}}
	\subfigure[$t=0.36$]{\label{fig:f4p}
		\includegraphics[width=0.225\textwidth]{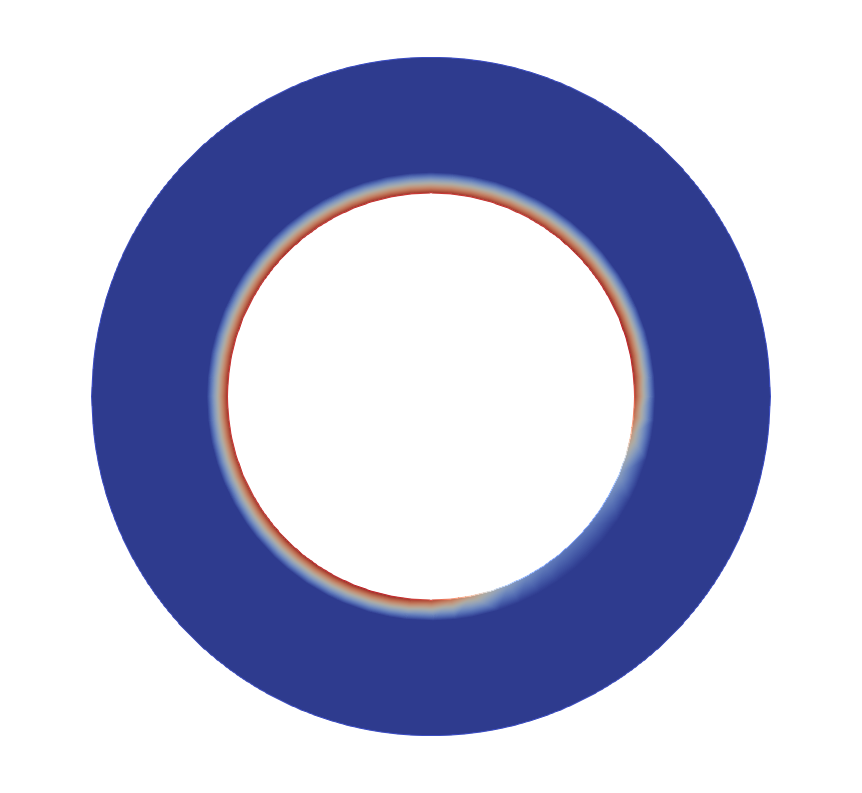}}
	\subfigure[$t=0.40$]{\label{fig:g4p}
		\includegraphics[width=0.225\textwidth]{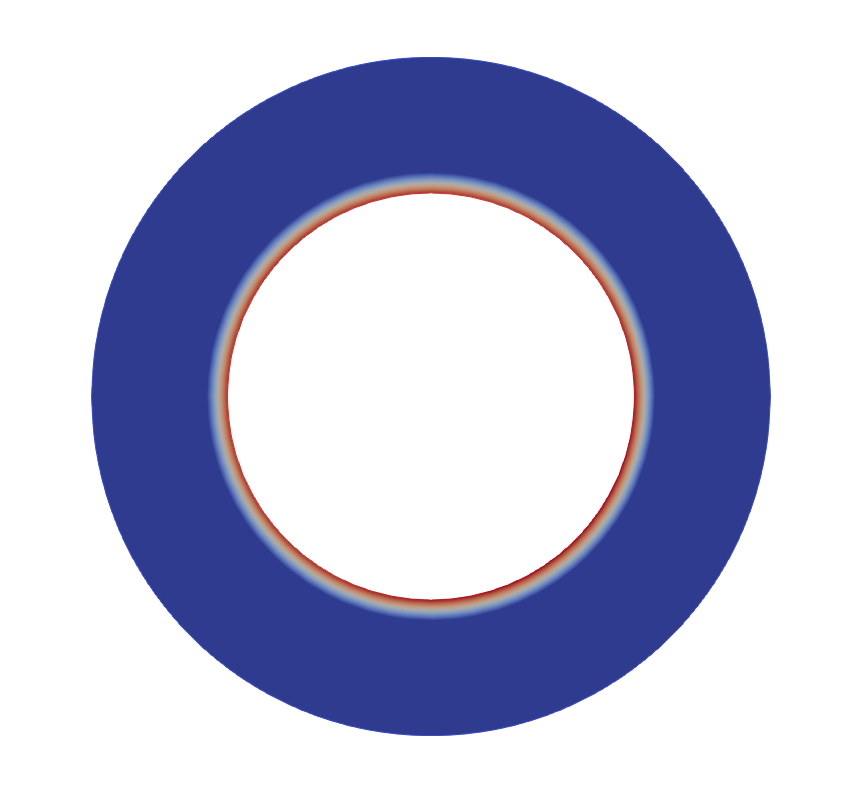}}
	\subfigure[$t=1.20$]{\label{fig:h4p}
		\includegraphics[width=0.225\textwidth]{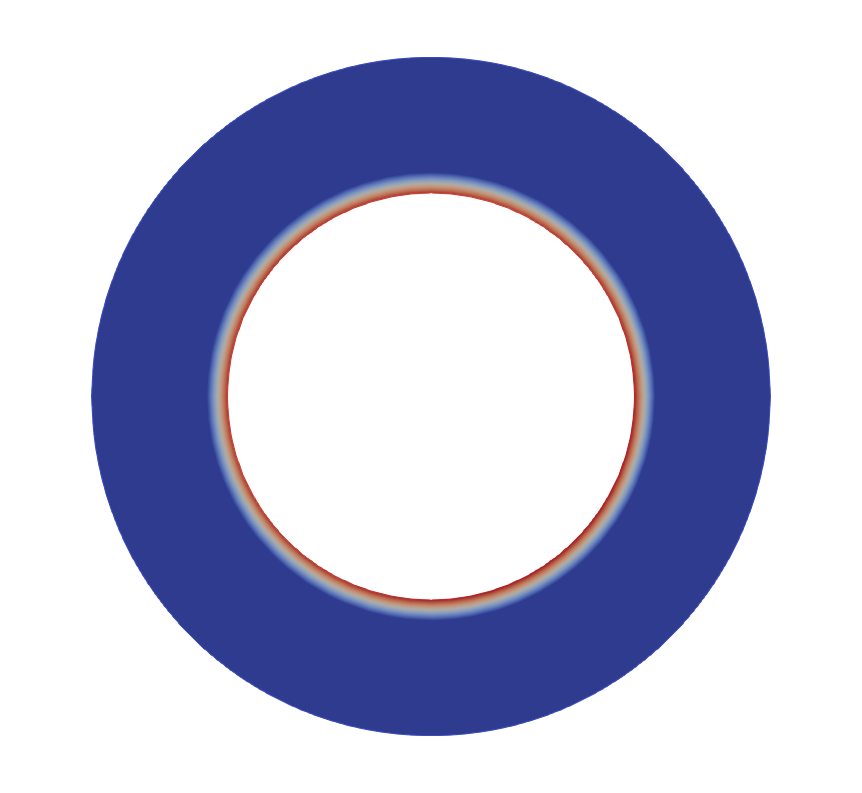}}
	\quad
	
	\subfigure[$t=1.8$]{\label{fig:i4p}
		\includegraphics[width=0.225\textwidth]{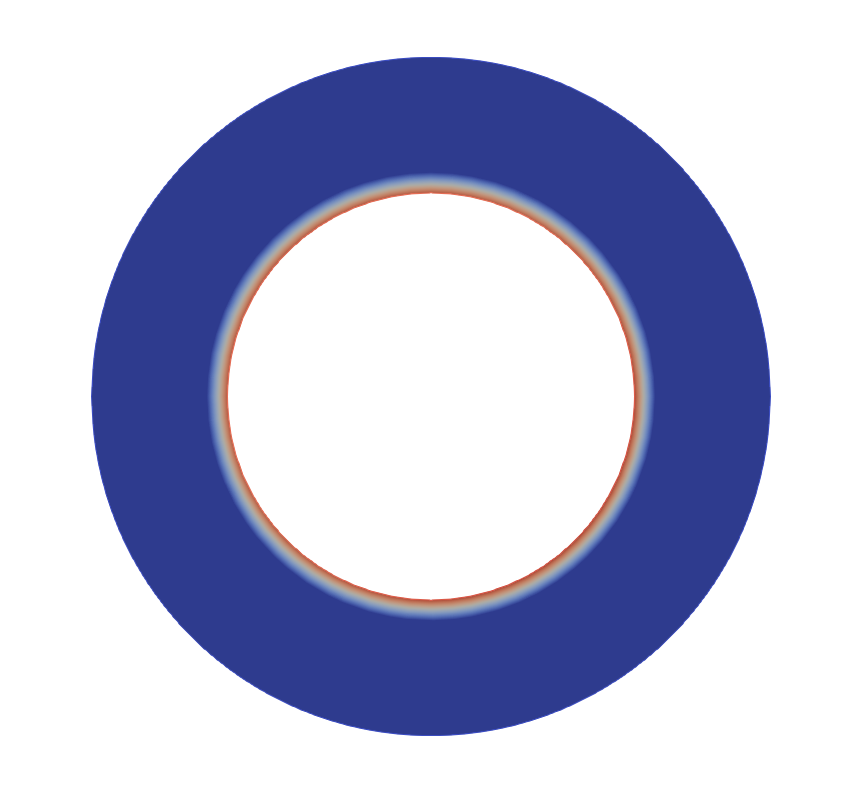}}
	\subfigure[$t=2.0$]{\label{fig:j4p}
		\includegraphics[width=0.225\textwidth]{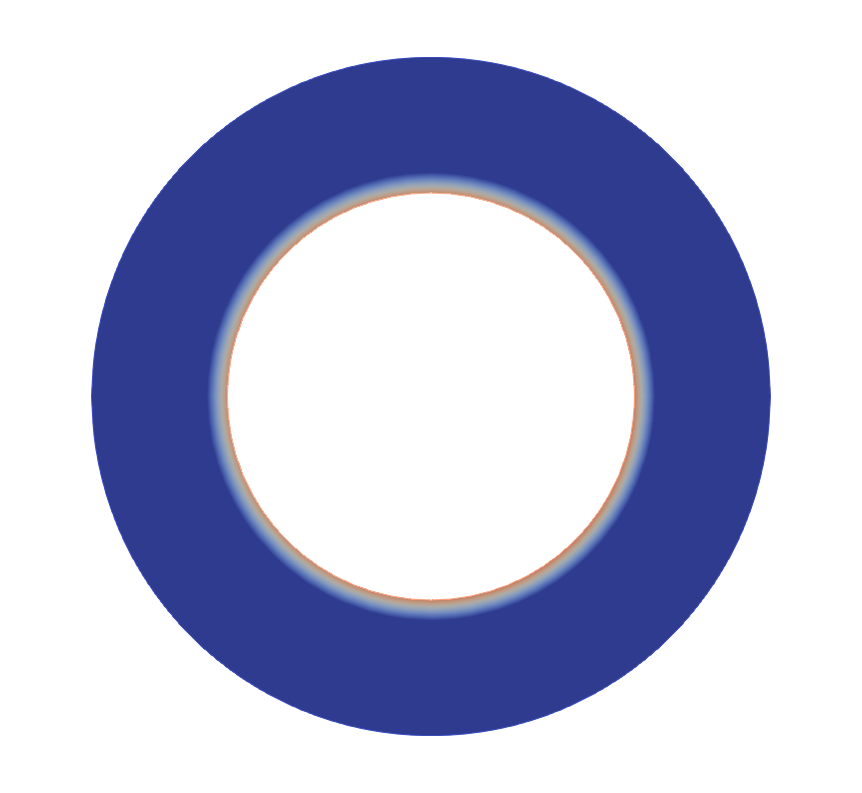}}
	\subfigure[$t=2.2$]{\label{fig:k4p}
		\includegraphics[width=0.225\textwidth]{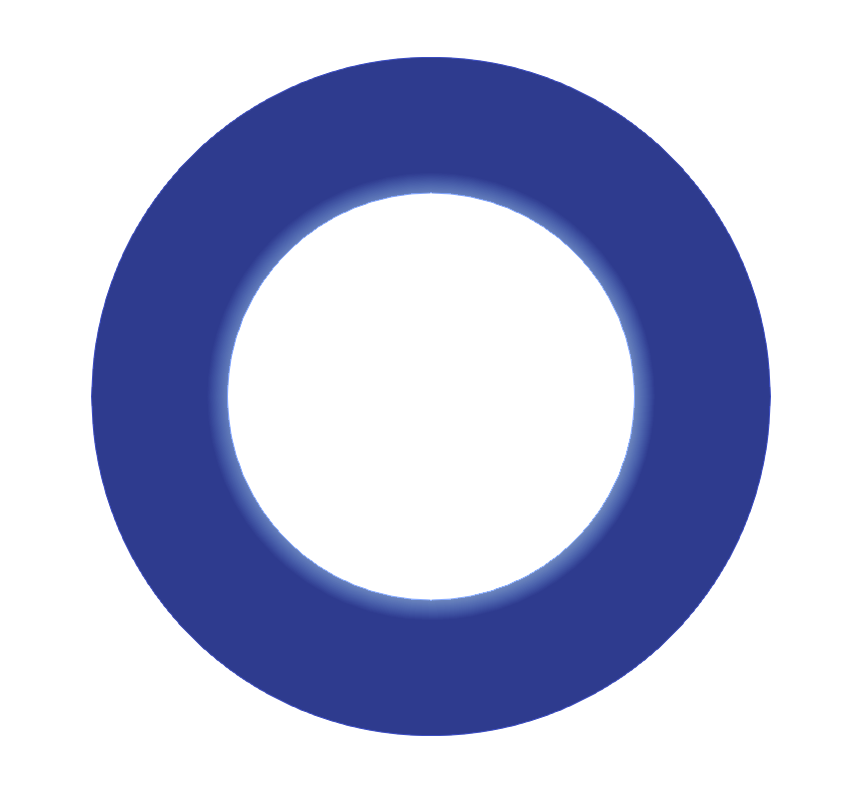}}
	\subfigure[$t=2.4$]{\label{fig:l4p}
		\includegraphics[width=0.225\textwidth]{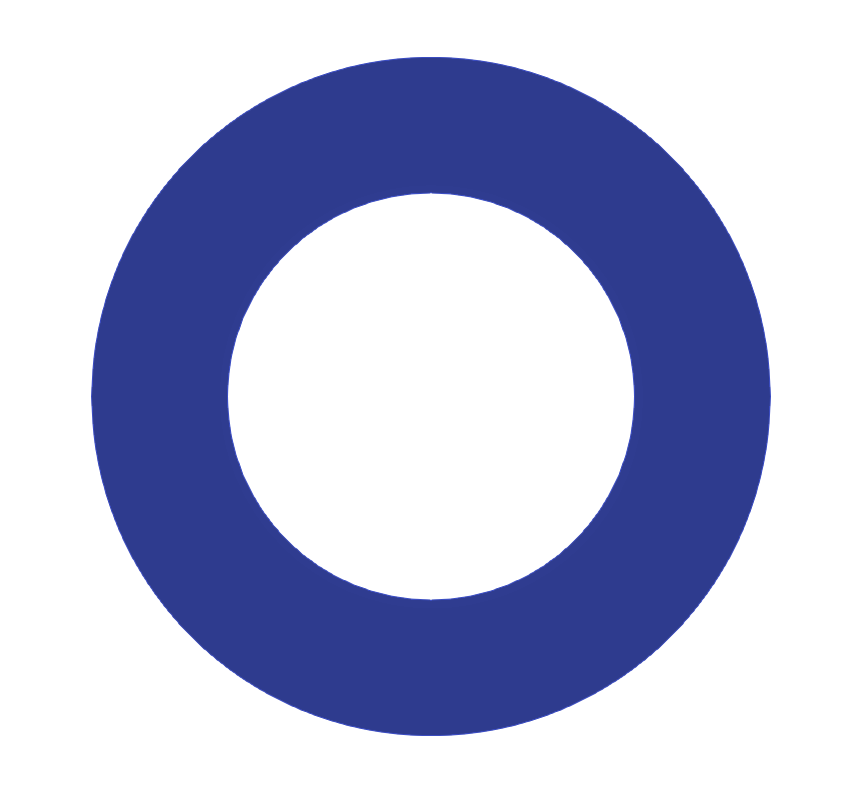}}
	\quad
	
	\subfigure[Color bar for the value of open probability on the inner 
	            circle -$\Upsilon$]{\label{fig:bar14p}
		\includegraphics[width=0.90\textwidth]{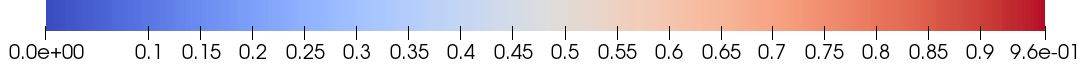}}
	\caption{The black region as in (l) is $\bar{\Omega}_c$, the value is always 0 on $\bar{\Omega}_c/\Upsilon$. Red color part on the inner circle (ER membrane) means the RyR channels are open. We can see that from (a)-(g), the open state propagates to the whole ER membrane which is faster compared with Example 3. From (h) to (l), the value of open probability decreases to the close state.}
	\label{fig:ca2+wave4p}
\end{figure}

\section{Conclusion}\label{con}
In this paper, we analyze the model of calcium dynamics in neurons with ER, obtain the existence, uniqueness and boundedness of the solution, and then propose an efficient implicit-explicit finite element scheme. The necessity of the ODE systems on interfaces is shown in Section \ref{Num}. The focus on calcium dynamics is motivated by the fact that intracellular calcium signals in response to electrical events (e.g. action potentials) trigger a multitude of calcium regulated processes which are relevant in cellular development, learning, and cell survival. The complexity of the cellular calcium-regulating machinery typically prohibits a systematic experimental study and computational models are highly relevant in studying the effect of morphological and biophysical changes on calcium dynamics.   
Models, theorems and algorithms are well established for electrical models, however, calcium dynamics has not been extensively studied, partly because lower-dimensional approximations are insufficient and detailed, high-resolution PDE-based simulations are required (integration of complex geometric structures of cells and intracellular organelles). Thus, 3D models are necessary to accurately capture calcium dynamics in cells and high-performance computing must be utilized. The $L^2$ error estimates, high order stable multi-step implicit-explicit schemes and a parallel implementation of the numerical methods for the 3D problem are part of ongoing work.

\bibliography{mybibfile}

\begin{thebibliography}{10}
\expandafter\ifx\csname url\endcsname\relax
  \def\url#1{\texttt{#1}}\fi
\expandafter\ifx\csname urlprefix\endcsname\relax\def\urlprefix{URL }\fi
\expandafter\ifx\csname href\endcsname\relax
  \def\href#1#2{#2} \def\path#1{#1}\fi

\bibitem{Amar2005}
M.~Amar, D.~Andreucci, P.~Bisegna, R.~Gianni, {Existence and uniqueness for an
  elliptic problem with evolution arising in electrodynamics}, Nonlinear
  Analysis: Real World Applications 6~(2) (2005) 367--380.
\newblock \href {http://dx.doi.org/10.1016/j.nonrwa.2004.09.002}
  {\path{doi:10.1016/j.nonrwa.2004.09.002}}.

\bibitem{Breit2018}
M.~Breit, G.~Queisser, {What Is Required for Neuronal Calcium Waves? A
  Numerical Parameter Study}, Journal of Mathematical Neuroscience 8~(1) (2018)
  1--22.
\newblock \href {http://dx.doi.org/10.1186/s13408-018-0064-x}
  {\path{doi:10.1186/s13408-018-0064-x}}.

\bibitem{MacKrill2012}
M.~Fill, J.~A. Copello, {Ryanodine receptor calcium release channels}, Physiol
  Rev 82 (2002) 893--922.
\newblock \href {http://dx.doi.org/10.1007/978-94-007-2888-2_7}
  {\path{doi:10.1007/978-94-007-2888-2_7}}.

\bibitem{HODGKIN1952}
A.~L. Hodgkin, A.~F. Huxley, A quantitative description of membrane current and
  its application to conduction and excitation in nerve, The Journal of
  physiology 117~(4) (1952) 500--544.

\bibitem{Luo1991}
C.~H. Luo, Y.~Rudy, \href{http://www.ncbi.nlm.nih.gov/pubmed/1709839}{{A Model
  of the Ventricular Cardiac Action Potential Depolarization, Repolarization,
  and Their Interaction}}, Circulation Research 68~(6) (1991) 1501--1526.
\newline\urlprefix\url{http://www.ncbi.nlm.nih.gov/pubmed/1709839}

\bibitem{Matano2011}
H.~Matano, Y.~Mori, {Global existence and uniqueness of a three-dimensional
  model of cellular electrophysiology}, Discrete and Continuous Dynamical
  Systems 29~(4) (2011) 1573--1636.
\newblock \href {http://dx.doi.org/10.3934/dcds.2011.29.1573}
  {\path{doi:10.3934/dcds.2011.29.1573}}.

\bibitem{Sneyd2003}
J.~Sneyd, K.~Tsaneva-Atanasova, J.~I. Bruce, S.~V. Straub, D.~R. Giovannucci,
  D.~I. Yule, {A model of calcium waves in pancreatic and parotid acinar
  cells}, Biophysical Journal 85~(3) (2003) 1392--1405.
\newblock \href {http://dx.doi.org/10.1016/S0006-3495(03)74572-X}
  {\path{doi:10.1016/S0006-3495(03)74572-X}}.

\bibitem{Veneroni2006}
M.~Veneroni, {Reaction-diffusion systems for the microscopic cellular model of
  the cardiac electric field}, Mathematical Methods in the Applied Sciences
  29~(14) (2006) 1631--1661.
\newblock \href {http://dx.doi.org/10.1002/mma.740}
  {\path{doi:10.1002/mma.740}}.

\bibitem{Clapham1995}
D.~E. Clapham, {Calcium signaling}, Cell 80~(2) (1995) 259--268.
\newblock \href {http://dx.doi.org/10.1016/0092-8674(95)90408-5}
  {\path{doi:10.1016/0092-8674(95)90408-5}}.

\bibitem{Ozturk2020}
Z.~{\"{O}}zt{\"{u}}rk, C.~J. O'Kane, J.~J. P{\'{e}}rez-Moreno, {Axonal
  Endoplasmic Reticulum Dynamics and Its Roles in Neurodegeneration} (jan
  2020).
\newblock \href {http://dx.doi.org/10.3389/fnins.2020.00048}
  {\path{doi:10.3389/fnins.2020.00048}}.

\bibitem{Raffaello2016}
A.~Raffaello, C.~Mammucari, G.~Gherardi, R.~Rizzuto, {Calcium at the Center of
  Cell Signaling: Interplay between Endoplasmic Reticulum, Mitochondria, and
  Lysosomes}, Trends in Biochemical Sciences 41~(12) (2016) 1035--1049.
\newblock \href {http://dx.doi.org/10.1016/j.tibs.2016.09.001}
  {\path{doi:10.1016/j.tibs.2016.09.001}}.

\bibitem{Rosales2004}
R.~A. Rosales, M.~Fill, A.~L. Escobar, {Calcium Regulation of Single Ryanodine
  Receptor Channel Gating Analyzed Using HMM/MCMC Statistical Methods}, Journal
  of General Physiology 123~(5) (2004) 533--553.
\newblock \href {http://dx.doi.org/10.1085/jgp.200308868}
  {\path{doi:10.1085/jgp.200308868}}.

\bibitem{Ali2019}
F.~Ali, A.~C. Kwan, {Interpreting in vivo calcium signals from neuronal cell
  bodies, axons, and dendrites: a review}, Neurophotonics 7~(01) (2019) 011402.
\newblock \href {http://dx.doi.org/10.1117/1.nph.7.1.011402}
  {\path{doi:10.1117/1.nph.7.1.011402}}.

\bibitem{Wu2017}
Y.~Wu, C.~Whiteus, C.~S. Xu, K.~J. Hayworth, R.~J. Weinberg, H.~F. Hess, P.~{De
  Camilli}, {Contacts between the endoplasmic reticulum and other membranes in
  neurons}, Proceedings of the National Academy of Sciences of the United
  States of America 114~(24) (2017) E4859--E4867.
\newblock \href {http://dx.doi.org/10.1073/pnas.1701078114}
  {\path{doi:10.1073/pnas.1701078114}}.

\bibitem{Atri1993}
A.~Atri, J.~Amundson, D.~Clapham, J.~Sneyd, {A single-pool model for
  intracellular calcium oscillations and waves in the Xenopus laevis oocyte},
  Biophysical Journal 65~(4) (1993) 1727--1739.
\newblock \href {http://dx.doi.org/10.1016/S0006-3495(93)81191-3}
  {\path{doi:10.1016/S0006-3495(93)81191-3}}.

\bibitem{Breit2018a}
M.~Breit, M.~Kessler, M.~Stepniewski, A.~Vlachos, G.~Queisser,
  {Spine-to-Dendrite Calcium Modeling Discloses Relevance for Precise
  Positioning of Ryanodine Receptor-Containing Spine Endoplasmic Reticulum},
  Scientific Reports 8~(1) (2018) 1--17.
\newblock \href {http://dx.doi.org/10.1038/s41598-018-33343-9}
  {\path{doi:10.1038/s41598-018-33343-9}}.

\bibitem{Keizer1996}
J.~Keizer, L.~Levine, {Ryanodine receptor adaptation and Ca$^{2+}$--induced
  Ca$^{2+}$ release-- dependent Ca$^{2+}$ oscillations}, Biophysical Journal
  71~(6) (1996) 3477--3487.
\newblock \href {http://dx.doi.org/10.1016/S0006-3495(96)79543-7}
  {\path{doi:10.1016/S0006-3495(96)79543-7}}.

\bibitem{Means2006}
S.~Means, A.~J. Smith, J.~Shepherd, J.~Shadid, J.~Fowler, R.~J. Wojcikiewicz,
  T.~Mazel, G.~D. Smith, B.~S. Wilson, {Reaction diffusion modeling of calcium
  dynamics with realistic ER geometry}, Biophysical Journal 91~(2) (2006)
  537--557.
\newblock \href {http://dx.doi.org/10.1529/biophysj.105.075036}
  {\path{doi:10.1529/biophysj.105.075036}}.

\bibitem{Calabro2006}
F.~Calabr{\`{o}}, P.~Zunino, {Analysis of parabolic problems on partitioned
  domains with nonlinear conditions at the interface. Application to mass
  transfer through semi-permeable membranes}, Mathematical Models and Methods
  in Applied Sciences 16~(4) (2006) 479--501.
\newblock \href {http://dx.doi.org/10.1142/S0218202506001236}
  {\path{doi:10.1142/S0218202506001236}}.

\bibitem{Cangiani2013}
A.~Cangiani, E.~H. Georgoulis, M.~Jensen, {Discontinuous Galerkin Methods for
  Mass Transfer through Semipermeable Membranes}, SIAM Journal on Numerical
  Analysis 51~(5) (2013) 2911--2934.
\newblock \href {http://dx.doi.org/10.1137/120890429}
  {\path{doi:10.1137/120890429}}.

\bibitem{Cangiani2016}
A.~Cangiani, E.~H. Georgoulis, M.~Jensen, {Discontinuous Galerkin methods for
  fast reactive mass transfer through semi-permeable membranes}, Applied
  Numerical Mathematics 104 (2016) 3--14.
\newblock \href {http://dx.doi.org/10.1016/j.apnum.2014.06.007}
  {\path{doi:10.1016/j.apnum.2014.06.007}}.

\bibitem{Henriquez2017}
F.~Henr{\'{i}}quez, C.~Jerez-Hanckes, F.~Altermatt, {Boundary integral
  formulation and semi-implicit scheme coupling for modeling cells under
  electrical stimulation}, Numerische Mathematik 136~(1) (2017) 101--145.
\newblock \href {http://dx.doi.org/10.1007/s00211-016-0835-9}
  {\path{doi:10.1007/s00211-016-0835-9}}.

\bibitem{connors2009}
J.~M. Connors, J.~S. Howell, W.~J. Layton, Partitioned time stepping for a
  parabolic two domain problem, SIAM Journal on Numerical Analysis 47~(5)
  (2009) 3526--3549.

\bibitem{Gil2021}
D.~Gil, A.~H. Guse, G.~Dupont, A.~Tepikin, G.~Ullah, {Three-Dimensional Model
  of Sub-Plasmalemmal Ca$^{2+}$ Microdomains Evoked by the Interplay Between
  ORAI1 and InsP3 Receptors}, Front. Immunol. 12 (2021) 1.
\newblock \href {http://dx.doi.org/10.3389/fimmu.2021.659790}
  {\path{doi:10.3389/fimmu.2021.659790}}.

\bibitem{Douglas1973}
J.~Douglas, T.~Dupont, {Galerkin methods for parabolic equations with nonlinear
  boundary conditions}, Numerische Mathematik 20~(3) (1973) 213--237.
\newblock \href {http://dx.doi.org/10.1007/BF01436565}
  {\path{doi:10.1007/BF01436565}}.

\bibitem{Fri1967}
A.~Friedman, {Partial Differential Equations of Parabolic Type}, Courier Dover
  Publications, 2008.

\bibitem{Pao1992}
C.~V. Pao, {Nonlinear Parabolic and Elliptic Equations}, Plenum Press, New
  York, 1992.

\bibitem{calabro2013}
F.~Calabr{\`o}, Numerical treatment of elliptic problems nonlinearly coupled
  through the interface, Journal of Scientific Computing 57~(2) (2013)
  300--312.

\bibitem{09iso}
J.~A. Cottrell, T.~J. Hughes, Y.~Bazilevs, Isogeometric analysis: toward
  integration of CAD and FEA, John Wiley \& Sons, 2009.

\bibitem{iso1986}
M.~Lenoir, Optimal isoparametric finite elements and error estimates for
  domains involving curved boundaries, SIAM Journal on Numerical Analysis
  23~(3) (1986) 562--580.

\bibitem{Hecht12}
F.~Hecht, \href{https://doi.org/10.1515/jnum-2012-0013}{New development in
  freefem++}, Journal of Numerical Mathematics 20~(3-4) (2012) 251--266.
\newblock \href {http://dx.doi.org/doi:10.1515/jnum-2012-0013}
  {\path{doi:doi:10.1515/jnum-2012-0013}}.
\newline\urlprefix\url{https://doi.org/10.1515/jnum-2012-0013}

\end{thebibliography}

\end{document}